\documentclass[12pt,reqno]{amsart}
\usepackage[margin=1in]{geometry}
\usepackage[utf8]{inputenc}
\usepackage[T1]{fontenc}
\usepackage{mathtools}

\usepackage{amsmath,amsxtra,amssymb,amsthm,amsfonts,eufrak,bm}
\usepackage[usenames,dvipsnames]{xcolor}
\usepackage{tikz-cd}
\definecolor{DarkRed}{rgb}{0.9,0.7,0.3}
\usetikzlibrary{matrix,arrows,decorations.pathmorphing}

\newcommand{\bc}[1]{\lan #1 \ran}
\newcommand{\p}{\partial}
\newcommand{\bx}{{\bf x}}
\newcommand{\by}{{\bf y}}

\newcommand{\bv}{{\bf v}}
\newcommand{\be}{{\bf e}}
\newcommand{\bw}{{\bf w}}
\newcommand{\bh}{{\bf h}}

\newcommand{\bet}{\bm{\eta}}
\newcommand{\bxi}{\bm{\xi}}
\newcommand{\e}{\epsilon}

\newcommand{\R}{{\mathbb R}}
\newcommand{\Z}{{\mathbb Z}}
\newcommand{\C}{{\mathbb C}}
\newcommand{\ten}{\otimes}

\newcommand{\pl}{\hspace{.1cm}}

\newcommand{\ran}{\rangle}

\newcommand{\lan}{\langle}
\newcommand{\al}{\alpha}
\newcommand{\si}{\sigma}

\newcommand{\la}{\lambda}

\newcommand{\E}{{\mathbb E}}

\newcommand{\A}{{\mathcal A}}

\newcommand{\ddd}{{\Delta}}

\newcommand{\M}{{\mathcal M}}
\numberwithin{equation}{section}
\newcommand{\K}{{\mathcal K}}
\renewcommand{\S}{{\mathcal S}}

\newcommand{\N}{{\mathcal N}}

\newcommand{\norm}[2]{\parallel \! #1 \! \parallel_{#2}}

\newtheorem{lemma}{Lemma}[section]
\newtheorem{prop}[lemma]{Proposition}
\newtheorem{theorem}[lemma]{Theorem}
\newtheorem{defi}[lemma]{Definition}
\newtheorem{cor}[lemma]{Corollary}

\newtheorem{rem}[lemma]{Remark}

\newcommand{\re}{\begin{rem}\rm}
\newcommand{\mar}{\end{rem}}

\newcommand{\ket}[1]{|{#1}\rangle}

\newcommand{\qd}{\end{proof}\vspace{0.5ex}}
\newcommand{\prf}{\begin{proof}[\bf Proof:]}
\newcommand{\xspace}{\hbox{\kern-2.5pt}}

\usepackage{hyperref}

\newcommand{\hl}{\color{black}}

\begin{document}
\title{Quantum Euclidean Spaces with noncommutative derivatives}
\author{Li Gao}
\address{Department of Mathematics\\
Texas A\&M University, College Station, TX 77843, USA} \email[Li Gao]{ligao@math.tamu.edu}

\author[Marius Junge]{Marius Junge$^*$}\thanks{$^*$ Partially supported by NSF grant DMS-1501103  and  DMS-1800872}
\address{Department of Mathematics\\
University of Illinois, Urbana, IL 61801, USA} \email[Marius
Junge]{mjunge@illinois.edu}

\author[Edward McDonald]{Edward McDonald}
\address{School of Mathematics and Statistics\\
University of New South Wales, UNSW Sydney
NSW, 2052, Australia} \email[Edward McDonald]{edward.mcdonald@unsw.edu.au}
\begin{abstract}Quantum Euclidean spaces, as Moyal deformations of Euclidean spaces, are the model examples of noncompact noncommutative manifold. In this paper, we study the quantum Euclidean space equipped with partial derivatives satisfying canonical commutation relation (CCR). This gives an example of semi-finite spectral triple with non-flat geometric structure. We develop an abstract symbol calculus for the pseudo-differential operators with noncommuting derivatives. We also obtain a simplified local index formula (even case) that is similar to the commutative setting.
\end{abstract}
\maketitle

\section{Introduction}
The theory of pseudo-differential operators ($\Psi$DOs) plays an influential role in {\hl the} index theory of elliptic operators. This approach also prevails in noncommutative geometry. In \cite{CM95}, Connes and Moscovici established the local index formula for spectral triples, which gives an analytic expression for the index pairing between $K$-theory of noncommutative algebras and the $K$-homology class induced by a Dirac type operator. This local index formula was extended to the locally compact {\hl (i.e., non-unital)} setting by Carey, Gayral, Rennie and Sukochev \cite{CGRSmemo}. In both {\hl proofs of the local index formula} \cite{CM95,CGRSmemo}, an abstract theory of $\Psi$DOs {\hl is crucial to the analysis}. On {\hl the prototypical example of a noncommutative geometry}--quantum tori, pseudo-differential operators been widely used in studying curvatures and other geometric structures (see e.g. \cite{CT2011,fathizadeh2013weyl,lesch2016modular,bhuyain2012ricci,CM14}).
Recently several works \cite{Jim,Ponge18,Ponge18s, GJP17} give detailed accounts of {\hl the symbol calculus for} $\Psi$DOs on quantum tori.

Quantum Euclidean spaces are {\hl model examples of noncommutative spaces in the locally compact setting}, and can be viewed as locally compact counterparts of quantum tori. They are noncommutative deformation of Euclidean spaces which originate from the Heisenberg relation and Moyal products in quantum mechanics. Let $\theta=(\theta_{jk})_{j,k=1}^{d}$ be a skew-symmetric $d\times d$ matrix. Roughly speaking, a $d$-dimensional quantum Euclidean space is given by the von Neumann algebra $\R_\theta$ generated by the spectral projections of $d$ self-adjoint operators $x_1,\cdots, x_d$ satisfying the the canonical commutation relation (CCR) \[[x_j,x_k]=-i\theta_{jk}\pl.\] {\hl We will review a rigorous definition of $\R_\theta$ in Section \ref{sectionpre}.}
{\hl Despite having a relatively simple algebraic structure (a type ${\rm I}$ von Neumann algebra)} the connection to Euclidean spaces and quantum physics make them indispensable in various
scenarios. For example, from {\hl the perspective of harmonic and functional analysis}, Calder\'on-Zygmund theory {\hl and pseudodifferential operator theory} on quantum Euclidean spaces was established in {\hl the recent article} \cite{GJP17} and the {\hl theory of distributions} goes back to \cite{moyal1,moyal2}. In noncommutative geometry, quantum Euclidean spaces serve as model examples for non-unital spectral triples \cite{moyalplane}. {\hl In mathematical physics, noncommutative Euclidean spaces have been heavily studied under the name of canonical commutation relation (CCR) algebras \cite[Section 5.2.2.2]{Bratteli-Robinson2} and in the context of Weyl quantization \cite[Chapter 14]{Hall}, \cite[Chapter 2, Section 3]{Takhtajan2008}.}
Also, the discovery of instantons on noncommutative $\R^4$ makes an influential connection to string theory \cite{connes01,nekrasov98,seiberg99}.

In this paper, we revisit the connection between $\Psi$DOs and {\hl the }local index formula for quantum Euclidean spaces. Both topics have been considered for $\R_\theta$, with its standard geometric structure. Recall that $\R_\theta$ is associated with a Weyl quantization map, defined for functions in the Schwartz class $S(\R^d)$ as:
\[\la_\theta: f\in S(\R^d)\mapsto \frac{1}{2\pi^d}\int_{\R^d}\hat{f}(\xi)\la_\theta(\xi)d\xi\in \R_\theta\pl.\]
where $\la_\theta(\bxi)=e^{\bxi_1x_1+\cdots+\bxi_dx_d},\bxi \in \R^d$ is a projective unitary representation of $\R^d$, \[\la_\theta(\bxi)\la_\theta(\bet)=e^{i\frac{\theta}{2}\bxi\bet}\la_\theta(\bxi+\bet)\] (see Section \ref{sectionpre} for further details).
The canonical trace {\hl associated to} $\R_\theta$ is {\hl defined on the image of $S(\R^d)$ under $\la_\theta$ as} $\displaystyle \tau_\theta(\la_\theta(f))=\int f$. {\hl Differentiation operators $\frac{\partial}{\partial \bx_j}$ admit a canonical extension to} $\R_\theta$, defined on $\la_\theta(S(\R^d))$ by $D_j\la_\theta(f)=\la_\theta(-i\frac{\partial}{\partial \bx_{j}}f)$.
{\hl The operators $D_j$ have self-adjoint extensions to the Hilbert-Schmidt space} $L_2(\R_\theta, \tau_\theta)$. {\hl Since partial differentiation operators on $S(\R^d)$ commute, it follows immediately that $[D_j,D_k] = 0$ for $1\leq j,k\leq d$. The fact that these partial derivatives mutually commute reflects a ``flat" geometry of $\R_\theta$.}

The scope of this paper is to consider a more general but still computable differential structure on $\R_\theta$. More precisely, we shall equip $\R_\theta$ with ``covariant derivatives" $\xi_1,\cdots,\xi_d$ satisfying (another) CCR relation. Unlike the standard case
\begin{align} [x_j,x_k]=-i\theta_{j,k}, [D_j,x_k]=-i\delta_{j,k} \pl, [D_j,D_k]=0\pl,\label{st}\end{align}
we consider that $x_j$'s and $\xi_k$'s together have the commutation relations
\begin{align}[x_j,x_k]=-i\theta_{j,k},  [\xi_j,x_k]=-i\delta_{jk}  \pl, \pl [\xi_j,\xi_k]=-i\theta_{jk}'\pl.\label{ncnc}\end{align}
where $\delta$ is the Kronecker Delta notation and $\theta'$ is an arbitrary but fixed skew-symmetric matrix. In the classical case when $\theta=0$ and $\R_0=L_\infty(\R^d)$, such $\xi_j$'s are covariant derivatives of connections with a constant curvature form (see Section 3.1). From this perspective, \eqref{ncnc} can be viewed as a natural deformation of \eqref{st} by adding a nonzero curvature form. {\hl From the perspective
of quantum physics, noncommuting derivatives occur in the presence of a magnetic field \cite{AHS1978}. One can view the matrix $\theta'$ as representing a constant magnetic field on $\R_\theta$.}
The noncommutativity of {\hl the} covariant derivatives $\xi_j$ adds essential difficulty in developing the {\hl theory of $\Psi$DOs}. When $\theta'=0$, the commutativity of $D_j$'s makes the phase space (or the Fourier transform side) a commutative space, and then the symbol of a $\Psi$DO is a operator-valued function $a:\R^d \to \R_\theta$. In our setting for noncommuting $\xi_j$'s, the symbol will become purely abstract as operators affiliated to $\R_\theta\ten \R_\theta'$. Moreover, due to the unbounded natural of symbol functions, we have to inevitably deal with unbounded but smooth elements. {\hl The idea of incorporating noncommuting derivatives into pseudodifferential calculus has also appeared in the related context of magnetic pseudodifferential
calculus \cite{MP2004,MPR2005}.}

We now briefly explain our setting and illustrate the main results.
Let $\R_\theta\overline{\ten}\R_\theta'$ be the $2d$-dimensional quantum Euclidean space generated by {\hl the relations}
\[[x_j,x_k]=-i\theta_{j,k}\pl, [\xi_j,\xi_k]=-i\theta'_{j,k}\pl, [x_j,\xi_k]=0\pl\]
and let $\R_\Theta$ be the {\hl$2d$-dimensional} space generated by \eqref{ncnc} with parameter matrix $\Theta=\left[\begin{array}{cc}\theta&I_d\\ -I_d&\theta'
\end{array}\right]$.
{\hl We will consider pseudodifferential calclulus defined with symbols as operators} affiliated to $\R_\theta\overline{\ten}\R_\theta'$ and the $\Psi$DOs {\hl themselves }are operators affiliated to $\R_\Theta$. The operator {\hl or quantization} map ``$Op$'' sending symbols to $\Psi$DOs is simple: for $a\in\R_\theta, b\in \R_{\theta'}$
\begin{align}\label{o}Op(a\ten b)=ab\in \R_\Theta\pl,\end{align}
where $\R_\theta,\R_\theta'$ are viewed as subalgebras of $\R_\Theta$. The domain of {\hl $Op$} can extended to the following abstract symbol class.
\begin{itemize}
\item We say an operator $a$ affiliated to $\R_\theta\overline{\ten} \R_{\theta'}$ is a symbol of order $m$ (write as $a\in \Sigma^m$) if for any multi-indices $\al$ and $\beta$, $D_x^\al D_\xi^\beta(a)(1+\sum_{j}\xi_j^2)^{-\frac{m+|\beta|}{2}}$ extends to a bounded operator in $\R_\theta\overline{\ten} \R_{\theta'}$.
\end{itemize}
Here $D_x$ are {\hl the canonical (commuting) differentiation operators acting on} the first component $\R_\theta$ and $D_\xi$ {\hl are the same} for $\R_{\theta'}$. {\emph A priori} it is not clear that this definition is closed under {\hl multiplication},
and {\hl adjoint, or if we have the expected properties }$\Sigma^m\cdot \Sigma^n=\Sigma^{m+n}$ and $(\Sigma^m)^*=\Sigma^m$, which {\hl are important components for the development of a } symbol calculus. To resolve that, we introduce in Section \ref{sectionasy} a notation of ``asymptotic degree'' to measure the unboundedness of operators affiliated to $\R_\theta$. {\hl This is a notion directly inspired by the abstract pseudodifferential calculus developed
by Connes and Moscovici \cite[Appendix B]{CM95} and Higson \cite{higson}.} With this definition of symbol class,
 we establish in Section \ref{phido} the two core parts of $\Psi$DOs calculus---the $L_2$-boundedness theorem for $0$-order $\Psi$DOs and the composition formula.
\begin{theorem}[c.f. Theorem \ref{l2}] Let $a$ be a symbol of order $0$ {\hl(i.e., $a \in \Sigma^0$)}. Then $Op(a)$, {\hl initially defined on $\la_{\Theta}(S(\R^{2d}))$ has unique extension to} a bounded operator {\hl on the Hilbert space $L_2(\R_\Theta)$.}
\end{theorem}
\begin{theorem}[c.f. Theorem \ref{compo}] Let $a$ be a symbol of order $m$ and $b$ be a symbol of order $n$. Then $Op(a)Op(b)=Op(c)$ for some symbol $c$ of order $m+n$. Moreover \[c\sim \sum_{\al} \frac{ i^{-|\al|}}{\al !}D_\xi^\al (a)D_x^\al(b)\]
in the sense that for any positive integer $N$, $c-\sum_{|\al|\le N} \frac{ i^{-|\al|}}{\al !}D_\xi^\al (a)D_x^\al(b)$ is a symbol of order $m+n-N-1$.
\end{theorem}
The proofs of the above theorems use the idea of co-multiplication maps. The co-multiplication maps enables us to convert {\hl the }operator map $Op$ as an operator-valued classical operator map on the $\R^d$. In particular, this gives an alternative approach to some parts of symbol calculus in \cite{GJP17} for $\theta'=0$.

In Section \ref{sectionloc}, we apply the $\Psi$DO calculus {\hl prove that}
\begin{align}\label{spectraltriple}(W^{\infty,1}(\R_\theta), L_2(\R_\Theta)\ten \mathbb{C}^N, D=\sum_{j}\xi_j\ten c_j) \pl,\end{align}
{\hl forms a semifinite non-unital spectral triple (in the sense of \cite[Definition 2.1]{CGRSmemo}).}
Here, $c_j$ are generators of the Clifford algebra $Cl^d$ and $W^{\infty,1}(\R_\theta)=\{a | D^\al(a)\in L_1(\R_\theta)\pl \forall \pl \al\}$ is the noncommutative Sobolev spaces. We denote $W^{\infty,1}(\R_\theta)^\sim=W^{\infty,1}(\R_\theta)+\mathbb{C}$ {\hl for the minimal} unitalization. {\hl The triple \eqref{spectraltriple} forms} a smoothly summable semifinite spectral triple with isolated spectrum dimension (see Section \ref{sectionloc} for further details).
{\hl We are able to apply the even case of the local index formula \cite[Theorem 3.33]{CGRSmemo}, yielding the following:}
\begin{theorem}[c.f. Corollary \ref{formula}] \label{C}Let $d$ be even and $\R_\theta$ be a $d$-dimensional quantum Euclidean space. Then $(A,H,D):=(W^{\infty,1}(\R_\theta), L_2(\R_\Theta)\ten M_N, \sum_{j}\xi_j\ten c_j)$ is an even, smoothly summable, semi-finite spectral triple with isolated spectrum dimension. Moreover, for a projection $e\in M_n(W^{\infty,1}(\R_\theta)^\sim)$, the index pairing is given by
\begin{align*}&\lan [e]-[1_e], (A,H,D)\ran =\pi^\frac{d}{2}(\tau_\theta\ten tr(\gamma (e-1_e) \frac{\omega^\frac{d}{2}}{\frac{d}{2}!})+\sum_{m=1}^{\frac{d}{2}}\frac{1}{2m!}\tau_\theta\ten tr(\gamma e(de)^{2m}
\frac{\omega^{\frac{d}{2}-m}}{(\frac{d}{2}-m)!}))\pl,\end{align*}
where $\omega=\frac{i}{2}\sum_{j,k}\theta_{j,k}c_jc_k$.
\end{theorem}
Note that the Dirac Laplacian {\hl has square given by}
\[D^2=(\sum_{j}\xi_j\ten c_j)^2=\sum_{j}\xi_j^2-\omega\pl.\]
Where $\omega$ plays the role of a curvature form in the index pairing.
The general local index formula in \cite{CM95,CGRSmemo} contains residue cocycles which involve higher order residues at $z=0$ for zeta functions
\[\zeta_{k}(z)=tr(\gamma a_0d a_1^{(k_1)}\cdots d a_m^{(k_m)}(1+D^2)^{-\frac{m}{2}-k-z})\pl\]
where $a_j\in A$, $da=[D,a]$ and $\displaystyle da^{(k)}:=\underbrace{[D^2,[D^2,\cdots [D^2}_{\text{$k$-times}},da]]$.
Theorem \ref{C} basically observes that the above zeta functions has nonzero residue only for $|k|=0$ and the poles are simple. For a Dirac operator on compact spin Riemannian manifolds, such a simplification was observed in \cite{CM95} and fully developed by Ponge \cite{Ponge} using Getzler calculus. {\hl The local index formula of Connes and Moscovici \cite{CM95}} recovers the Atiyah-Singer index theorem for spin Dirac operators. Theorem \ref{C} shows that a similar simplified index formula holds for the noncommutative {\hl spectral triple}
$(W^{\infty,1}(\R_\theta), L_2(\R_\Theta)\ten M_N, \sum_{j}\xi_j\ten c_j)$. We also provide a concrete example of the index pairing in $d=2$ (Theorem \ref{example}).

The paper is organized as follows: We first reviews some preliminary facts about quantum Euclidean spaces in Section 2. Section 3 introduces and discuss the notation ``asymptotic degree'', which is a key tool in the subsequent discussions. In Section 4, we discuss the symbol calculus of $\Psi$DOs and prove Theorem 1.1 and 1.2. Section 5 is devoted to the local index formula and Theorem 1.4.

{\bf Acknowledgement}-The authors are grateful to Alexander Gorokhovsky for helpful discussion on the local index formula.

\section{Preliminaries on Quantum Euclidean spaces}\label{sectionpre}
In this section we review the basic structures of Quantum Euclidean spaces. Quantum Euclidean spaces in the literature has been studied under several different names: Moyal plane \cite{moyalplane,moyal1,moyal2}, canonical commutatation
relation (CCR) algebras \cite[Section 5.2.2.2]{Bratteli-Robinson}, noncommutative Euclidean Spaces \cite{gao18,dao2} and quantum Euclidean spaces \cite{GJP17}. In particular, \cite{Bratteli-Robinson} gives a detail account from the operator theoretic perspective. The distribution theory was studied in \cite{moyal1,moyal2}. More recently \cite{GJP17} studies harmonic analysis on quantum Euclidean spaces. From the noncommutative geometric perspective, an early exposition is in \cite{moyalplane}.

\subsection{Definitions and notations}
Throughout the paper we use the usual letters $x_1,x_2,\cdots,$ and $\xi_1,\xi_2,\cdots$ for operators and the boldface letters $\bx=(\bx_1,\bx_2,\cdots,\bx_d), \bxi=(\bxi_1,\bxi_2,\cdots,\bxi_d)$ for vectors and scalars. Let $d\ge 2$ and ${\theta=(\theta_{jk})_{j,k=1}^d}$ be a real skew-symmetric $d \times d$ matrix. Let $\S(\R^d)$ the space of complex Schwartz
functions (smooth, rapidly decreasing) on $\R^d$. The Moyal product $\star_\theta$ associated to $\theta$ is defined as (see \cite{Rieffel93}),
\begin{align*}
f\star_\theta g (\bx):=(2\pi)^{-d}\int_{\R^d}\int_{\R^d}f(\bx+\frac{{\bf \theta}}{2} \bv)g(\bx-\bw)e^{i\bv\cdot \bw}d\bv d\bw \pl, \pl f,g\in \S(\R^d)
\end{align*}
The Moyal product is bilinear, associative and reversed under complex conjugation $\overline{f}\star_\theta \overline{g}= \overline{g\star_\theta f}$, which makes  $(\S(\R^d),\star_\theta)$ a $*$-algebra.
The left Moyal multiplication gives the following $^*$-homomorphism $\la_\theta: (\S(\R^d),\star_\theta) \to B(L_2(\R^d))$,
\begin{align}\label{left}\la_\theta(f) g= f\star_\theta g, \la_\theta(f)\la_\theta(g)=\la_\theta(f\star_\theta g) \pl.\end{align}
\begin{defi}The quantum Euclidean space associated to $\theta$ is given by the following objects in $B(L_2(\R^d))$,
\begin{enumerate}
\item[i)]$\S_\theta:=\la_\theta(\S(\R^d))$ as the quantized Schwartz class ;
\item[ii)]$\E_\theta:= \overline{\S_\theta^{||\cdot ||}}$ as the $C^*$-algebra generated by $\S_\theta$;
\item[iii)] $\R_\theta:=(\S_\theta)''$ as the von Neumann algebra generated by $\S_\theta$.
\end{enumerate}
\end{defi}
\noindent When $\theta=0$, $\star_0$ is the usual point-wise multiplication,  $\E_0=C_0(\R^d)$ is the space of continuous functions on $\R^d$ which vanish at infinity and $\R_0=L_\infty(\R^d)$ is the space of essentially bounded functions on $\R^d$. An equivalent approach is the $\theta$-twisted regular representation of the group $\R^d$. For each vector $\bxi\in \R^d$, we define the unitary operator $\la_\theta(\bxi)$ on $L_2(\R^d)$,
\begin{align}
 (\la_\theta(\bxi) g)(\bx)=e^{i\bxi\cdot\bx}g(\bx-\frac{\theta}{2}\bxi)\pl\label{rep}
\end{align}
They satisfies the commutation relation
\[ \la_\theta(\bxi)\la_\theta(\bet)=e^{\frac{i}{2}\bxi \cdot \theta \bet}\la_\theta(\bxi+\bet)=e^{{i}\bxi \cdot \theta \bet}\la_\theta(\bet)\la_\theta(\bxi)\pl.\]
The map $\la_\theta:\R^d\to B(L_2(\R^d)$ is a projective unitary representation of $\R^d$ called the twisted left regular representation. The Moyal multiplication \eqref{left} for $(S(\R^d),\star_\theta)$ is equivalent to the corresponding Weyl quantization
\[\la_\theta(f)=\frac{1}{(2\pi)^d}\int_{\R^d}\hat{f}(\bxi)\la_\theta(\bxi)d \bxi \pl,\pl f\in \S(\R^d). \]
Here $\hat{f}(\bxi)=\int_{\R^d} f(\bx)e^{- i\bx\cdot\bxi}d\bx$ is the Fourier transform of $f$ and the integral converges in strong operator topology. Let $u_j(t)=\la_\theta(0,0,\cdots,t,\cdots,0)$ be the one parameter unitary group associated to the $j$-th coordinate. The generator $x_j$ of $u_j(t)$ satisfying $u_j(t)=e^{ix_jt}$ is given by.
\[(x_jg)(\bx)= \bx_jg(\bx)+\frac{i}{2}\sum_{k}\theta_{jk}\frac{\p g}{\p \bx_k}(\bx)\pl.\]
$(x_1,\cdots,x_d)$ are $d$ self-adjoint operators on $L_2(\R^d)$ affiliated to $\R_\theta$ which satisfies the CCR relation $[x_j,x_k]=-i\theta_{jk}$. The projective unitary representation $\bxi\to \la_\theta(\bxi)$ can be recovered from $(x_1,\cdots, x_d)$ using Baker--Campbell--Hausdorff formula i.e.
\[\la_\theta(\bxi):=e^{i(\bxi_1x_1+\cdots+\bxi_d x_d)}=e^{-\frac{i}{2}\sum_{j<k}\theta_{jk}\bxi_j\bxi_k}e^{i\bxi_1x_1}\cdots e^{i\bxi_dx_d}\pl, \pl\bxi\in\R^d\pl\]
The generator $(x_1,\cdot, x_d)$, unitary $\la_\theta(\bxi)$ and the quantized Schwartz class $\la_\theta(f)$ are equivalent formulations of quantum Euclidean spaces. We will use them interchangeably in the paper.

\subsection{The Stone-von Neumann Theorem}We say two self-adjoint operator $P,Q$ satisfies the Heisenberg relation $[P,Q]=-iI$ if for any $s,t\in \R$, \[e^{isP}e^{itQ}=e^{ist}e^{itQ}e^{isP}\pl \]
The well-known Stone-von Neumann Theorem states that any irreducible representations of $[P,Q]=-iI$ is unitarily equivalent to the $1$-dimensional Schrodinger picture that
\[Pf=-i\frac{df}{d\bx}\pl, \pl (Qf)(\bx)=\bx f(\bx)\pl ,\pl f\in \S(\R)\pl.\]
Here $P,Q$ are unbounded self-adjoint operators on $L_2(\mathbb{R})$ and the one-parameter unitary groups are
\begin{align}(e^{itP}f)(\bx)=f(\bx+t)\pl,  \pl (e^{isQ}f)(\bx)=e^{is\bx }f(\bx)\pl, \label{action}\end{align}
The Stone-von Neumann Theorem extends to $n$ pairs of Heisenberg relations that mutually commute, i.e.
\begin{align} [P_j, Q_k]=\begin{cases}
                           -iI, & \mbox{if } j=k \\
                           0, & \mbox{if } j\neq k.
                         \end{cases}\pl, \pl\pl\pl [P_j, P_k]=[Q_j, Q_k]=0 \pl, \pl\pl\pl\forall\pl j,k
\label{standard}\end{align}
The following is the Theorem 14.8 of \cite{Hall}.
\begin{theorem}[Stone–von Neumann Theorem]
Suppose $P_1, \cdots ,P_d$ and $Q_1,\cdots,Q_d$ are self-adjoint operators on $H$ satisfying the CCR relations \eqref{standard}. Then $H$ can be decomposed as an orthogonal direct sum of closed subspaces $\{H_j\}$ satisfying \begin{enumerate}
\item[i)]each $H_l$ is invariant under $e^{itP_j}$ and $e^{itQ_j}$ for all $j$ and $t$.
\item[ii)]there exist unitary
operators $U_l: H_l \to L_2(\R^d )$ such that
\begin{align}\label{mp}U_l P_j U_l^*f= -i\frac{\p}{\p \bx_j}f\pl, \pl(U_l Q_j U_l^*f)(\bx)= \bx_jf(\bx)\pl.\end{align}
\end{enumerate}
\end{theorem}
The above theorem says that any representation of \eqref{standard} is a finite or infinite multiple of the $n$-dimensional Schrodinger picture on $L_2(\R^n)$.
When $d=2n$ is even dimensional,
this gives the standard noncommutative case for $\R_\theta$ that $\displaystyle \theta=\left[ {\begin{array}{cc} 0 & -I_n \\ I_n & 0 \end{array} } \right]$, where $I_n$ is the $n$-dimensional identity matrix. In this case, $\E_\theta\cong K(L_2(\R^n))$ the compact operators and $\R_\theta\cong B(L_2(\R^n))$. The following proposition gives change of variables between $\R_\theta$'s with different $\theta$.
\begin{prop}\label{change}Let $T=(T_{jk})^{d}_{j,k=1}$ be a real invertible matrix and $T^t$ be its transpose. Let $\theta$ and $\tilde{\theta}$ be two skew-symmetric matrices such that $\tilde{\theta}=T\theta T^t$. Then the map $\Phi_T$:
\[\Phi_T(\la_{\tilde{\theta}}(\bxi))=\la_{\theta}(T^t\bxi)\pl,\pl  \Phi_T(\la_{\tilde{\theta}}(f))=\la_{\theta}(f\circ T)\pl\]
extends to a $*$-isomorphism from $\E_{\tilde{\theta}}$ to $\E_\theta$ and a normal $*$-isomorphism from $\R_{\tilde{\theta}}$ to $\R_\theta$.
\end{prop}
\begin{proof}Define the operator $U_T$ on $L_2(\R^d)$ as follows,
\[ (U_Tf)(\bx)=f(T^{-1}\bx)\pl.\]
$U_T$ is bounded and invertible with $\norm{U_T}{}=|\det(T)|^\frac{1}{2}$ and $(U_T)^{-1}=U_{T^{-1}}$. For any Schwartz function $f$, one verifies that
\begin{align*} (U_T^{-1}\la_{\tilde{\theta}}(\bxi)U_T f) (\bx)
= e^{i\bxi\cdot T\bx }f(T^{-1}(T\bx+\frac{1}{2}\tilde{\theta}\bxi))=e^{i(T^t\bxi)\cdot \bx }f(\bx+\frac{1}{2}\theta T^t\bxi)=\la_\theta(T^t\bxi)f (\bx)\pl.
\end{align*}
Then it is clear that $U_T^{-1}\S_{\tilde{\theta}} U_T=\S_\theta$. Since $U_T$ is a bounded invertible operator on $L_2(\R^d)$, then $\Phi_T(\cdot )=U_T^{-1}(\cdot)U_T$ extends to a $*$-isomorphism from  $\E_{\tilde{\theta}}$ to $\E_\theta$ and a normal $*$-isomorphism from $\R_{\tilde{\theta}}$ to $\R_\theta$.
\end{proof}
In general, let $\theta$ be a skew-symmetric matrix of rank $2n\le d$. There exists an invertible matrix $T$ such that $\tilde{\theta}=T\theta T^t$ is the following standard form
\begin{align}
\left[
  \begin{array}{ccc}
    0         &-I_n          &           \\
    I_n       &0            &           \\
             &             &0_{d-2n}
  \end{array}
\right] \label{sf},
\end{align}
where $0_{d-2n}$ is $(d-2n)\times (d-2n)$ zero matrix.
Let $x_1,\cdots, x_d$ be the generators of $\E_{(\tilde{\theta})}$. Then $x_1,\cdots, x_2n$ by Stone-von Neumann theorem are unitary equivalent to (a multiple of) the derivatives and position operators $\displaystyle -i\frac{\partial}{\partial \bx_1},\cdots, -i\frac{\partial}{\partial \bx_n}, \bx_1,\cdots, \bx_n$ on $L_2(\R^n)$, and $x_{2n+1},\cdots, x_{d}$ are $d-2n$ the position operators $\bx_{n+1},\cdots, \bx_{d-n}$ on $L_2(\R^{d-2n})$. Hence if $\theta$ is of rank $2n<d$, we have up to multiplicity \[\E_\theta \cong \K(L_2(\R^n))\ten C_0(\R^{d-2n})\pl , \pl \R_\theta \cong B(L_2(\R^n))\overline{\ten} L_\infty(\R^{d-2n})\] In particular, the $C^*$-algebra $\E_\theta$ is simple if and only if the matrix $\theta$ is of full rank.
\subsection{Integrals and Derivatives } \label{diff} We start with the noncommutative integrals.
\begin{prop}\label{trace}The linear functional  \[\tau_\theta(\la_\theta(f))=\int_{\R^d}f\pl , \pl f\in \S(\R^d)\]
extends to a normal faithful semi-finite trace on $\R_\theta$.
\begin{enumerate}
\item[i)] Let $T$ be a real invertible matrix and $\theta,\tilde{\theta}$ be two skew-symmetric matrix such that $\tilde{\theta}=T\theta T^t$. Then the normal $*$-isomorphism
\begin{align}\Phi_T:\R_{\tilde{\theta}} \to \R_{\theta}\pl, \Phi_T(\la_{\tilde{\theta}}(f))=\la_{\theta}(f\circ T) ,\label{Q}\end{align}
satisfies $\tau_\theta\circ \Phi_T = |det T|^{-1}\tau_{\tilde{\theta}}$.
 \item[ii)]
 Let $\bx\in\R^d$ and $\al_\bx$ be the translation action $\al_\bx(f)(\cdot)=f(\cdot+\bx)$. Define the map
\begin{align*}\al_\bx(\la_\theta(\bxi))=e^{i\bxi\cdot \bx} \la_\theta(\bxi)\pl , \pl \al_\bx(\la_\theta(f))=\la_\theta(\al_\bx(f))\pl.\end{align*}
Then $\al_\bx$ is a $\tau_\theta$-preserving automorphism on $\R_\theta$.
\end{enumerate}\label{trace}
\end{prop}
\begin{proof}The fact $\tau_\theta$ is a normal faithful trace on $\R_\theta$ was proved in \cite{GJP17} by writing $\R_\theta$ as an iterated crossed product $L_\infty(\R)\rtimes \R\rtimes \cdots \rtimes\R$. Here we present a proof using change of variables, which is useful for our later discussion. A similar discussion can be found in \cite{cwikel}. Denote the multiplier and translation unitary groups on $L_2(\R^n)$ as follows,
\[(u(\bxi)f)(\bx)=f(\bx+\bxi)\pl, \pl (v(\bet)f)(\bx)=e^{i\bet\cdot \bx}f(\bx)\pl.\]
We first consider the case $d=2n$ and $\displaystyle \theta=\left[ {\begin{array}{cc} 0 & -I_n \\ I_n & 0 \end{array}}\right]$. By the Stone-von Neumann theorem, there exists some Hilbert space $H$ and a unitarily $W:L_2(\R_\theta)\to L_2(\R^n)\ten I_H$ such that
\[W\la_\theta(\bxi,{\bf 0})W^*=u(\bxi)\ten I_H\pl, \pl W\la_\theta({\bf 0},\bet)W^*= v(\bet)\ten I_H\pl,\]
where $\bxi\in \R^n$ are the first $n$ coordinates and $\bet\in \R^n$ are the last $n$ coordinates. For $f_1,f_2\in \S(\R^n)$, the quantization $\la_\theta(f_1\ten f_2)$ is unitarily equivalent to (a multiple of) the following operator $T_{f_1,f_2}$. For $h\in L_2(\R^n)$
\begin{align*}
(T_{f_1,f_2} h)(\by)&=(2\pi)^{-2n}\int \int \hat{f_1}(\bxi)\hat{f_2}(\bet)e^{-\frac{i}{2}\bxi\cdot\bet}e^{i
\bet\cdot (\by+\bxi)}h(\by+\bxi) d\xi d\bet\\
&=(2\pi)^{-2n}\int \int\hat{f_1}(\bx-\by)\hat{f_2}(\bet)e^{-\frac{i}{2}(\bx-\by)\cdot \bet}e^{i
\bx\cdot\bet}h(\bx) d\bx d\eta
\\
&= (2\pi)^{-n}\int \hat{f_1}({\bx-\by})f_2(\frac{\bx+\by}{2})h(\bx) d\bx \pl.
\end{align*}
Bacause $f_1,f_2\in \S(\R^n)$, it follows from \cite[Proposition 1.1 and Theorem 3.1]{brislawn88} that $T_{f_1,f_2}$ is a trace class operator on $L_2(\R^n)$ and
\begin{align*}tr(T_{f_1,f_2})=&(2\pi)^{-n}\int_{\R^n}\hat{f_1}(\by-\by)f_2(\frac{\by+\by}{2})d\by
\\=&(2\pi)^{-n}\int_{\R^n}\hat{f_1}(0)f_2(\by)d\by=(2\pi)^{-n}\int_{\R^n} f_1 \cdot \int_{\R^n} f_2\pl,
\end{align*}
which coincides with $\tau_\theta$ on $\R_\theta$ up to a normalization constant $(2\pi)^{-n}$. Now we consider $\theta$ is a singular standard form $\displaystyle \theta=\left[ {\begin{array}{ccc} 0 & -I_n & 0\\ I_n & 0 & 0\\ 0 & 0 & 0\end{array}}\right]$.
Let $\theta_1=\left[\begin{array}{cc} 0 & -I_n \\ I_n & 0 \end{array}\right]$ be the nonsingular part. ${\R_{\theta_1}}\cong B(L_2(\R^n))$ is a Type I factor and the degenerated part gives the left regular representation $\la_{0}:\R^{d-2n}\to B(L_2(\R^{d-2n}))$. Then,
\[\R_\theta\cong \R_{\theta_1}\overline{\ten} \R_0\cong B(L_2(\R^n))\overline{\ten} L_\infty(\R^{d-2n})\]
 as von Neumann algebras. The trace $\tau_\theta$ on $\R_\theta$ is the product trace $\tau_{\theta_1}\ten \tau_0$, where $\tau_{0}$ on $L_\infty(\R^{d-2n})$ is the Lebesgue integral and $\tau_{\theta_1}$ is up to a constant the standard trace $tr$ on $B(L_2(\R^n))$. Then $\tau_\theta$ is normal faithful semifinite and the case for general $\theta$ follows from $i)$. Recall that the $*$-isomorphism $\Phi_T$ is implemented by the bounded invertible operator
\[U_T:L_2(\R_{\tilde{\theta}})\to L_2(\R_{\theta})\pl, \pl U_T \la_{\tilde{\theta}}(f)=\la_{\theta}(f\circ T^{-1})\pl.\]
For $f\in \S(\R^d)$,
\begin{align*}\tau_\theta\circ \Phi_T(\la_{\tilde{\theta}}(f))=&\tau_\theta\Big(\int_{\R^d}\hat{f}(\bxi)\la_\theta(T\bxi)d\bxi \Big)= |\det T|^{-1}\tau_\theta\Big(\int_{\R^d}\hat{f}\big(T^{-1}\bet\big)\la_\theta(\bet)d\bet\Big)
\\=&|\det T|^{-1}\hat{f}(0)=|\det T|^{-1}\tau_{\tilde{\theta}}(\la_{\tilde{\theta}}(f))\pl.\end{align*}
For ii), $\al_\bx$ is implemented by the shifting unitary $U_\bx$ on $L_2(\R^d))$ that
\[\al_\bx(\la_\theta(f))= U_\bx \la_\theta(f) U_\bx^*\pl , \pl U_\bx f (\by)=f(\by+\bx) \pl.\]
Hence $\al_\bx$ extends to an automorphism on $\R_\theta$.
\end{proof}
 The automorphisms $\al_\bx,\bx\in \R^d$ is called the \emph{transference action} on $\R_\theta$. For $1\le p\le \infty$, we write $L_p(\R_\theta)$ for the noncommutative $L_p$ space with respect to $\tau_\theta$ and identify $L_\infty(\R_\theta)= \R_\theta$. For all $\theta$, $L_2(\R_\theta)\cong L_2(\R^d)$ and $\la_\theta$ is exactly the left regular representation of $\R_\theta$ on $L_2(\R_\theta)$. It is clear that $S_\theta$ is dense in $\E_\theta$ and $L_2(\R_\theta)$.

 \begin{lemma}$\S_\theta$ is dense in $L_1(\R_\theta)$.
\end{lemma}
\begin{proof}If $a\in L_1(\R_\theta)$, then $a=a_1a_2$ for some $a_1,a_2\in L_2(\R_\theta)$ and $\norm{a_1}{2}=\norm{a_2}{2}=\norm{a}{1}^{\frac{1}{2}}$. Then we can find $f_1,f_2\in S(\R^d)$ such that $\norm{\la_\theta(f_j)-a_j}{2}\le \epsilon, j=1,2$. Then
\begin{align*}\norm{a-\la_\theta(f_1)\la_\theta(f_2)}{}&\le \norm{a_1a_2-a_1\la_\theta(f_2)}{1}+\norm{a_1\la_\theta(f_2)-\la_\theta(f_1)\la_\theta(f_2)}{1}\\&\le \norm{a_1}{2}\epsilon +\norm{f_2}{2}\epsilon\le (2\norm{a}{1}^{\frac{1}{2}}+\epsilon)\epsilon\pl. \qedhere
\end{align*}
 \end{proof}
The noncommutative Lorentz space $L_{p,\infty}(\R_\theta)$ is the space of measurable operators $a$ affiliated to $\R_\theta$ such that the following quasi-norm is finite
\[\norm{a}{L_{p,\infty}}^p=\sup_{t>0}t^p\tau_\theta(1_{|a|>t})\pl,\]
where $1_{|a|>t}$ denote the spectral projection of $|a|$. In other words, $a\in L_{p,\infty}(\R_\theta)$ if $\tau_\theta(1_{|a|>t})$ is asymptotically at most $O(t^{-p})$. For $det(\theta)\neq 0$, the above (weak) $L_p$ spaces are nothing but the (weak) Schatten $p$-spaces.
\begin{prop}\label{interg}Denote $|x|:=(\sum_{j}x_j^2)^{\frac12}$ and $\lan x \ran:=(1+\sum_j x_j^2)^{\frac12}$. For all $\theta$,
\begin{enumerate}
 \item[i)] $\bc{x}^{-1}\in L_{d,\infty}(\R_\theta)$.
  \item[ii)]
$\displaystyle \tau_\theta(e^{-t|x|^2}) = t^{-\frac{d}{2}} det(\frac{\pi it\theta}{\sinh (it\theta)})^{1/2}$ for $t>0$.
\end{enumerate}
Here the function $\displaystyle \mu \mapsto \frac{\pi \mu}{\sinh  \mu}$ is a real function continuously extended to $\mu=0$ and $\displaystyle \frac{\pi i\theta  }{\sinh (i\theta )}$ is the function calculus for self-adjoint matrix $i\theta$.
\end{prop}
\begin{proof}
 Let us first consider that $\theta$ is the standard form \eqref{sf} of rank $2n$.
 We have shown in Proposition \ref{trace} that there is (up to a factor $(2\pi)^n$) a trace preserving $*$-isomorphism $\pi:\R_\theta \to B(L_2(\R^n))\overline{\ten} L_\infty(\R^{d-2n})$ on $L_2(\R^{d-n})$ such that for $1\le j\le n, 1\le k\le d-2n$
\[x_j \mapsto D_{\by_j}\pl , \pl  x_{j+n}\mapsto \by_j\pl ,\pl  x_{2n+k}\mapsto \by_{n+k} \pl.\]
where $D_{\by_j}$ and $\by_j$ are the self-adjoint derivative and position operators on $L_2(\R^{d-n})$
\[ D_{\by_j}g= -i\frac{\p g}{\p \by_j} \pl, (y_jg)(\by)=\by_j g(\by)\pl.\]
Then $\bc{x}^2$ is unitary equivalent to (a multiple) of the following operator on $L_2(\R^{d-n})$,
\[H:=(\sum_{j=1}^n D_{\by_j}^2+ \by_j^2)\ten id_{L_2(\R^{d-2n})} +id_{L_2(\R^n)}\ten (1+\sum_{l=n+1}^{d-n} \by_l^2)\pl.\]
The first part is the Hamiltonian of $n$-dimemsional quantum harmonic oscillator and the second part is a multiplier on $L_2(\R^{d-2n})$. It is known (see \cite[Chapter 11]{Hall}) that $H_1:=(\sum_{j=1}^n D_{\by_j}^2+ \by_j^2)$ has discrete spectrum $\mu_N=2N+n$ and the degeneracy of $\mu_N$ is $\binom{N+n-1}{N}$.
Combined with the continuous part on $L_\infty(\R^{d-2n})$, we have
\begin{align*}\tau_\theta(1_{H\le \mu})&=(2\pi)^n\sum_{
2N\le \mu-n}\binom{N+n-1}{N}\int_{\R^{d-2n}} {1}_{(1+|\by|^2)\le \mu -2N-n}d\by \\ &\lesssim \pl\mu\cdot\mu^{n-1}\cdot\mu^{\frac{d-2n}{2}}= \mu^{\frac{d}{2}} \pl.\end{align*}
Thus $\displaystyle \tau_\theta( 1_{ H^{-1/2} >\mu})\lesssim \mu^{-d}$ which implies $H^{-1/2}\in L_{d,\infty}$. The case for general $\theta$ follows from the change of variable in Proposition \ref{trace}. Moreover, if $T$ is a real invertible matrix such that $T\theta T^t$ is the standard form \eqref{sf}, then $det(T)=(\mu_1\mu_2\cdots\mu_n)^{-1}$, where $\mu_1,\mu_2,\cdots,\mu_n$ are imaginary parts of eigenvalues of $\theta$.  Thus, by the isomorphism in \eqref{Q}, we have
\begin{align*}
\tau_\theta(e^{-t|x|^2})&=\mu_1\mu_2\cdots \mu_n (2\pi)^n \cdot tr(e^{-t\sum_{j=1}^n\mu_j(D_{\by_j}^2+\by_j^2)}) \cdot \int_{\R^{d-2n}} e^{-t\sum_{j=n+1}^{d-n}\by_j^2}d\by_{n+1}\cdots d\by_{d-n}
\\ &= \mu_1\mu_2\cdots \mu_n (2\pi)^n\cdot  \big(\prod_{j=1}^{n}\sum_{k=0}e^{-t\mu_j(1+2k)}\big)\cdot(\frac{\pi}{t})^{\frac{d-2n}{2}}
\\&  =  \big(\prod_{j=1}^{n}\frac{2\pi t\mu_{j}}{e^{t\mu_j}-e^{-t\mu_j}}\big)(\pi)^{\frac{d-2n}{2}}t^{-\frac{d}{2}}
\\&  =  t^{-\frac{d}{2}}\big(\prod_{j=1}^{n}\frac{\pi t\mu_{j}}{\sinh t\mu_j}\big)(\pi)^{\frac{d-2n}{2}}
\\&  =  t^{-\frac{d}{2}} det(\frac{\pi it\theta}{\sinh (it\theta)})^{1/2}\pl.
\end{align*}
The last equality follows from $\displaystyle \lim_{\mu\to 0}\frac{\pi \mu}{\sinh (\mu)}=\pi$.
\end{proof}

Let $D_{\bx_1},\cdots, D_{\bx_d}$ be the partial derivatives operator
$\displaystyle D_{\bx_j}f=-i\frac{\partial}{\partial \bx_j}f$, which are unbounded self-adjoint operators on $L_2(\R^d)$ with a common core $\S(\R^d)$. On $\R_\theta$, we define for $\la_\theta(f)$ in $\S_\theta\subset B(L_2(\R^d))$ the partial derivatives
    \[ D_j\la_\theta(f):=[D_{\bx_j},\la_\theta(f)]=\la_\theta(D_{\bx_j}f).\]
Here $\be_j=(0,\cdots,1,\cdots,0)$ is the $j$-th standard basis of $\R^d$. Since $D_{\bx_j}$ is the same as $D_j$ for $\theta=0$, we will often write $D_{\bx_j}$ simply as $D_j$. Let $\S'(\R^d)$ be the space of tempered distribution on $\R^d$. In \cite{moyal1,moyal2} (see also \cite{moyalplane}), Moyal product and the Weyl quantization are weakly extended to $\S'(\R^d)$ as follows,
\[\lan T\star_\theta f, g\ran = \lan T, f\star_\theta g\ran\pl, \lan f\star_\theta T, g\ran = \lan T,  g\star_\theta f\ran\pl.\]
where the bracket is the pairing between $\S(\R^d)$ and $\S'(\R^d)$. For $T\in S'(\R^d)$, $\la_\theta(T)$ is the quantized operator $\la_\theta(T) f= T\star_\theta f$ and satisfies
\begin{align*}\la_\theta(T)\la_\theta(f)=\la_\theta(T\star_\theta f) ,  \la_\theta(f)\la_\theta(T)=\la_\theta(f\star_\theta T)\pl.\end{align*}
For all $T\in \S'(\R^d)$, $\la_\theta(T)$ commutes with the right Moyal multiplication hence affiliates to $\R_\theta$. We will use the multiplier algebra introduced in \cite{moyal2},
\begin{align*}
\M_\theta=\{\la_\theta(T) \pl | \pl T\in \S'(\R^d),  \la_\theta(T) \S_\theta\subset \S_\theta, \S_\theta\la_\theta(T)\subset \S_\theta \}\pl.
\end{align*}
The pairing between $\S(\R^d)$ and $\S'(\R^d)$ coincides with the $\tau_\theta$-trace duality for the quantization. Namely for $\la_\theta(T)\in \M_\theta, \la_\theta(f)\in \S_\theta$,
\[\textstyle \tau_\theta(\la_\theta(T)\la_\theta(f))=\tau_\theta(\la_\theta(T\star_\theta f))=\int T\star_\theta f= \lan T, f\ran \]
In particular, $\M_\theta$ contains the noncommutative polynomials of $x_1,\cdots, x_d$ as the quantized coordinate function $\bx_j$,
\[\textstyle\la_\theta(\bx_j)=x_j\pl, \pl x_j \la_\theta(f)=\la_\theta(\bx_j f)+\frac{1}{2}\sum_{k}\theta_{jk}D_k\la_\theta(f)\pl.\]
The transference automorphism $\al_\bx$ and the partial derivatives $D_j$ weakly extend to $\M_\theta$
\begin{align*}\lan \al_\bx(a), \la_\theta(f)\ran:=\lan a, \al_{-\bx}\la_\theta(f)\ran\pl,
\lan D_j(a), \la_\theta(f)\ran= \lan a, D_j\la_\theta(f)\ran \pl.\end{align*}
Viewing $a\in \M_\theta$ as an unbounded operator densely defined on $S(\R^d)\subset L_2(\R^d)$, the weak derivatives satisfies $D_j(a)=[D_j,a]$.

\section{Asymptotic degrees}\label{sectionasy}
\label{asymptotic} In this section, we introduce a notation of ``asymptotic degrees'' to measure the ``growth'' of unbounded elements in $\R_\theta$, which serves as a key technical tool for later discussions. The idea is inspired from
the abstract $\Psi DO$s introduced by Connes and Moscovici in \cite{CM90,CM95}. We briefly recall the basic setting here. Let $D$ be a (possibly unbounded) self-adjoint operator on a Hilbert space $H$ such that $|D|$ is strictly positive. For each $s\in\R$, put $H^s= Dom(|D|^s)$ with inner product
\[\lan v_1,v_2\ran_{H^s}:=\lan |D|^sv_1,|D|^sv_2\ran_H\pl, \pl v_1,v_2\in Dom(|D|^s)\]
Let $H^{\infty}=\cap_{s\in\Z} H^s$. Because $Dom(e^{|D|^2})\subset H^{\infty}$, $H^{\infty}$ is a dense subspace of $H$. Let $F$ be a closed operator on $H$ such that $H^\infty\subset Dom(F), F(H^\infty)\subset H^\infty$. Because $|D|^{-s}: H^0 \to H^s$ is an isometric isomorphism, one sees that
\[\norm{F: H^s\to H^{s-r}}{}=\norm{|D|^{s-r}F|D|^{-s}}{}\]
For a fixed $r\in \R$, $F$ extends to a bounded operator from $H^s$ to $H^{s-r}$ for any $s$ if and only if $|D|^{s-r}F|D|^{-s}$ are bounded on $H$. Such $F$ is considered as an abstract $\Psi$DO of order $r$.

We use the above idea to characterize the asymptotic degree (we use the word ``degree'' to distinguish with the notation ``order'' for $\Psi$DOs) of elements in $\M_\theta$. We choose the strictly positive operator $D$ as $\bc{x}:=(1+\sum_{j}x_j^2)^{\frac{1}{2}}$.
\begin{defi}
We say an operator $a\in \M_\theta$ is of {\bf asymptotic degree} $r$ if for any $s\in\R$, \[\bc{x}^{s}a\bc{x}^{-s-r}\] extends to a bounded operator in $B(L_2(\R_\theta))$ (hence also in $\R_\theta\subset B(L_2(\R_\theta))$). We denote $O^r$ the set of all elements of asymptotic degree $r$ and write $O^{-\infty}=\cap_{r\in\mathbb{Z}} O^r$. 
\end{defi}
Let $L_2^s(\R_\theta)$ be the Hilbert space completion of $\S_\theta$ with respect to the inner product
\[\lan \la_\theta(f),\la_\theta(g)\ran_s = \tau_\theta(\la_\theta(f)^*\bc{x}^{2s}\la_\theta(g)) \pl.\] It is clear that $a\in O^r$ if and only if for any $s\in \R$, the left multiplication operator $\la_\theta(f)\mapsto a\la_\theta(f)$ extends continuously from $L_2^s(\R_\theta) $ to $L_2^{s-r}(\R_\theta)$. The following theorem estimates the degrees of some common elements. We introduce the standard notation of multi-indices that for $\al=(\al_1,\al_2,\cdots,\al_d)$, \[x^\al:=x_1^{\al_1}x_2^{\al_2}\cdots x_d^{\al_d}\pl,  \pl  D^\al:=D_1^{\al_1}D_2^{\al_2}\cdots D_d^{\al_d}\pl.\]
Note that the product $x^\al$ is ordered because $x_j$'s are noncommutative.
\begin{theorem}\label{order}For all multi-indices $\al$ and $r\in \R$,
\begin{align*}x^{\al}\in O^{|\al|}\pl, \pl
[x^{\al},\bc{x}^r]\in O^{r+|\al|-2}\pl,\pl
D^{\al}(\bc{x}^r)\in O^{r-|\al|}\pl.\end{align*}
\end{theorem}
\begin{proof}We divide the proof into several steps. \\
\emph{Step 1.}: $[D_j,\bc{x}^{-r}]\bc{x}^{r+1}, [x_j,\bc{x}^{-r}]\bc{x}^{r+1}$ are bounded for $0< r<2$.\\
We use the fractional power for a positive operator $A$,
\[A^{-s}= C_s\int_0^\infty (t+A)^{-1} t^{-s}dt \pl, \pl 0<s <1\pl,\]
where $C_s$ is a nonzero constant depending on $s$. Since the constant does not affect the boundedness, we suppress all constant $C_s$'s. Denote $\Delta:=\bc{x}^2=1+\sum_j x_j^2$. For $0< r<2$,
\begin{align*}
[D_j, \bc{x}^{-r}]&=\int_0^{\infty}  [D_j, (t+\Delta)^{-1}]t^{-\frac{r}{2}}dt\\
&=\int_0^{\infty}  (t+\Delta)^{-1}[(t+\Delta),D_j](t+\Delta)^{-1}t^{-\frac{r}{2}}dt\\
&=2i\int_0^{\infty}  (t+\Delta)^{-1}x_j(t+\Delta)^{-1}t^{-\frac{r}{2}}dt\\
&=2i\int_0^{\infty}  x_j(t+\Delta)^{-2}t^{-\frac{r}{2}}dt+ 2i\int_0^{\infty}  [(t+\Delta)^{-1}, x_j](t+\Delta)^{-1}t^{-\frac{r}{2}}dt\\
&=2i\int_0^{\infty}  x_j(t+\Delta)^{-2}t^{-\frac{r}{2}}dt+ 2i\int_0^{\infty} (t+\Delta)^{-1} [x_j,(t+\Delta)](t+\Delta)^{-2}t^{-\frac{r}{2}}dt\\
&=2ix_j\int_0^{\infty}  (t+\Delta)^{-2}t^{-\frac{r}{2}}dt+ 2\sum_{k}\theta_{jk}\int_0^{\infty} (t+\Delta)^{-1} x_k(t+\Delta)^{-2}t^{-\frac{r}{2}}dt
\end{align*}
For the first integral,
\begin{align*}2ix_j\int_0^{\infty} (t+\Delta)^{-2}t^{-\frac{r}{2}}dt\cdot \ddd^{\frac{1+r}{2}}=2ix_j \ddd^{-1-\frac{r}{2}}\ddd^{\frac{1+r}{2}} =2ix_j\ddd^{-\frac{1}{2}}\end{align*}
is bounded. For the second integral,
\begin{align*}\norm{\int_0^{\infty} (t+\Delta)^{-1} x_k(t+\Delta)^{-2}t^{-\frac{r}{2}}dt\bc{x}^{1+r}}{}&\le \int_0^{\infty} \norm{(t+\Delta)^{-2+\frac{r}{2}}}{}t^{-\frac{r}{2}} d t
\\&\le \int_0^{\infty} (t+1)^{-2+\frac{r}{2}}t^{-\frac{r}{2}} d t<\infty
\end{align*}
converges absolutely. For the commutator with $x_j$, we have
\begin{align*}[x_j,\bc{x}^{-r}]=&\int  (t+\Delta)^{-1}[(t+\Delta),x_j](t+\Delta)^{-1}t^{-\frac{r}{2}}dt
\\ &= 2i\sum_{k}\theta_{jk}\int  (t+\Delta)^{-1}x_k(t+\Delta)^{-1}t^{-\frac{r}{2}}dt
= 2i\sum_{k}\theta_{jk}[D_j, \bc{x}^{-r}].
\end{align*}Then $[x_j,\bc{x}^{-r}]\bc{x}^{r+1}$ for $0<r<2$ which is bounded by previous case. In particular, we also obtained \[\bc{x}^{-r}x_j\bc{x}^{r+1}=[\bc{x}^{-r},x_j]\bc{x}^{r+1}+x_j\bc{x}^{-1}\] is bounded for $0<r<2$.\\
\emph{Step 2.} $[x_j, \bc{x}^{-r}]\bc{x}^{r+1},[D_j, \bc{x}^{-r}]\bc{x}^{r+1}$ are bounded for all $r$.\\
First for $-2<r<0$, the bounededness follows from
\[[x_j, \bc{x}^{-r}]\bc{x}^{r+1}=[x_j,\bc{x}^{-r-2}]\bc{x}^{r+3}+2i\sum_{k}\theta_{jk}\bc{x}^{-r-2}x_k\bc{x}^{r+1}\pl.
\]
Then we have the initial case for $-2<r<2$ and use the the following induction steps $r\to -r+1$ for $r <0$ and $r\to -r-1$ for $r>0$,
\begin{align*}[x_j,\bc{x}^{r}]\bc{x}^{-r+1}&=\bc{x}[x_j,\bc{x}^{r-1}]\bc{x}^{-r+1}+[x_j,\bc{x}]\\&=\bc{x}^{r}[\bc{x}^{-r+1},x_j]+[x_j,\bc{x}]\\
[x_j,\bc{x}^{r}]\bc{x}^{-r+1}&=\bc{x}^{-1}[x_j,\bc{x}^{r+1}]\bc{x}^{-r+1}+[x_j,\bc{x}^{-1}]\bc{x}^2\\&=\bc{x}^{r}[\bc{x}^{-r-1},x_j]\bc{x}^2+[x_j,\bc{x}^{-1}]\bc{x}^2
\\&=\bc{x}^{r}[\bc{x}^{-r-1},x_j]-\bc{x}^{-1}[\bc{x}^{2},x_j] +[x_j,\bc{x}^{-1}]\bc{x}^2\pl.
\end{align*} The argument for $[D_j, \bc{x}^{-r}]\bc{x}^{r+1}$ is similar. \\
\emph{Step 3.} $x^\al \in O^{|\al|}$ and $[x^\al,\bc{x}^r]\in O^{|\al|+r-2}$ for all $\al$ and $r$.\\
First, by Step 2 we have that for all $s$
\begin{align*}\bc{x}^sx_j\bc{x}^{-s-1}&=[\bc{x}^s,x_j]\bc{x}^{-s-1}+x_j\bc{x}^{-1}\\
\bc{x}^{-s}[x_j,\bc{x}^{r}]\bc{x}^{-r+s+1}&=[x_j,\bc{x}^{r-s}]\bc{x}^{-r+s+1}+[x_j,\bc{x}^{-s}]\bc{x}^{s+1}\\
\bc{x}^{-s}[D_j,\bc{x}^{r}]\bc{x}^{-r+s+1}&=[D_j,\bc{x}^{r-s}]\bc{x}^{-r+s+1}+[D_j,\bc{x}^{-s}]\bc{x}^{s+1}
\end{align*}
are all bounded. This implies
\[x_j\in O^1\pl, \pl[x_j,\bc{x}^{r}]\in O^{r-1}\pl,\pl[D_j,\bc{x}^{r}]\in O^{r-1}\pl.\]
Thus $x^\al \in O^{|\al|}$ by product. For $[x^\al,\bc{x}^{r}]$, we use the induction step that by the Leibniz's rule
\begin{align*}[x_j x^\al ,\bc{x}^r]=x_j[x^{\al},\bc{x}^r]+
[x_j,\bc{x}^r]x^{\al}\pl,
\end{align*}
and $[x_j,x^\al]$ is a polynomial of order less than $|\al|$.
\emph{Step 4.} $D^\al(\bc{x}^r)\in O^{r-|\al|}$ for all $r\in\R$.\\
We first do induction on $|\al|$ for $-2<r=-2s<0$. For $0<s<1$, we introduce the following notation
\[I_s(a_1,a_2,\cdots,a_l):=\int_0^\infty t^{-s}(t+\ddd)^{-1}a_1(t+\ddd)^{-1}a_2(t+\ddd)^{-1}\cdots (t+\ddd)^{-1}a_l(t+\ddd)^{-1}dt\pl.\]
 For $|\al|=1$, $[D_j,\bc{x}^{-2s}]=2iI_s(x_j)$. Note that by Leibniz rules \begin{align}[D_j,I_\al(a_1,\cdots,a_l)]=&\sum_{1\le k\le l}I_\al(a_1,\cdots,\underbrace{[D_j,a_k]}_{k\text{th}},\cdots,a_l)\nonumber \\&+\sum_{1\le k\le l+1}I_\al(a_1,\cdots,\underbrace{[\ddd,D_j]}_{k\text{th}},a_k,\cdots,a_l)\pl.\label{induction}
\end{align}
Then all higher order derivatives of $\bc{x}^{-2s}$ are sum of
$I_s(a_1,a_2,\cdots,a_l)$ terms with $a_1,\cdots,a_l\in\{1,x_1,\cdots,x_n\}$. Moreover, their degree can be tracked inductively. Let $s_k$ be the degree of $a_k$. We show in the next lemma that $I_s(a_1,\cdots,a_l)$ is at most of degree $-2l-2s+\sum_{k}s_k$. Now assume that for $|\al|\le N$, $D^\al(\bc{x}^r)$ is a sum of the terms $I_s(a_1,a_2,\cdots,a_l)$ with $-2l-2s+\sum_{k}s_k\le r-|\al|$. Then $[D_j,D^\al(\bc{x}^r)]$ is a sum of commutators as \eqref{induction}. The degree of the first part in \eqref{induction} is lowered by $1$ because $[D_j,x_j]=-i$ and $[D_j,1]=0$, and the second part has the degrees at most \[-2(l+1)-2s+(1+\sum_{k}s_k)=-2l-2s-1+\sum_{k}s_k\] because $[\ddd,D_j]=2ix_j$ and the length $l$ is increased by $1$. Thus by induction on $|\al|$ we prove the case $-2<r<0$. For general $r$, one can always write $r=r_1+r_2+\cdots+ r_l$ as a finite sum of $r_k\in (-2,0]\cup 2\mathbb{N}$. Then by Leibniz rule
\[D_\al(\bc{x}^{r})= \sum_{\al_1+\cdots+\al_l=\al} \binom{\al}{\al_1,\cdots, \al_n} D_{\al_1}(\bc{x}^{r_1})\cdots D_{\al_l}(\bc{x}^{r_l})\pl,\]
where $\binom{\al}{\al_1,\cdots, \al_n}=\al!(\al_1!)^{-1}\cdots (\al_d!)^{-1}$ is the multi-nomial coefficient. For positive integer $m$, $D_\al(x^{2m})$ is a polynomial of degree $2m-|\al|$ and the term
$D_\al(\bc{x}^{r_k}), -2<r_k<0$ has degree at most $r_k-|\al|$ as proved above. Therefore, $D_\al(\bc{x}^{r})$ is of degree at most $\sum_k r_k-|\al_k|=r-|\al|$.
\end{proof}
The following lemma is inspired from the abstract $\Psi$DO calculus in \cite{higson}.
\begin{lemma}Let $0<s<1$ and let $I_s$ be the notation
\begin{align*}I_s(a_1,a_2,\cdots,a_l):=\int_0^\infty t^{-s}(t+\ddd)^{-1}a_1(t+\ddd)^{-1}a_2(t+\ddd)^{-1}\cdots (t+\ddd)^{-1}a_l(t+\ddd)^{-1}dt\pl.\end{align*}
Then\begin{enumerate}
\item[i)] if $a_k\in O^{s_k}$, $I_s(a_1,a_2,\cdots,a_l)\in O^{-2l-2s+\sum_{k}s_k+\epsilon}$ for any $\e>0$
\item[ii)] if $a_k\in \{1,x_1,x_2,\cdots,x_n\}$,
$I_\al(a_1,a_2,\cdots,a_l)\in O^{-2l-2s+\sum_{k}s_k}.$
\end{enumerate}
\end{lemma}
\begin{proof}
Let $q,r\in \R$ with $-q+r=-2l-2s+\sum_{k}s_k+\epsilon$.
\begin{align*}&\bc{x}^{q}\int_0^\infty t^{-s}(t+\ddd)^{-1}a_1(t+\ddd)^{-1}a_2(t+\ddd)^{-1}\cdots (t+\ddd)^{-1}a_l(t+\ddd)^{-1}dt\bc{x}^{-r}
\\=&\int_0^\infty t^{-s}(t+\ddd)^{-1+\al-\epsilon/2} \bc{x}^{q}(t+\ddd)^{-s+\epsilon/2} a_1(t+\ddd)^{-1}\cdots (t+\ddd)^{-1}a_l(t+\ddd)^{-1}\bc{x}^{-r}dt
\end{align*}
Note that
\begin{align*}&\norm{\bc{x}^{q}(t+\ddd)^{-s+\epsilon/2} a_1(t+\ddd)^{-1}a_2(t+\ddd)^{-1}\cdots (t+\ddd)^{-1}a_n(t+\ddd)^{-1}\bc{x}^{-r}}{}
\\ \le &\norm{\bc{x}^{2q-\e}(t+\ddd)^{-q+\epsilon/2}}{}\norm{ \bc{x}^{q-2s+\e} a_1\bc{x}^{-q+2s-\e-s_1}}{}\norm{\bc{x}^2(t+\ddd)^{-1}}{}
\\\cdots &\norm{\bc{x}^2(t+\ddd)^{-1}}{}\norm{\bc{x}^{q+\sum_{k\le l-1}s_k -2(n-1)-2s+\e}a_l\bc{x}^{-q-\sum_{k\le l}s_k+2s+2(n-1)-\e}}{}\norm{\bc{x}^2(t+\ddd)^{-1}}{}
\\\le &\norm{ \bc{x}^{q-2s+\e} a_1\bc{x}^{-q+2s-\e-s_1}}{}\cdots \norm{\bc{x}^{q+\sum_{k\le l-1}s_k -2(l-1)-2s+\e}a_l\bc{x}^{-q-\sum_{k\le l}s_k+2s+2(l-1)-\e}}{}
\end{align*}
which is uniformly bounded. Thus
\begin{align*}&\norm{\bc{x}^{q}\int_0^\infty t^{-s}(t+\ddd)^{-1}a_1(t+\ddd)^{-1}a_2(t+\ddd)^{-1}\cdots (t+\ddd)^{-1}a_n(t+\ddd)^{-1}dt\bc{x}^{-r}}{}
\\ \lesssim & \int_0^\infty \norm{t^{-q}(t+\ddd)^{-1+s-\epsilon/2}}{}dt \le \int_0^\infty t^{-s}(t+1)^{-1+s-\epsilon/2}{}dt <\infty \pl.
\end{align*}
For ii), note that
\[I_s(\underbrace{1,\cdots,1}_{l})=\int_0^\infty (t+\ddd)^{-l}t^{-s} dt = C_s \bc{x}^{-2(l-1)-2s}\]
Let $k$ be the last position in $I_s(a_1,\cdots,a_{l})$ such that $a_k$ is not scalar. That is, for all $n\le k$, $a_n=x_{j_n}$ for some $1\le j_n\le d$ and $a_m=1$ for all $k<m\le l$.
\begin{align*}
&I_s(\underbrace{a_1,\cdots,a_{k-1},x_j,1,\cdots,1}_{l})\\=&I_s(\underbrace{a_1,\cdots,a_{k-1},1,x_j,1,\cdots,1}_{l})+I_s(\underbrace{a_1,\cdots,a_{k-1},1,[\ddd,x_j],1,\cdots,1}_{l+1})
\\=&I_s(\underbrace{a_1,\cdots,a_{k-1},1,\cdots,1}_{l})x_j+
\sum_{k+1\le m\le l+1}I_s(a_1,\cdots,a_{k-1},1,\cdots ,\underbrace{[\ddd,x_j]}_{m\pl\text{th}},\cdots,1)
\end{align*}
Note that $[\ddd,x_j]= -2i\sum_{k}\theta_{kj}x_k$. Then by i), the second part belongs to $O^{-2l-2+\sum_{k}s_k-2s+\e}\subseteq O^{-2l+\sum_{k}s_k-2s}$. We then finish the proof by the induction on the last non-scalar position.
\end{proof}
\begin{prop}\label{indu}
i) Let $s\in \R$. If $D^\al(a)\bc{x}^{-s}$ is bounded for all $\al$, then $a\in O^s$. \\
ii) $\S_\theta=\{a\in \R_\theta\pl|\pl  D^\al (a) \in O^{-\infty} \pl \text{for all} \pl \al  \}$. Moreover, the map $f\mapsto \la_\theta(f)$ is bi-continuous from $\S(\R^d)$ equipped with the standard semi-norms to $\S_\theta$ with the semi-norms $\norm{D^\al(\cdot)\bc{x}^{2n}}{}$ for all $\al$ and $n$. In particular, $\bc{x}^r\S_\theta\subset \S_\theta$ for any $r$.
\end{prop}
\begin{proof} i) Define the notation
\begin{align*}  &a^{(1)}:=[\ddd,a]= i\sum_{l}\theta_{jl}(x_jD_{l}(a)+D_{l}(a)x_j);\\
&a^{(2)}:=[\ddd,[\ddd,a]]= -2\sum_{l}\sum_{m}\theta_{jl}\theta_{mj}(x_mD_{l}(a)+D_{l}(a)x_m)
\\&-\sum_{l,m}\theta_{jl}\theta_{km} (x_jx_kD_{l}D_{m}(a)+x_jD_{l}D_{m}(a)x_k+x_kD_{l}D_{m}(a)x_j+D_{l}D_{m}(a)x_kx_j)
\end{align*}
We first give the proof for $s=0$. Assume that $D^\al(a)$ is bounded for all $\al$.
Then $a^{(1)}\bc{x}^{-1}$ is bounded because
\begin{align*}
x_jD_{l}(a)\bc{x}^{-1}&=D_{l}(a)x_j\bc{x}^{-1}+[x_j,D_{l}(a)]\bc{x}^{-1} \\&=D_{l}(a)x_j\bc{x}^{-1}-\sum_{k}\theta_{jk}D_{k}D_{l}(a)\bc{x}^{-1}\pl.\end{align*}
and similarly one can verify that $a^{(2)}\bc{x}^{-2}$ is bounded.
Then for $0<r<2$,
\begin{align*}
[a,\bc{x}^{-r}]\bc{x}^{r}&=I_{\frac{r}{2}}([\ddd,a])\bc{x}^{r}=I_{\frac{r}{2}}(a^{(1)})\bc{x}^{r}
\\&=a^{(1)}I_{\frac{r}{2}}(1)\bc{x}^{r}+I_{\frac{r}{2}}(a^{(2)},1)\bc{x}^{r}
=a^{(1)}\bc{x}^{-1}+I_{\frac{r}{2}}(a^{(2)},1)\bc{x}^{r}\pl.
\end{align*}
The second part is bounded because
\begin{align*}\norm{I_{\frac{r}{2}}(a^{(2)},1)\bc{x}^{r}}{}&\le \int_0^\infty {t^{-\frac{r}{2}}}\norm{(\ddd+t)^{-1}}{} \norm{a^{(2)}(t+\ddd)^{-1}}{}\norm{\bc{x}^{r}(t+\ddd)^{-1}}{}d t
\\ &\lesssim \int_0^\infty t^{-\frac{r}{2}}\norm{\bc{x}^{r}(t+\ddd)^{-2}}{}d t
\le
\int_0^\infty t^{-\frac{r}{2}}(t+1)^{-2+\frac{r}{2}}d t<\infty
\end{align*}
Thus we have $\bc{x}^{-r}a\bc{x}^{r}$ is bounded for $0\le r\le 2$, and for $-2\le r\le 0$ by taking the adjoint.
Moreover, the same argument applies to $D^\beta(a)$ for all $\beta$.
Consider $b=\bc{x}^{-r}a\bc{x}^{r}$.
\[D^\al(b)=\sum_{\al_1+\al_2+\al_3=\al}\binom{\al}{\al_1,\al_2,\al_3}D^{\al_1}(\bc{x}^{-r})D^{\al_2}(a)D^{\al_3}(\bc{x}^{r})\pl.\]
 is bounded for all $\al$ by Leibniz rule and Theorem \ref{order}. Thus we have shown that $\bc{x}^{-r}a\bc{x}^{r}$ bounded for $-4\le r\le 4$. By induction this can be extended for all $r\in\R$ which proves the case $s=0$. For general $s$, we have
\[D^\al(a\bc{x}^{-s})=\sum_{\al_1+\al_2=\al}\binom{\al}{\al_1,\al_2}D^{\al_1}(a)D^{\al_2}(\bc{x}^{-s})\pl,\]
which the assumption $D^\al(a)\bc{x}^{-s}$ is  bounded and $D^{\al_2}(\bc{x}^{-s})\in O^{s-|\al|}$ by Theorem \ref{order}. Thus by the case of $s=0$, we know $a\bc{x}^{-s}\in O^0$ which implies $a\in O^s$.

For ii), we first show that for $f\in \S(\R^d)$, $\la_\theta(f)\bc{x}^{2m}$
is bounded for all positive integers $m$. Note that $\bc{x}^{2m}$ is a polynomial of $x$ with degree $2m$. And
\begin{align*}&x_j\la_\theta(f)=\la_\theta(x_j f + \frac{i}{2}\sum_{k}\theta_{jk}\p_jf) \pl, \\
&\la_\theta(f)x_j=(x_j\la_\theta(\bar{f}))^*=(\la_\theta(x_j \bar{f} +\frac{i}{2}\sum_{k}\theta_{jk}\overline{\p_jf}))^*=\la_\theta(x_j f)  -\frac{i}{2}\sum_{k}\theta_{jk}\la_\theta(\p_jf)
\end{align*}
Then $\la_\theta(f)\bc{x}^{2m}$ are again in $\S_\theta$ hence bounded. Therefore for any $r>0$,
$\la_\theta(f)\bc{x}^{r}$
is bounded and similarly for the derivatives $D^\al(\la_\theta(f))$. Thus by i), $D^\al(\la_\theta(f))\in O^{-\infty}$ for all $\al$. For the other direction, $a\in O^{r}$ for $r<-\frac{d}{2}$ implies
\[\norm{a}{2}\le \norm{\bc{x}^{r}}{2}\norm{\bc{x}^{-r}a}{\infty}<\infty \pl.\]
Thus $a=\la_\theta(f)$ for some $f\in L_2(\R^d)$ and $D^\al(a)=\la_\theta( D_\al(f))$ in the distribution sense.
Then all the derivatives of $f$ belongs to $L_2(\R^d)$ and hence $f$ is in the Sobolev space $H^s(\R^d)=\{f \pl | \pl (1+\ddd)^s f\in L_2(\R^d)\}$ for all $s$. Using Sobolev embedding theorem, $f\in C_0^\infty(\R^d)$ with all derivatives bounded. To see $\bx^\beta f$ are bounded functions for $\beta$, we use induction on $|\beta|$ and
\begin{align}\label{lm}\la_\theta(\bx_j f)=x_j\la_\theta(f)-\frac{i}{2}\sum_{k}\theta_{jk}\la_\theta(D_jf)
\pl.\end{align}
Similarly we know that $D_\al(f)\bx^\beta$ are bounded for all $\al,\beta$. To show the semi-norms are equivalent, let $f\in \S(\R^d)$ and denote $\hat{f}$ as its Fourier transform. Let $n$ be the smallest even integer greater than $\frac{d}{2}$,
\begin{align*} \norm{D^\beta(f)\bc{\bx}^{2m}}{\infty}\le \norm{\widehat{D^\beta(f)\bc{\bx}^{2m}}}{1}
\le \norm{\bc{\bxi}^{n}\widehat{D^\beta(f)\bc{\bx}^{2m}}}{2}\norm{\bc{\bxi}^{-n}}{2}\pl.\end{align*}
Let $\bc{\bxi}^{n}\widehat{D^\beta(f)\bc{\bx}^{2m}}\in \S(\R^d)$ be the Fourier transform of $g$. $g$ can be expressed as a linear combination of $\bx^\beta D^\al(f)$ with $|\al|$ up to $n$, $\beta$ up to $2m$. Therefore,
\begin{align*}
\norm{D^\beta(f)\bc{\bx}^{2m}}{\infty} &\lesssim \pl
\norm{\la_\theta(g)}{2}\lesssim \pl\norm{\la_\theta(g)\bc{x}^{n}}{\infty}\\&\lesssim \pl \sup\{ \norm{D^\al\la_\theta(f)x^\beta}{\infty}\pl |\pl  |\al|\le n,|\beta|\le n+2m\}\pl.
\end{align*}
Finally, we note that $D^\al\la_\theta(f)\in\S_\theta \subset O^{-\infty}$ and by Theorem \ref{order} $D^\al\bc{x}^r\in O^{r-|\al|}$. By product rule, $D^\al(\bc{x}^r\la_\theta(f))\in O^{-\infty}$ for all $\al$. Then
$\bc{x}^r\S_\theta\subset\S_\theta$. \qedhere
\end{proof}
\begin{lemma}\label{bi}Let $\by\in \R^d$. Denote $\bc{x+\by}:=(1+\sum_{j}(x_j+\by_j)^2)^{\frac{1}{2}}$. Then
\begin{enumerate}\item[i)] $\al_{\by}(\bc{x}^r)=\bc{x+\by}^r$.
\item[ii)] for any $0<r\le 2n$ with $n$ integer, there exists a constant $c_{r,n}$ such that
\[\norm{\bc{x+\by}^r\bc{x}^{-r}}{\infty}\le c_{r,n}\bc{\by}^{2n}\pl ,\pl  \norm{\bc{x}^r\bc{x+\by}^{-r}}{\infty}\le c_{r,n}\bc{\by}^{2n}\pl.\]
\end{enumerate}
\end{lemma}
\begin{proof}
It is clear that $\bc{\al_{\by}(x)}^2= 1+\sum_j(x_j+ \by_j)^2=\al_{\by}(\bc{x}^2)$. Then by the fact $\al_{\by}$ is a $*$-isomorphism on $\M_\theta$, $\al_{\by}(\bc{x}^{-2})=\bc{\al_{\by}(x)}^{-2}$.
Then we apply the operator integral for $0< s< 2$,
\[\bc{x}^{-s}= C_r\int_0^\infty (t+\bc{x}^2)^{-1} t^{-\frac{s}{2}}dt \pl. \]
Then the general case follows from writing $r=2n-s$. For ii), for $r=2$,
\begin{align*}&\norm{\bc{x+\by}^2\bc{x}^{-2}}{}\le \norm{1+\sum_{j}2\by_jx_j\bc{x}^{-2}+\sum_{j}\by_j^2\bc{x}^{-2}}{}\lesssim \bc{\by}^2
\\
&\norm{(\bc{x}^{-2}-\bc{x+\by}^2)(t+\bc{x}^2)^{-1}}{}\le \norm{\sum_{j}2\by_jx_j(t+\bc{x}^2)^{-1}+\sum_{j}\by_j^2(t+\bc{x}^2)^{-1}}{}\lesssim t^{-\frac{1}{2}}\bc{\by}^2
\end{align*}
For $r=2n$, $\bc{x}^{2n}$ is a $2n$-degree polynomial of $x_j$ whose largest coefficient is the constant term $\bc{\by}^2n$. By
a similar argument for $\bc{x}^{2n}$, we have
\begin{align*}&\norm{\bc{x+\by}^{2n}\bc{x}^{-2n}}{}\lesssim \bc{\by}^{2n}\pl, \pl \norm{(\bc{x}^{-2n}-\bc{x+\by}^{2n})(t+\bc{x}^{2n})^{-1}}{}\lesssim t^{-\frac{1}{2n}}\bc{\by}^2\pl.
\end{align*}
Using the transference,
\[\norm{\bc{x}^{2n}\bc{x+\by}^{-2n}}{}=\norm{\al_{\by}(\bc{x-\by}^{2n}\bc{x}^{-2n})}{}=\norm{\bc{x}^{2n}\bc{x+\by}^{-2n}}{}\lesssim \bc{\by}^{2n}\]
This proves the inequality for $r=2n$ even integers. For general positive $r$, choose integer $n$ such that $0<r< 2n-1$,  consider $1-\bc{x}^r\bc{x+\by}^{-r}=\bc{x}^r(\bc{x}^{-r}-\bc{x+\by}^{-r})$. Take $s=\frac{r}{2n}<1-\frac{1}{2n}$, we have
\begin{align} \label{inte}&\bc{x}^r( \bc{x}^{-r}-\bc{x+\by}^{-r})\nonumber\\=& C_s \bc{x}^r\int_0^\infty \Big((t+\bc{x}^{2n})^{-1}-(t+\bc{x+\by}^{2n})^{-1}\Big) t^{-s}dt\pl.\nonumber
\\=&C_s \int_0^\infty \Big(  \bc{x}^r(t+\bc{x}^{2n})^{-1}\Big) \Big((\bc{x+\by}^{2n}-\bc{x}^{2n})(t+\bc{x+\by}^{2n})^{-1}\Big) t^{-s}dt\pl.
\end{align}
Note that $\norm{\bc{x}^r(t+\bc{x}^{2n})^{-1}}{}\le (t+1)^{s-1}$ and
\begin{align*}
&\norm{(\bc{x+\by}^{2n}-\bc{x}^{2n})(t+\bc{x}^{2n})}{}\lesssim t^{-\frac{1}{2n}}\bc{\by}^{2n}\pl.
\end{align*}
Therefore,
\begin{align*} &\norm{\bc{x}^r( \bc{x}^{-r}-\bc{x+\by}^{-r})}{}\lesssim \int_0^\infty (1+t)^{s-1}t^{-\frac{1}{2n}-s}\bc{\by}^{2n}dt\lesssim  \bc{\by}^{2n}
\end{align*}
This proves the inequality for $\bc{x}^r\bc{x+\by}^{-r}$ and the other case follows from transference.
\end{proof}
Using the above lemma, we show that quantized partial derivatives defined in Section \ref{diff} are indeed the vector derivatives of transference action.
\begin{prop} \label{strongd}Let $\be_j=(0,\cdots, 1, \cdots, 0)$ be the $j$-th basis vector.\begin{enumerate}
\item[i)] for $\la_\theta(f)\in \S_\theta$,
$\displaystyle D_j\la_\theta(f)=-i\lim_{h\to 0}\frac{1}{h}(\al_{h\be_j}(\la_\theta(f))-\la_\theta(f)) $ in $\S_\theta$.
\item[ii)]Let $m\in \R$. If $a \in \M_\theta$ and $ D^\al(a)\bc{x}^{m} \in \R_\theta$ for all $|\al|\le 2$, then
\[ \lim_{h\to 0} \frac{1}{h} \norm{\Big(\al_{h\be_j}(a)-a-hD_j(a)\Big)\bc{x}^m}{\infty}=0\pl. \]
\end{enumerate}
\end{prop}
\begin{proof}For a Schwartz function $f\in \S(\R^d)$, we have that
\[f(\bx+\by)-f(\bx)=\sum_j\int_{0}^1 \by_j (\partial_jf)(\bx+t\by)dt\pl.\]
In terms of the function $f$, we have
\[ \al_\by(f)-f=\sum_j\int_{0}^1 \by_j \al_{t\by}(iD_jf)dt\pl.\]
Since $\{\al_{t\by}(iD_jf)\pl |\pl 0\le t\le 1\}$ is uniformly bounded for every semi-norm of $\S(\R^d)$, we have $\by\to \al_\by(f)$ is continuous in $\S(\R^d)$. Because $\S_\theta$ and $\S(\R^d)$ have equivalent semi-norms, we have $\by\mapsto \al_\by(\la_\theta(f))=\la_\theta(\al_\by f)$ is also continuous.
\begin{align*}\frac{1}{h}\Big(\al_{h\be_j}(\la_\theta(f))-\la_\theta(f)-h\la_\theta(iD_jf)\Big)&=\int_{0}^1   \al_{t\bh_j}\la_\theta(iD_jf) -\la_\theta(iD_jf)dt\\&=\int_{0}^1  \Big(\al_{th\be_j}\la_\theta(iD_jf)-\la_\theta(iD_jf)\Big)dt \end{align*}
which goes to $0$ in $\S_\theta$ for $h\to 0$ because of the continuity of $\by\to\al_{\by}(\la_\theta(D_jf))$. For ii), we have the integral
\begin{align}\label{in1} \al_{\by}(a)\bc{x}^m-a\bc{x}^m=\sum_j \by_j\int_{0}^1 \al_{t\by}(iD_ja)\bc{x}^m dt\pl.\end{align}
which holds weakly. Suppose $a\bc{x}^m$ and $D_j(a)\bc{x}^m$ are bounded. Then
\[\norm{\al_{\by}(D_ja)\bc{x}^m}{}\le \norm{\al_{\by}(D_ja\bc{x}^m)}{}\norm{\bc{x+\by}^{-m}\bc{x}^m}{}\le \norm{D_ja\bc{x}^m}{}\bc{\by}^{2n}\pl. \]
for some $2n>|m|$. So $\al_{\by}(D_ja)\bc{x}^m$ is uniformly bounded for small $\by$, which by the integral \eqref{in1} implies $\by\mapsto \al_\by(a)\bc{x}^m$ is continuous in norm. Now if $ D^\al(a)\bc{x}^{m}$ bounded for all $|\al|\le 2$, then
\begin{align*}
\norm{\frac{1}{h}\Big(\al_{h\be_j}(a)-a-hD_j(a)\Big)\bc{x}^m}{\infty}\le \int_{0}^1  \norm{\big(\al_{th\be_j}(iD_ja)-iD_ja\big)\bc{x}^m}{\infty}dt
\end{align*}
This goes $0$ in norm as $h\to 0$ because $\by \to \al_\by(D_ja)\bc{x}^m$ is continuous.
\end{proof}

The next proposition gives an approximation of identity for $L_p(\R_\theta)$.
\begin{prop}\label{ai}There exists a sequence $f_n\in \S(\R^d)$ independent of $\theta$ such that i) for any $a\in \E_\theta$ and $p=\infty$; and ii) for any $a\in L_p(\R_\theta)$ and $1\le p< \infty$,
\[ \lim_{n\to \infty}\norm{a\la_\theta(f_n)-a}{p}=\lim_{n\to \infty}\norm{\la_\theta(f_n)a-a}{p}=0\pl.\]
\end{prop}
\begin{proof}Let $\phi\in \S(\R^d)$ be a smooth positive function such that $\phi$ supported on $|\bx|\le 1$ and $\int \phi=(2\pi)^d$. Take $\phi_n=n^d\phi(n\bx)$ and the inverse Fourier transform $\check{\phi}_n$. We first show that for any $\la_\theta(g)\in \S_\theta$, $\norm{\la_\theta(g)\la_\theta(\check{\phi}_n)-\la_\theta(g)}{\infty} \to 0$.
Indeed
\begin{align*} \la_\theta(g)\la_\theta(\check{\phi}_n)=&\Big(\frac{1}{2\pi^d}\int_{\R^d} \hat{g}(\bxi)\la_\theta(\bxi)d\bxi\Big) \Big(\frac{1}{2\pi^d}\int_{\R^d} \phi_n(\bet)\la_\theta(\bet)d\bet\Big)\\
=&\frac{1}{2\pi^{2d}}\int_{\R^d} \int_{\R^d}  \hat{g}(\bxi)\phi_n(\bet)e^{\frac{i}{2}\bxi\theta\bet}\la_\theta(\bxi+\bet)d\bxi d\bet\\
=&\frac{1}{2\pi^{2d}}\int_{\R^d} \Big(\int_{\R^d}  \hat{g}(\bxi)\phi_n(\bet-\bxi)e^{\frac{i}{2}\bxi\theta(\bet-\bxi)}d\bxi\Big) \la_\theta(\bet) d\bet:=\la_\theta(g_n)
\end{align*}
where $\displaystyle \hat{g}_n=\frac{1}{2\pi^{d}}\int_{\R^d}  \hat{g}(\bxi)\phi_n(\bet-\bxi)e^{\frac{i}{2}\bxi\theta(\bet-\bxi)}d\bxi$. Given $\epsilon>0$, we can find $R$ and $n$ large such that $\displaystyle\int_{|\bxi|<R} |\hat{g}(\bxi)|<\frac{\epsilon}{3}$ and $\displaystyle|1-e^{\frac{i}{2}\bxi\theta\bet}|< \frac{\epsilon}{3||\hat{g}||_{1}} $ for all $|\bxi|<R$. Then,
\begin{align*}\norm{\hat{g}-\hat{g}_n}{1}=& \frac{1}{2\pi^{d}}\int_{\R^d} |\hat{g}(\bet)-\int_{\R^d}  \hat{g}(\bxi)\phi_n(\bet-\bxi)e^{\frac{i}{2}\bxi\theta(\bet-\bxi)}d\bxi|d\bet\\
\le &\frac{1}{2\pi^{d}}\int_{\R^d} \int_{\R^d}  | \hat{g}(\bxi)\phi_n(\bet-\bxi)(1-e^{\frac{i}{2}\bxi\theta(\bet-\bxi)})|d\bxi d\bet\\
 \le & \frac{1}{2\pi^{d}}\int_{|\bxi|>R}  \int_{\R^d}  | \hat{g}(\bxi)\phi_n(\bet-\bxi)(1-e^{\frac{i}{2}\bxi\theta(\bet-\bxi)})| d\bet d\bxi \\&+ \frac{1}{2\pi^{d}}\int_{|\bxi|<R}  \int_{\R^d}  | \hat{g}(\bxi)\phi_n(\bet-\bxi)(1-e^{\frac{i}{2}\bxi\theta(\bet-\bxi)})| d\bet d\bxi
\\ \le & \frac{1}{2\pi^{d}}\int_{|\bxi|>R}  \int_{\R^d}  2| \hat{g}(\bxi)|\phi_n(\bet-\bxi) d\bet d\bxi+ \frac{1}{2\pi^{d}}\int_{|\bxi|<R}  \int_{\R^d} \epsilon| \hat{g}(\bxi)|\phi_n(\bet-\bxi) d\bet d\bxi
\\ \le & \frac{2\epsilon}{3}+\frac{\epsilon}{3}=\epsilon
\end{align*}
Hence $\norm{\la_\theta(g_n)-\la_\theta(g)}{\infty}\le \norm{\hat{g}_n-\hat{g}}{1}\to 0$. For $1\le p<\infty$, we apply the argument for $\bc{x}^{d}\la_\theta(g)$. Note that $\bc{x}^{d+1}\la_\theta(g)\in \S_\theta$ by Proposition \ref{indu}. Thus we have
\[\norm{\la_\theta(g)\la_\theta(f_n)-\la_\theta(g)}{p}\le\norm{\bc{x}^{d+1}(\la_\theta(g)\la_\theta(f_n)-\la_\theta(g))}{\infty} \norm{\bc{x}^{-d-1}}{p}\to 0\pl.\] Given $a\in L_1(\R_\theta)$, we choose $g\in \S_\theta$ so that $\norm{\la_\theta(g)-a}{1}\le \epsilon/3$. Note that for all $n$, \[\norm{\la_\theta(\check{\phi}_n)}{\infty}\le \norm{\phi_n}{1}=1\pl.\]
Then for $n$ large enough,
\begin{align}\norm{a-a\la_\theta(\check{\phi}_n)}{1}\le& \norm{a-\la_\theta(g)}{1}+\norm{\la_\theta(g)-\la_\theta(g)\la_\theta(\check{\phi}_n)}{1}+\norm{\la_\theta(g)\la_\theta(\check{\phi}_n)-a\la_\theta(\check{\phi}_n)}{1}
\nonumber\\\le & \norm{a-\la_\theta(g)}{1}+\norm{\la_\theta(g)-\la_\theta(g)\la_\theta(\check{\phi}_n)}{1}+\norm{\la_\theta(g)-a}{1}\norm{\la_\theta(\check{\phi}_n)}{\infty}
\nonumber\\\le & \frac{\epsilon}{3}+\frac{\epsilon}{3}+\frac{\epsilon}{3}=\epsilon\label{3e}
\end{align}
The argument for $\infty$-norm and $a\in \E_\theta$ is similar. For $1<p<\infty$, we use interpolation inequality that
\begin{align*}
\norm{a-a\la_\theta(\check{\phi}_n)}{p}\le \norm{a-a\la_\theta(\check{\phi}_n)}{1}^{\frac{1}{p}}\norm{a-a\la_\theta(\check{\phi}_n)}{\infty}^{1-\frac{1}{p}}\to 0\pl.
\end{align*}
for any $a\in L_1(\R_\theta)\cap L_\infty(\R_\theta)$. Since $L_1\cap L_\infty$ is dense in $L_p$, the argument  for general $a\in L_p$ is similar to \eqref{3e}.
\end{proof}

\section{Pseudo-differential Calculus for Non-commutative Derivatives}
\label{phido}
On $\R^d$ the CCR relation for covariant derivatives corresponds to a constant curvature form. Consider connection
\begin{align}\textstyle \nabla:\mathbb{C}^\infty(\R^d)\to \Omega^1(\R^d)\pl, \pl \nabla f=df+\frac{i}{2}\sum_{j,k}\theta_{j,k}'\bx_jd\bx_k \label{connection}\end{align}
with curvature form $d\omega=\frac{i}{2}\sum_{j,k}\theta_{jk}d\bx_j\wedge d\bx_k\pl$. The self-adjoint covariant derivatives $\nabla_{j}=\nabla_{-\frac{\partial}{\partial_j}}$ satisfy that \[\textstyle\nabla_{j}f= -i\frac{\partial}{\partial \bx_j}(f)- \sum_{k}\frac{1}{2}\theta_{jk}'\bx_k \pl,\pl  [\nabla_{j},\nabla_{k}]= -i\theta_{jk}'\pl.\]
The physical meaning behind this is a constant magnetic field perpendicular to the space $\R^d$. 
In this section, we develop the symbol calculus of $\Psi$DOs of the above structure for a noncommutative $\R_\theta$. Let $\R_\theta$ be the quantum Euclidean space generated by $[x_j,x_k]=-i\theta_{jk}$. 
We equipped $\R_\theta$ with noncommuting covariant derivatives $\xi_j$ satisfying
\begin{align} [\xi_j,x_k]=-i\delta_{jk}, [ \xi_j,\xi_k]=-i\theta_{jk}'\pl.  \label{relation} \end{align}
where $\delta$ is the Kronecker delta notation. For $\theta'=0$,
\cite{GJP17} establish the $\Psi$DOs as operators on $L_2(\R_\theta)$ via $\xi_j=D_j$. For general $\theta$ and $\theta'$, $x_j$'s and $\xi_k$'s satisfying above commutation relations together generate a $2d$-dimensional quantum Euclidean space $\R_\Theta$ with parameter $\Theta=\left[\begin{array}{cc}\theta & -I\\ I &\theta' \end{array}\right]$. 
In general $x_j$'s and $\xi_k$'s do not admit a canonical representation on $L_2(\R_\theta)$ because $\Theta$ can be singular. Hence we consider the $\Psi$DOs as operators (densely) defined on $L_2(\R_\Theta)\cong L_2(\R_\theta)\ten_2 L_2(\R_{\theta'})$ affiliated to $\R_\Theta$ . Here $\ten_2$ is the Hilbert space tensor product.




\subsection{Abstract symbols }
In the classical case for $\R^d$, a symbol of order $m$ is a smooth bi-variable function $a\in C^\infty(\R^d\times \R^d)$ such that the
\begin{align}|D_x^\al D_\xi^\beta (a)(\bx,\bxi)|\le C_{\al,\beta}(1+|\bxi|^2)^{(m-|\beta|)/2}\pl.\label{om}\end{align}
In our setting, the symbols are operators affiliated to the von Neumann algebra tensor product $\R_\theta\overline{\ten}\R_{\theta'}$. Let us denote $\R_{\theta,\theta'}:=\R_\theta\overline{\ten}\R_{\theta'}$, $\M_{\theta,\theta'}$ for the multiplier algebra of $\R_{\theta,\theta'}$ and $\S_{\theta,\theta'}$ for the Schwartz class. $\R_{\theta,\theta'}$ is a $2d$-dimensional quantum Euclidean space with parameter matrix $\left[\begin{array}{cc}\theta & 0\\ 0 &\theta' \end{array}\right]$, in which $x$ and $\xi$ variables are mutually commuting, i.e. $[x_j,\xi_k]=0$ for all $j,k$. We specify the canonical partial derivatives for $x$ variables by $D_{x_1}, \cdots, D_{x_d}$ and for $\xi$ variables by $D_{\xi_1}, \cdots, D_{\xi_d}$. That is, for $a\in \M_{\theta,\theta'}$
\[D_{x_j}(a)=[D_{j}\ten 1,a]\pl , \pl D_{\xi_j}(a)=[1\ten D_j,a]\pl.\]
We index the transference action by the position: $\al_\by\ten \al_{\bet}(a)=\al^1_{\bet}\al^2_\by(a)$.
We use the standard multi-derivative notation that for $\al=(\al_1,\al_2,\cdots, \al_d)\in \mathbb{N}^d$ , \[D_x^\al(a)=D_{x_1}^{\al_1}D_{x_2}^{\al_2}\cdots D_{x_n}^{\al_d}\pl, \pl D_\xi^\al(a)=D_{\xi_1}^{\al_1}D_{\xi_2}^{\al_2}\cdots D_{\xi_d}^{\al_d} (a) \pl.\]
Write $\bc{\xi}:=(1+\sum_{j}\xi_j^2)^{\frac{1}{2}}$ where $\xi_j$'s are the non-commuting generators for $\R_{\theta'}$. We start with the abstract reformulation of the definition \eqref{om}.
\begin{defi}\label{abstract}For a real number $m$, define $\Sigma^m$ as the set of all operators $a \in \M_{\theta,\theta'}$ such that for all $\al,\beta$,
\[D_x^\al D_\xi^\beta (a)\bc{\xi}^{|\beta|-r}\]
extends to be a bounded operator in $\R_{\theta,\theta'}$. We call $\Sigma^m$ the space of symbols of order $m$ and write $\Sigma^{-\infty}=\cap_{m}\Sigma^m, \Sigma^{\infty}=\cup_{m}\Sigma^m$.
\end{defi}
Apriori it is not clear that the above definition satisfy the properties that $\Sigma^m\cdot \Sigma^n=\Sigma^{m+n}$ and $(\Sigma^m)^*=\Sigma^m$. To resolve it, we use the asymptotic degree discussed in Section \ref{asymptotic}.
\begin{defi}Given two real numbers $s$ and $r$,
we say an operator $a \in \M_{\theta, \theta'}$ is of bi-degree $(s,r)$ if for all $s', r'\in \R$
\[\bc{x}^{s'} \bc{\xi}^{r'} a \bc{x}^{-s'-s} \bc{\xi}^{-r'-r}\]
extends to a bounded element in $\R_{\theta,\theta'}$. We denote $O^{s,r}$ the set of all elements of bi-degree $(s,r)$ and write $O^{-\infty,r}= \cap_{s\in \R} O^{s,r}, O^{-\infty,-\infty}= \cap_{s,r\in \R} O^{s,r}$.
\end{defi}
Note that in $\R_{\theta,\theta'}$, $\bc{x}$ and $\bc{\xi}$ commute so the order of the product $ \bc{x}^{s} \bc{\xi}^{r}$ does not matter. The ``bi-degree'' gives a characterization of abstract symbols.
\begin{theorem}
\label{char}Let $m$ be a real number and $a\in \M_{\theta,\theta'}$. Then $a \in \Sigma^m$ if and only if for all $\al, \beta$, \[D^\al_xD^\beta_\xi(a)\in O^{0,m-|\beta|}\pl. \]
\end{theorem}
\begin{proof} The sufficiency is clear by the definition. Let $a\in \Sigma^m$. It follows from the Lemma \ref{indu} that for all $\al, \beta$, $D^\al_xD^\beta_\xi(a)$ is of degree $0$ for $x$ and degree ${m-|\beta|}$ for $\xi$. Because $\bc{x}$ and $\bc{\xi}$ commute, we have $D^\al_xD^\beta_\xi(a)\in O^{0,m-|\beta|}$.
\end{proof}

\begin{prop}\label{taylor}
$\Sigma^m$ equipped with the seminorms $\norm{\cdot}{\al,\beta}:=\norm{D_x^\al D_\xi^\beta (\cdot)\bc{\xi}^{|\beta|-m}}{}$ is a Frechet spaces. In particular, for $a\in \Sigma^m$, $D_{x_j}(a)$ and $D_{\xi_j}(a)$ are the vector derivatives
\[ D_{x_j}(a)=i\lim_{h\to 0} \frac{1}{h}(\al^1_{h\be_j}(a)-a)\pl, D_{\xi_j}(a)=i\lim_{h\to 0} \frac{1}{h}(\al^2_{h\be_j}(a)-a)\pl, \]
where the limit converges in the $\Sigma^m$.
\end{prop}
\begin{proof}
Let $a_n\in \Sigma^m$ be a converging sequence in $\Sigma^m$ with respect to all the seminorms $\norm{\cdot}{\al,\beta}$. Then there exists $b_{\al,\beta}\in \R_{\theta,\theta'}$ such that
\[\norm{D_x^\al D_\xi^\beta (a_n)\bc{\xi}^{|\beta|-m}-b_{\al,\beta}}{\infty}\to 0\pl \pl \text{as} \pl \pl n\to \infty\pl.\]
Denote that $c_{\al,\beta}=b_{\al,\beta}\bc{\xi}^{m-|\beta|}$ and $C_{0,0}=b_{0,0}\bc{\xi}^{m}$. Let $\la_{\theta,\theta'}(f)\in \S_{\theta,\theta'}$.
\begin{align*}\lan c_{\al,\beta}, \bc{\xi}^{|\beta|-m}\la_{\theta,\theta'}(f)\ran&=\lan b_{\al,\beta} \bc{\xi}^{|\beta|-m},\la_{\theta,\theta'}(f) \ran=\lan b_{\al,\beta},\la_{\theta,\theta'}(f)\ran
\\&=\lim_{n\to\infty}\lan D_x^\al D_\xi^\beta (a_n)\bc{\xi}^{|\beta|-m},\la_{\theta,\theta'}(f) \ran\\&= \lim_{n\to\infty}\lan a_n\bc{\xi}^{-m},\bc{\xi}^{m}D_x^\al D_\xi^\beta (\bc{\xi}^{|\beta|-m}(\la_{\theta,\theta'}(f))\ran \\&=\lan b_{0,0},\bc{\xi}^{m}D_x^\al D_\xi^\beta (\bc{\xi}^{|\beta|-m}\la_{\theta,\theta'}(f) )\ran \\&=\lan D_x^\al D_\xi^\beta (c_{0,0}),\bc{\xi}^{|\beta|-m}\la_{\theta,\theta'}(f) \ran \pl.\end{align*}
Note that the set $\bc{\xi}^{|\beta|-m}\S_{\theta,\theta'}=\S_{\theta,\theta'}$ by Proposition \ref{indu}. We have $c_{\al,\beta}=D_x^\al D_\xi^\beta (c_{0,0})$ weakly. To see that $c_{0,0}$ is again in the multiplier algebra $\M_{\theta,\theta'}$, it suffices to show that for any $\la_{\theta,\theta'}(f)\in \S_{\theta,\theta'}$,
\[\norm{ D_x^\al D_\xi^\beta (c_{0,0}\la_{\theta,\theta'}(f))(1+\sum_j x_j^2+\xi_j)^\gamma}{}  \]
is bounded for any $\al,\beta,\gamma$. This follows from Leibniz rule and the fact  $\la_{\theta,\theta'}(f)$ and all its derivatives $D_x^\al D_\xi^\beta(\la_{\theta,\theta'}(f))$ are in $O^{-\infty,-\infty}$. The vector derivatives are consequence of applying Proposition \ref{strongd} to $\R_{\theta,\theta'}$.
\end{proof}
\begin{cor}For all multi-indices $\al$ and real numbers $m,n$,
\begin{enumerate}
\item[i)]$\xi^\al \in \Sigma^{|\al|}$, $\bc{\xi}^m \in \Sigma^{m}$;
\item[ii)]if $a\in \Sigma^m$, then $a^*\in \Sigma^m$;
\item[iii)]if $a\in \Sigma^m, b\in \Sigma^n$, then $ab\in \Sigma^{m+n}$.
\end{enumerate}
\end{cor}
\begin{proof}i) is a direct consequence of Theorem \ref{order}. ii) follows from the fact that \[D_x^\al D_\xi^\beta (a^*)=(-1)^{|\al|+|\beta|}\Big(D_x^\al D_\xi^\beta (a)\Big)^*\pl.\] For iii), by the Leibniz rule
\begin{align}
\label{sum}
  D_x^\al D_\xi^\beta (ab)=\sum_{\al_1+\al_2= \al,\pl \beta_1+\beta_2= \beta}\binom{\al}{\al_1,\al_2}\binom{\beta}{\beta_1,\beta_2}D_x^{\al_1} D_\xi^{\beta_1} (a)D_x^{
\al_2} D_\xi^{\beta_2} (b)\pl.
\end{align}
Using Theorem \ref{char},
\[D_x^{\al_1} D_\xi^{\beta_1} (a)\in O^{0,m-|\beta_1|}\pl, \pl D_x^{
\al_2} D_\xi^{\beta_2} (b)\in O^{0,n-|\beta_2|}\pl.\]
Hence all summands in \eqref{sum} are belongs to $O^{0,m+n-|\beta_1|-|\beta_2|}=O^{0,m+n-|\beta|}$. Again by Theorem \ref{char}, $ab\in \Sigma^{n+m}$.
\end{proof}

\subsection{Comultiplications} One key tool that will be used in the proof of our symbol calculus is the the comultiplication maps of $\R_\theta$ and $\R_{\theta,\theta'}$. The comultiplication map of $\R^d$ as an abelian group is
\[\si: L_\infty(\R^d)\to L_\infty(\R^d\times \R^d)\cong L_\infty(\R^d\times \R^d)\pl, \pl \si(f)(\bx,\by)=f(\bx+\by)\pl.\]
Algebraically, $\si(u(\bxi))=u(\bxi)\ten u(\bxi)$ where $u(\bxi)$ is the unitary function $u(\bxi)(\bx)=e^{i\bxi\cdot \bx}$. For $\R_\theta$, we consider the a deformed comultiplication map \[\si_\theta:\R_\theta\to L_\infty(\R^n)\overline{\ten} \R_\theta\pl, \pl \si_\theta(\la_\theta(\bxi))=u(\bxi)\ten\la_\theta(\bxi)\pl,\]
where $\overline{\ten}$ is the von Neumann algebra tensor product. $L_\infty(\R^n)\overline{\ten} \R_\theta$ can be identified with $\R_\theta$-valued functions $L_\infty(\R^d,\R_\theta)$, and at a point $\bx\in \R^d$,
\[\si_\theta(\la_\theta(\bxi))(\bx)=e^{i\bx\cdot \bxi}\la_\theta(\bxi)=\al_\bx(\la_\theta(\bxi)) \pl. \]
A different co-multiplication map is used in \cite{GJP17} to studied $\Psi$DOs of $\R_\theta$ with commuting derivatives.

\begin{prop}\label{comulti}The map $\si_\theta: \S_\theta\to L_\infty(\R^d, \R_\theta)$
\begin{align*}     \si_\theta(\la_\theta(f))(\bx)=\al_\bx(\la_\theta(f))\pl, 
\end{align*}
\begin{enumerate}
\item[i)] extends to an injective normal $*$-homomorphism from $\R_\theta$ to $L_\infty(\R^d, \R_\theta)$.
\item[ii)] extends to an injective algebraic $*$-homomorphism from $\M_\theta$ to $L_\infty(\R^d, \M_\theta)$. Moreover, for all $a\in \M_\theta$, $
 \pl \si_\theta(D_j a)= D_{x_j}(\si_\theta(a))=D_{\bx_j}(\si_\theta(a))\pl.
$
\item[iii)] extends to an complete isometry $V_\theta$ right from $L_2(\R_\theta)^c$ to $L_2^c(\R^d)\ten_{wh} \R_\theta$.
    Here $\ten_{wh}$ denotes the $W^*$-Haagerup tensor product (see \cite{blechersmith92}) and $L_2^c(\R^d)$ is the column space.
\end{enumerate}
\end{prop}
\begin{proof}
i) follows from the fact that at each point $\bx\in \R^d$, $\al_\bx$ is a $*$-automorphism of $\R_\theta$. The normality was proved in \cite[Corollary 1.4]{GJP17}.  ii) is similar to i). For the derivatives,
let $D_{\bx_j}$ denote the $j$th partial derivatives for $\R_d$ and $D_{x_j}$ denote the partial derivatives on $\R_\theta$. For all $\bx\in\R^d$ and $a\in \M_\theta$,
\begin{align*}D_{\bx_j}(\si_\theta(a))(\bx)&=\lim_{h\to 0}-\frac{i}{h}\big(\al_{\bx+h\be_j}(a)-\al_{\bx}(a)\big)=D_{x_j}(\al_\bx(a))=\al_\bx(D_{x_j}a)\pl.
\end{align*}
For iii), let $b=\sum_{k}  b_{k}\la_\theta(f_{k})$ with $b_k\in \mathbb{C}$ and $\la_\theta(f_{k})$ being an orthonormal set in $L_2(\R_\theta)$. Then $\norm{b}{L_2(\R_\theta)}^2=\sum_{k}|b_k|^2$.
The norm of $L^c_2(\R^d)\ten_{wh} R_\theta$ is given by the $\R_\theta$-valued inner product that for $f,g\in L_2(\R^d)$ and $a,c\in \R_\theta$
\begin{align*} &\lan f\ten a, g\ten c\ran_{\R_\theta}=\lan f,g\ran_{L_2(\R^d)} a^*c\pl,\pl \norm{B}{L^c_2(\R^d)\ten_{wh} \R_\theta}=\norm{\lan B, B\ran_{\R_\theta}}{\R_\theta}
\end{align*}
Note that on the Fourier transform side, 
\begin{align*}V_\theta(\la_\theta(f))(\bxi)=\hat{f}(\bxi) \la_\theta(\bxi)\pl.
\end{align*}
Therefore,
\begin{align*}&\norm{V_\theta(\sum_{k} b_k\la_\theta(f_{k}))}{L^c_2(\R^d)\ten_{wh} \R_\theta}=\norm{\sum_{k,k'}b_k\bar{b}_{k'} \int \hat{f}_k(\bxi) \overline{\hat{f}}_{k'}(
\bxi) \la_\theta(\bxi)\la_\theta(\bxi)^* d \bxi }{\R_\theta}
\\= &\norm{(\sum_{k} |b_k|^2)1 }{\R_\theta}= \sum_{k} |b_k|^2\pl. \end{align*}
Replacing $b_k\in \mathbb{C}$ with matrices $b_k\in M_n$ in the above argument gives the complete isometry.
\end{proof}

Let us write $\la_{\theta,\theta'}(\bet,\by):=\la_\theta(\bet)\ten \la_{\theta'}(\by)$ for the generators of $\R_{\theta,\theta'}:=\R_\theta\overline{\ten}\R_{\theta'}$. The quantization map for $\R_{\theta,\theta'}$ is
\[\la_{\theta,\theta'}(F)=(2\pi)^{-2d}\int_{\R^{2d}}\hat{F}(\bet,\by)\la_{\theta,\theta'}(\bet,\by)d\eta d\by\pl,\]
where $\displaystyle \hat{F}(\bet,\by)=\int_{\R^{2d}} F(\bx,\bxi) e^{-i(\bx\bet+\bxi\by)}d\bx d\bxi$ is the Fourier transform.
 By the Proposition \ref{comulti}, we can dilate the symbols affiliated to $\R_{\theta,\theta'}$ to operator valued symbols, \begin{align*}\si_\theta\ten \si_{\theta'}: \R_{\theta,\theta'}\to L_\infty (\R^d\times \R^d, \R_\theta\overline{\ten}\R_{\theta'})\pl, \pl \la_{\theta,\theta'}(F)(\bx,\by)=\al_\bx^1\al_\by^2(\la_{\theta,\theta'}(F)), \end{align*}
 where $\al^1$ (resp. $\al^2$) is the transference action on $\R_\theta$ (resp. $\R_{\theta'}$). For the $\Psi$DOs, we consider the comultiplication maps for  $\R_\Theta$ with $\Theta=\left[\begin{array}{cc}\theta & -I_n\\ I_n & \theta'\end{array}\right]$.
We use the following quantization for $\R_\Theta$,
\[\la_\Theta(F)=(2\pi)^{-2d}\int_{\R^d}\int_{\R^d}\hat{F}(\bet,\by)\la_\theta(\bet)\la_{\theta'}(\by)d\eta d\by\pl,\pl F\in \S(\R^d\times \R^d)\pl.\]
Note that the unitary generators in $\R_{\Theta}$ satisfy the commutation relation \[\la_\theta(\bet)\la_{\theta'}(\by)=e^{i\bet\by}\la_{\theta'}(
\by)\la_{\theta}(\bet)\pl.\]
%
%
We have the Hilbert space isometry between two quantizations,
\[W:L_2(\R_\Theta)\to L_2(\R_{\theta,\theta'})\pl, \pl W\ket{\la_\Theta(F)}=\ket{\la_{\theta,\theta'}(F)}\pl.\]
Here and in the following, we will use the ``ket'' notation $|\cdot\ran$ to emphasis $L_2$ vector.
 \begin{prop} \label{cothe1}Define the unitary \[u_\theta(\by):L_2(\R_\theta)\to L_2(\R_\theta)\pl, v_\theta(\by)\ket{\la_\theta(f)}=\ket{\la_\theta(\al_yf)}\pl.\] The map $\si_\Theta:\S_\Theta \to B(L_2(\R_\theta))\overline{\ten }\R_{\theta'}$
\begin{align*}\la_\Theta(F) \mapsto (2\pi)^{-2d}\int_{\R^{2d}} \hat{F}(\bet,\by)\la_\theta(\bet)v_\theta(\by)\ten\la_{\theta'}(\by) d\bet d\by\pl
\end{align*}
\begin{enumerate}
\item[i)] satisfies that $\si_\Theta(\la_{\Theta}(F))=W\la_\Theta(F)W^*$ by viewing \[\S_\Theta \subset B(L_2(\R_\Theta))\pl, \pl B(L_2(\R_\theta))\overline{\ten }\R_{\theta'}\subset B(L_2(\R_\theta)\ten_2L_2(\R_{\theta'}))\pl.\]
\item[ii)] extends to an injective normal $*$-homomorphism from $\R_\Theta$ to  $B(L_2(\R_\theta))\overline{\ten} \R_{\theta'}$.
\end{enumerate}
\end{prop}
\begin{proof} By linearity, it suffices to verify that $W\la_\theta(\bet_0)\la_{\theta'}(\by_0)W^*= \la_\theta(\bet_0) v_\theta(\by_0)\ten \la_{\theta'}(\by_0)$. Indeed, for $\la_{\theta,\theta'}(G)\in \S_{\theta,\theta'}$,
\begin{align*}
W\la_\theta(\bet_0)\la_{\theta'}(\by_0)W^*\ket{\la_{\theta,\theta'}(G)}=& W\la_\theta(\bet_0)\la_{\theta'}(\by_0)\ket{\la_{\Theta}(G)}=W\ket{\la_{\Theta}(G_1)}\end{align*}
where
\begin{align*}\la_\Theta(G_1)=&\int_{\R^{2d}}\hat{G}(\bet,\by) \la_\theta(\bet_0)\la_{\theta'}(\by_0) \la_\theta(\bet)\la_{\theta'}(\by)d\by d\bet\\
=& \int_{\R^{2d}} \hat{G}(\bet-\bet_0,\by-\by_0)e^{i\bet\by_0}e^{\frac{i}{2}(\bet\theta\bet_0+\by\theta'\by_0)} \la_\theta(\bet)\la_{\theta'}(\by) d\by d\bet\pl.
 \end{align*}
 Then
 \begin{align*}
 &W\ket{\la_\Theta(G_1)}=\ket{\la_{\theta,\theta'}(G_1)}=\Big(\la_\theta(\bet_0) v_\theta(\by_0)\ten \la_{\theta'}(\by_0)\Big)\ket{\la_{\theta,\theta'}(G)}\pl. \qedhere
\end{align*}
\end{proof}
Now let us consider the GNS-construction of $B(L_2(\R_\theta))$ with respect to its standard trace. Define for a Schwartz function $F$ the operator
\[ T_F=(2\pi)^{-2d}\int_{\R^{2d}} \hat{F}(\bet,\by)\la_\theta(\bet)v_\theta(\by)d\bet d\by\pl\pl.\]
For $\ket{\la_\theta(f)}\in L_2(\R_\theta)$,
\begin{align*}T_F\ket{\la_\theta(f)}= (2\pi)^{-2d}\int \hat{F}(\bet,\by)\la_\theta(\bet)v_\theta(\by) d\bet d\by \ket{\la_\theta(f)}= :\ket{\la_\theta(g)}
\end{align*}
where $T_F$  has the following kernel representation,
\[ \widehat{g}(\bet)= (2\pi)^{-2d}\int \hat{F}(\bet-\bxi,\by)e^{i\by\bet}e^{\frac{i}{2}\bet\theta\bxi} d\by  \hat{f}(\bxi)d\bxi \pl.\]
Since $F\in \S(\R^{2d})$, $T_F$ is trace class and \[tr(T_F)= (2\pi)^{-2d}\int \hat{F}(0,\by)e^{i\by\bet} d\by d\bet = (2\pi)^{-d}\int F\pl. \] One calculates that
\begin{align*}   T_F^*T_F=(2\pi)^{-4d} &\int_{\R^{2d}}\Big(\int_{\R^{2d}} \overline{\hat{F}}(\bet_1,\by_1)\hat{F}(\bet+\bet_1,\by+\by_1) e^{-\frac{i}{2}\bet\theta\bet_1} e^{-i\bet_1\by} d\bet_1 d\by_1\Big) \la_\theta(\bet)v_\theta(\by) d\bet d\by
     \end{align*}
Hence $tr(T_F^*T_F)=\displaystyle (2\pi)^{-2d}\int_{\R^{2d}} \overline{\hat{F}}(\bet_1,\by_1)\hat{F}(\bet_1,\by_1)d\bet_1 d\by_1= (2\pi)^{-2d}\norm{F}{2}^2$\pl. Up to a scalar we have a Hilbert space isometry \[V: L_2(B(L_2(\R_\theta)),tr)\to L_2(\R^d,L_2(\R_\theta))\pl ,
V(T_F)(\bx)= \la_\theta(F(\bx,\cdot)) \pl.\] Write $\tilde{\pi} $
as the GNS construction of $B(L_2(\R_\theta))$ on $L_2(B(L_2(\R_\theta)),tr)$. Then $\pi(\cdot)=V\tilde{\pi}(\cdot)V^*$ gives a normal faithful $*$-homomorphism form $B(L_2(\R_\theta))$ to $B(L_2(\R^d))\overline{\ten}\R_\theta$ as follow,
\[ \pi(T_F):=V\tilde{\pi}(T_F)V^*= (2\pi)^{-2d}\int_{\R^{2d}} \hat{F}(\bet,\by)v(\bet)u(\by)\ten \la_\theta(\bet)d\bet d\by \in B(L_2(\R^d))\overline{\ten}\R_\theta\pl,\]
where $v(\bet)$ is translation unitary on $L_2(\R^d)$. Combining $\pi$ with the co-multiplication $\si_\Theta$, we obtain another co-multiplication of $\R_\Theta$.
\begin{prop} \label{48}The map $\tilde{\si}_\Theta: \S_\Theta \to  B(L_2(\R^d))\overline{\ten}\R_{\theta,\theta'} $
\begin{align*} \la_\Theta(F)\longmapsto (2\pi)^{-2d}\int \hat{F}(\bet,\by)\Big(u(\bet)v(\by)\ten \la_{\theta,\theta'}(\bet,\by)\Big) d\bet d\by\pl
\end{align*}
 \begin{enumerate}\item[i)]extends to a normal injective $*$-homomorphism from $\R_\Theta$ to $B(L_2(\R^d))\overline{\ten}\R_{\theta,\theta'}$.
 \item[ii)] satisfies the intertwining relation  $(V_\theta\ten id_{\R_{\theta'}})\tilde{\si}_\Theta(\cdot)=\si_\Theta(\cdot) (V_\theta\ten id_{\R_{\theta'}})$ for the isometry \[V_\theta\ten id_{\R_{\theta'}}: L_2^c(\R_\theta)\ten_{wh}\R_{\theta'}\to L_2^c(\R^d)\ten_{wh}(\R_{\theta}\overline{\ten}\R_{\theta'})\pl.\]
 \end{enumerate}
\end{prop}
\begin{proof}i) We verify that $\tilde{\si}_\Theta=(\pi\ten id_{\R_{\theta'}})\circ\si_\Theta$. Indeed
\begin{align*}
(\pi\ten id_{\R_{\theta'}})\circ&\si_\Theta(\la_\Theta(F))
= \pi\ten id_{\R_{\theta'}}\Big(  (2\pi)^{-2d}\int_{\R^{2d}} \hat{F}(\bet,\by)\la_\theta(\bet)v_\theta(\by)\ten\la_{\theta'}(\by) d\bet d\by \Big)
\\=& (2\pi)^{-2d}\int \hat{F}(\bet,\by)\Big(u(\bet)v(\by)\ten \la_\theta(\bet)\ten  \la_{\theta'}(\by)\Big) d\bet d\by= \tilde{\si}_\Theta(\la_\Theta(F)). 
\end{align*}
For ii),
recall that $B(L_2(\R_\theta))\overline{\ten}\R_{\theta'}$ is canonically isomorphic to
the adjointable $\R_\theta'$-module map $\mathcal{L}(L_2^c(\R_\theta)\ten_{wh}\R_{\theta'})$ and similarly $B(L_2(\R^d))\overline{\ten}\R_{\theta}\overline{\ten}\R_{\theta'}\cong \mathcal{L}(L_2^c(\R_\theta)\ten_{wh}\R_{\theta,\theta'})$ as $\R_{\theta,\theta'}$-module map (see \cite{lance95}). The complete isometry $V_\theta$ in Proposition \ref{comulti} give an isometry \[V_\theta\ten id_{\theta'}: L_2^c(\R_\theta)\ten_{wh}\R_{\theta'}\to L_2^c(\R^d)\ten_{wh}(\R_{\theta}\overline{\ten}\R_{\theta'})\pl.\]
We verify that the intertwining relation  $(V_\theta\ten id)\si_\Theta(\cdot)=\tilde{\si}_\Theta(\cdot) (V_\theta\ten id)$. For any $\la_\Theta(F)\in \S_\Theta$ and $\la_{\theta,\theta'}(G)\in \S_{\theta,\theta'}$, we have
$\si_\Theta(\la_\Theta(F))\ket{\la_{\theta,\theta'}(G)}=\ket{\la_{\theta,\theta'}(G_1)}$ where
\begin{align*}
\hat{G}_1(\bet,\by)=(2\pi)^{-2d}\int\hat{F}(\bet-\bet_1,\by-\by_1) \hat{G}(\bet_1,\by_1)e^{i\bet_1(\by-\by_1)}e^{\frac{i}{2}\bet \theta \bet_1}e^{\frac{i}{2}\by \theta \by_1}  d\bet_1 d\by_1
\end{align*}
On the other hand, one verifies that
\begin{align*}
&\tilde{\si}_\Theta\ten id(\la_\Theta(F)) V_\theta \ket{\la_{\theta,\theta'}(G)}=\ket{\int \hat{G}_1(\bet,\by)u(\bet)\ten  \la_{\theta,\theta'}(\bet,\by)d\bet d\by}\\
=&V_\theta \ten id \Big(\si_\Theta(\la_\Theta(F))\ket{\la_{\theta,\theta'}(G)}\Big)
\end{align*}
We see that the representation $(V_\theta\ten id)^*\si_\Theta(\cdot)(V_\theta\ten id)$ is a restriction of $\tilde{\si}_\Theta$.
\end{proof}
\subsection{Pseudo-differential operator calculus}Recall that on $\R^d$ the pseudo-differential operator of a symbol $a(\bx,\bxi)$ is given by the singular integral form
\begin{align}\label{cop} op_0(a)(f)(\bx):=\frac{1}{(2\pi)^d}\int_{\R^d} e^{i\bx\cdot \xi}a(\bx,\bxi)\hat{f}(\bxi)d\bxi\pl, f\in \S(\R^d) \end{align}
In \cite{GJP17} the $\Psi$DOs on $\R_\theta$ are defined as  \begin{align}
\label{qop} op_\theta(a)(\la_\theta(f))=\frac{1}{(2\pi)^d}\int_{\R^d} a(\bxi)\la_\theta(\bxi)\hat{f}(\bxi)d\bxi\pl, f\in \S(\R^d)\pl. \end{align}
where $a:\R^d\to \R_\theta$ is the symbol as a $\R_\theta$-valued function. The $\Psi$DOs in our setting are operators densely defined on $L_2(\R_{\theta,\theta'})\cong L_2(\R_\theta)\ten_2 L_2(\R_{\theta'})$. For a symbol $a_1\ten a_2$ with $a_1\in \R_\theta, a_2\in \R_{\theta'}$, we define that
\[ Op(a_1\ten a_2)=\si_\Theta(a_1a_2)\in B(L_2(\R_{\theta,\theta'}))\]
where $a_1a_2$ is the product in $\R_\Theta$ by viewing $\R_\theta,\R_\theta'\subset\R_\Theta$ as subalgebras and $\si_\Theta$ is the representation of $\R_\Theta$ on $L_2(\R_{\theta,\theta'})$ defined in Proposition \ref{cothe1}.
\begin{defi}\label{defin}
For a symbol $a\in \Sigma^m$, we define the operator $Op(a):\S_{\theta,\theta'} \to \S_{\theta,\theta'}$ as follows,
\begin{align*}
 &Op(a)\la_{\theta,\theta'}(F)=\frac{1}{(2\pi)^{2d}}\int_{\R^{2d}}  \al_{\bet}^2(a) \hat{F}(\bet,\by)\la_{\theta,\theta'}(\bet,\by)d\bet d\by
 \end{align*}
We denote by $op^m$ the set of all $\Psi$DOs of order $m$.\label{def}
\end{defi}
We justifies the above definition below.
\begin{prop}\label{psuedo}For a symbol $a\in \Sigma^m$, $Op(a)$ is a continuous map form $\S_{\theta,\theta'}$ to $\S_{\theta,\theta'}$ and $Op(a)$ is an operator affiliated to $\si_\Theta(\R_\Theta)\subset B(L_2(\R_{\theta,\theta'}))$. In particular, if $a_1\in \R_\theta$ and $a_2\in\R_{\theta'}$,  $Op(a_1\ten a_2)=\si_\Theta(a_1a_2)$.
\end{prop}
\begin{proof}In the calculation below, the normalization constant $(2\pi)^{-d}$ will be omitted. Recall from Proposition \ref{cothe1} that
\[W:L_2(\R_\Theta)\to L_2(\R_{\theta,\theta'})\pl, \pl W\ket{\la_\Theta(F)}=\ket{\la_{\theta,\theta'}(F)}\pl,\]
  is the isometry such that $W^*\si_\Theta(\cdot )W$ is the left regular representation of $\R_\Theta$ on $L_2(\R_\Theta)$. To verify that $Op(a)$ is affiliated to $\si_\Theta(\R_\Theta)$, it suffices to show that $WOp(a)W^*$ commutes with right multiplication of $\R_\Theta$. For any $\bet_0,\by_0\in\R^d$,
\begin{align*}
\la_\Theta(F)\la_\theta(\bet_0)\la_{\theta'}(\by_0)=& \Big(\int_{\R^{2d}}  \hat{F}(\bet,\by)\la_\theta(\bet)\la_{\theta'}(\by)d\bet d\by\Big)\la_\theta(\bet_0)\la_\theta'(\by_0)\\
=& \int_{\R^{2d}}  \hat{F}(\bet,\by)e^{i\by\bet_0}\la_\theta(\bet)\la_\theta(\bet_0)\la_{\theta'}(\by)\la_{\theta'}(\by_0)d\bet d\by\pl.
\end{align*}
Then $W(\la_\Theta(F)\la_\theta(\bet_0)\la_{\theta'}(\by_0))=\al_{\bet_0}^2\big(\la_{\theta,\theta'}(F)\big)\la_{\theta,\theta'}(\bet_0,\by_0)$.
We verify that
\begin{align*}
&Op(a)W\Big(\la_\Theta(F)\la_\theta(\bet_0)\la_{\theta'}(\by_0)\Big)
\\=& Op(a)\Big(\al_{\bet_0}^2\big(\la_{\theta,\theta'}(F)\big)\la_{\theta,\theta'}(\bet_0,\by_0)\Big)
\\=& \int_{\R^{2d}}  \al_{\bet+\bet_0}^2(a) \hat{F}(\bet,\by)e^{i\by\bet_0}e^{\frac{i}{2}\bet\theta\bet_0}e^{\frac{i}{2}\by\theta'\by_0}\la_{\theta,\theta'}(\bet+\bet_0,\by+\by_0)d\bet d\by
\\=& \Big(\int_{\R^{2d}}  \al_{\bet+\bet_0}^2(a) \hat{F}(\bet,\by)\al_{\bet_0}^2(\la_{\theta,\theta'}(\bet,\by))d\bet d\by \Big )\la_{\theta,\theta'}(\bet_0,\by_0)
\\=& \al_{\bet_0}^2\Big(\int_{\R^{2d}}  \al_{\bet}^2(a) \hat{F}(\bet,\by)\la_{\theta,\theta'}(\bet,\by)d\bet d\by \Big ) \la_{\theta,\theta'}(\bet_0,\by_0)
\\=& \al_{\bet_0}^2\Big(Op(a)\la_{\theta,\theta'}(F) \Big ) \la_{\theta,\theta'}(\bet_0,\by_0)\pl.
\end{align*}
Hence \[W^*Op(a)W\Big(\la_\Theta(F)\la_\theta(\bet_0)\la_{\theta'}(\by_0)\Big)=\Big(W^*Op(a)W\la_\Theta(F)\Big)\la_\theta(\bet_0)\la_{\theta'}(\by_0)\pl,\]
which implies $Op(a)$ is affiliated to the representation on $\si(\R_\Theta) \subset B(L_2(\R_\theta)\ten_2 L_2(\R_{\theta'}))$.

Now we show that $Op(a):\S_{\theta,\theta'}\to \S_{\theta,\theta'}$ is continuous. Let us first assume that $a\in \Sigma^0$ is a zero order symbol. Then $a$ is bounded in $ \R_{\theta,\theta'}$ and $\norm{a}{\infty}=\norm{\al_{\bet}^2(a)}{\infty}$ for all $\bet$. Thus the singular integral
\[ \norm{\int_{\R^{2d}}  \al_{\bet}^2(a) \hat{F}(\bet,\by)\la_{\theta,\theta'}(\bet,\by)d\bet d\by}{\infty}\le \norm{\hat{F}}{1}\norm{a}{\infty}\]
converges in $\R_{\theta,\theta'}$. Write the set
$\Omega:=\{Op(a)\la_\Theta(F) \pl | F\in \S(\R^{2d})\pl, a\in \Sigma^0\} \subset \R_{\theta,\theta'}$. For derivatives, we know
$D_{x_j}(\la_\theta(\bet))= \bet_j\la_\theta(\bet) \pl,\pl D_{\xi_j}(\la_{\theta'}(\by))=\by_j\la_{\theta'}(\by)$
and $D_x^\beta D_\xi^\gamma(a) \in \Sigma^{-|\gamma|}$. Using product rules in the integral,
\begin{align*}&D_{\xi_j}\Big(Op(a)\la_{\theta,\theta'}(F)\Big)
\\=&D_{\xi_j}\Big(\int_{\R^{2d}}  \al_{\bet}^2(a) \hat{F}(\bet,\by)\la_\theta(\bet)\ten \la_{\theta'}(\by)d\bet d\by \Big)
\\=&\int_{\R^{2d}}  \al_{\bet}^2(D_{\xi_j}a) \hat{F}(\bet,\by)\la_{\theta,\theta'}(\bet,\by)d\bet d\by
+\int_{\R^{2d}}  \al_{\bet}^2(a) \hat{F}(\bet,\by)\by_j\la_{\theta,\theta'}(\bet,\by)d\bet d\by
\\=&Op(D_{\xi_j}a)\la_{\theta,\theta'}(F)+Op(a)\la_{\theta,\theta'}(D_{\bxi_j}F),
\end{align*}
which is again in the set $\Omega$ hence bounded in $\R_{\theta,\theta'}$. By induction, $D_x^\beta D_\xi^\gamma (Op(a)\la_{\theta,\theta'}(F))$ is in $\Omega$ for any $\beta,\gamma$. On the other hand, let $h\in\R$ and $\be_j=(0,\cdots,1,\cdots,0)$
\begin{align*} \la_\theta(\bet)e^{ix_jh}=e^{-\frac{i}{2}\sum_k h\theta_{jk}\bet_k}\la_\theta(\bet +h\be_j)\pl, \la_{\theta'}(\by)e^{i\xi_jh}=e^{-\frac{i}{2}\sum_k h\theta'_{jk}\by_k}\la_\theta(\by +h\be_j)\pl.
\end{align*}
Taking derivatives at $h=0$,
\begin{align*}
\la_\theta(\bet)x_j=D_{\bet_j}(\la_\theta(\bet))-\frac{1}{2}\sum_k \theta_{jk}\bet_k \la_\theta(\bet)\pl, \la_\theta(\by)\xi_j=D_{\by_j}(\la_{\theta'}(\by))-\frac{1}{2}\sum_k \theta'_{jk}\bet_k \la_{\theta'}(\by)\pl.
\end{align*}
holds weakly. Then
\begin{align*} &\Big(Op(a)\la_{\theta,\theta'}(F)\Big)x_j\\
=& \int  \al_{\bet}^2(a) \hat{F}(\bet,\by)D_{\bet_j}(\la_{\theta,\theta'}(\bet,\by))d\bet d\by
- \frac{1}{2}\int  \al_{\bet}^2(a) \hat{F}(\bet,\by)(\sum_k \theta_{jk}\bet_k)\la_{\theta,\theta'}(\bet,\by)d\bet d\by\\
=& -\int  \al_{\bet}^2(D_{\xi_j}a) \hat{F}(\bet,\by)(\la_{\theta,\theta'}(\bet,\by))d\bet d\by
-\int  \al_{\bet}^2(a) (D_{\bet_j}\hat{F})(\bet,\by)(\la_{\theta,\theta'}(\bet,\by))d\bet d\by\\
&- \frac{1}{2}\int  \al_{\bet}^2(a) \hat{F}(\bet,\by)(\sum_k \theta_{jk}\bet_k)\la_{\theta,\theta'}(\bet,\by)d\bet d\by\\
=&-Op(D_{\xi_j}a)\la_{\theta,\theta'}(F) -Op(a)\la_{\theta,\theta'}(\bxi_jF) -\frac{1}{2}\sum_k \theta'_{jk}Op( a)\la_{\theta,\theta'}(D_{\bxi_k}F)
\end{align*}
which is again in the set $\Omega$. By induction, $\Omega$ is stable under right multiplication of polynomials $x^\beta\xi^\gamma$. By Proposition \ref{indu}, we know $\Omega\subset \S_{\theta,\theta'}$ because for all $\beta_1,\beta_2,\gamma_1,\gamma_2$ \[\norm{D_x^{\beta_1}D_\xi^{\gamma_1}(Op(a)\la_{\theta,\theta'}(F))x^{\beta_2}\xi^{\gamma_2}}{\infty}<\infty\pl.\]
Moreover, one can track that these norms are controlled by the semi-norms of $a\in\Sigma^0$ and $\la_{\theta,\theta'}(F)\in S_{\theta,\theta'}$. Thus we proved $Op(a):\S_{\theta,\theta'}\to \S_{\theta,\theta'}$ is continuous for $0$-order $\Psi$DO. Now consider $b\in \Sigma^m$ with $m$ being an even integer, we know $b=b\bc{\xi}^{-m} \bc{\xi}^{m}$ and $b\bc{\xi}^{-m}$ is a zero order symbol, $\bc{\xi}^{m}$ is a polynomial. Note that for $a\in \Sigma^0$,
\begin{align*}&Op(a\xi_j)\la_{\theta,\theta'}(F)
\\=&\int_{\R^{2d}}  \al_{\bet}^2(a\xi_j) \hat{F}(\bet,\by)\la_{\theta,\theta'}(\bet,\by)d\bet d\by \\
=& \int_{\R^{2d}}  (\xi_j+\bet_j)\al_{\bet}^2(a) \hat{F}(\bet,\by)\la_{\theta,\theta'}(\bet,\by)d\bet d\by\\
=& \int_{\R^{2d}}  \xi_j\al_{\bet}^2(a) \hat{F}(\bet,\by)\la_{\theta,\theta'}(\bet,\by)d\bet d\by+\int_{\R^{2d}}  \al_{\bet}^2(a) \hat{F}(\bet,\by)\bet_j\la_{\theta,\theta'}(\bet,\by)d\bet d\by
\\=&\xi_jOp(a)\la_{\theta,\theta'}(F)+Op(a)\la_{\theta,\theta'}(D_{\bx_j}F)
\end{align*}
which is again in $\Omega$. Moreover, the continuity of $Op(a\xi_j)$ follows from the continuity of $Op(a)$. By induction, we obtain that $Op(a):\S_{\theta,\theta'}\to \S_{\theta,\theta'}$ is continuous for $Op(a)\in\Sigma^m$ for all $m$.
Finally, we verify the property that $Op(a_1\ten a_2)=\si(a_1a_2)$. It suffices to consider test functions $\la_{\theta,\theta'}(F)=\la_\theta(f_1) \ten\la_{\theta'}(f_2)$ with $F(\bx,\bxi)=f_1(\bx)f_2(\bxi)$. Then
\begin{align*}
Op(a_1\ten a_2)\la_{\theta,\theta'}(F)=& \int  \big(a_1\ten \al_{\bet}(a_2)\big)   \hat{f_1}(\bet)\hat{f_2}(\by)  \big(\la_\theta(\bet) \ten \la_{\theta'}(\by)\big)     d\bet d \by\\
=& \int \hat{f_1}(\bet) a_1  \la_\theta(\bet) \ten  (\al_{\bet}(a_2)\la_{\theta'}(f_2))    d\bet \\
=& W^*(\int  \hat{f_1}(\bet) a_1 \la_\theta(\bet)  \al_{\bet}(a_2)\la_{\theta'}(f_2)    d\bet) \\
=& W^*(a_1 a_2 \int  \hat{f_1}(\bet) \la_\theta(\bet) \la_{\theta'}(f_2)    d\bet) \\
=& W^*(a_1 a_2 \la_\theta(f_1) \la_{\theta'}(f_2))= W^*(a_1 a_2)W \Big(\la_\theta(f_1)\ten \la_{\theta'}(f_2)\Big) \pl.
\end{align*}
Here we use the fact that for $a_2\in \M_{\theta'}$, $a_2\la_\theta(\bet)= \la_\theta(\bet)\al_{\bet}(a_2)$ . This property be easily verified for $a_2\in\S_{\theta'}$ and then extends to $\M_{\theta'}$.
\end{proof}
Based on the above proposition, we can equivalently consider $Op(a)$ are operators affiliated to $\R_\Theta$ and $Op(a)\in \R_\Theta$ if it is bounded. The connection between our setting and $\Psi$DOs on $\R^d$ and $\R_\theta$ can be made explicit via the following commuting diagram. \\
\begin{tikzpicture}
  \matrix (m) [matrix of math nodes,row sep=3em,column sep=5em,minimum width=2em]
  {
     \Sigma^0\subset \R_{\theta,\theta'}  & \R_\Theta & \\
     \R_{\theta}\bar{\ten}L_\infty(\R^d,\R_{\theta'})  &  B(L_2(\R_\theta))\bar{\ten}\R_{\theta'}  &     \mathcal{L}(L_2^c(\R_\theta)\ten_{wh} \R_{\theta'})           \\
     L_\infty(\R^d\times \R^d, \R_{\theta,\theta'} ) &
     \mathcal{L}(L_2^c(\R^d)\ten_{wh} \R_{\theta,\theta'} ) \\};
  \path[-stealth]
    (m-1-1) edge node [left] {$id\ten\si_{\theta'}$} (m-2-1)
            edge node [above] {$Op$} (m-1-2)
    (m-2-1) edge 
            node [above] {$ op_\theta\ten id_{\R_{\theta'}}$} (m-2-2)
            edge node [left] {$\si_{\theta}\ten id$} (m-3-1)
    (m-1-2) edge node [left] {$\si_\Theta$} (m-2-2)
    (m-2-2) edge node [above] { }(m-2-3)
    (m-3-2) edge node [below] {$\ \ \ \ V_\theta(\cdot)V_\theta^*$ }(m-2-3)
    (m-3-1.east|-m-3-2) edge 
            node [above] {$ op_0\ten id_{\R_{\theta,\theta'}}$} (m-3-2);
\end{tikzpicture}\\
Here $\si_\theta,\si_{\theta'},\si_\Theta$ are the co-multiplication maps discussed in section 3.2. The composition $\si_\Theta\circ Op$ gives the definition \ref{defin}. On the second row, the co-multiplication
$id\ten \si_{\theta'}(a)(\bet)=\al^2_{\bet}(a)$ gives $\R_{\theta'}$-valued symbol, and Definition \ref{defin} is then coincides with the $\R_{\theta'}$-valued operator map $op_\theta\ten id$ on $\R_\theta$ in \eqref{qop}. Via the identification
$B(L_2(\R_\theta))\overline{\ten} \R_{\theta'}\cong \mathcal{L}(L_2(\R_\theta)^c\ten_{wh} \R_{\theta'})$ (\cite{lance95}),
this also gives operators on Hilbert $\R_{\theta'}$-module $L_2(\R_\theta)^c\ten_{wh} \R_{\theta'}$. On the bottom row, we have a $\R_{\theta,\theta'}$-valued classical symbol $\si_\theta\ten \si_\theta'(a)(\bx,\bxi)=\al^1_{\bx}\al^2_{\bxi}(a)$, and
$op_0\ten id_{\theta,\theta'}$ is the $\R_{\theta,\theta'}$-valued operator map on $\R^d$ in \eqref{cop}. The $\Psi$DOs are $\R_{\theta,\theta'}$-linear operators on the Hilbert module $L_2(\R^d)^c\ten_{wh} \R_{\theta,\theta'}$. By Proposition \ref{48}, we have the Hilbert space isometry
\[ V_\theta \ten id_{\R_{\theta'}}: L_2(\R_\theta)^c\ten_{wh} \R_{\theta'}\to L_2(\R^d)^c\ten_{wh} \R_{\theta,\theta'}\pl.\]
Moreover, for a symbol $a\in \Sigma^0$, the operator $Op(a)$ can be viewed as a restriction of the $\R_{\theta,\theta'}$-valued $\Psi$DO $op_0\ten id(\si_{\theta,\theta'}(a))$ as follows,
\begin{align*}
&op_0\ten id\big(\si_\theta\ten \si_\theta'(a)\big) \Big(V_\theta \ten id(\la_{\theta,\theta'}(F))\Big)\\
=&(2\pi)^{-d}\int e^{i\bx\bxi}\al^1_{\bx}\al^2_{\bxi}(a) \hat{F}(\bxi,\by)\la_{\theta,\theta'}(\bxi,\by)d\bxi d\by\pl.\\
=&\al_{\bx}\Big((2\pi)^{-d}\int \al^2_{\bxi}(a) \hat{F}(\bxi,\by)\la_{\theta,\theta'}(\bxi,\by)d\bxi d\by \Big)=V_\theta\ten id(Op(a)\la_{\theta,\theta'}(F))
\end{align*}
This enable us to reduce the $L_2$-boundedness to the operator-valued case. For that we recall the operator-valued Calderon-Vallicourt theorem proved by Merklen in \cite{merklen}.
\begin{theorem}[Theorem 2.1 of \cite{merklen}]\label{ovcv}
 Let $\A$ be a $C^*$-algebra and $CB^\infty(\R^d\times\R^d\pl, \A)$ be the set of smooth $\A$-valued functions with bounded
derivatives of all orders. Then for any $a\in CB^\infty(\R^d\times\R^d\pl, \A)$,
\[op(a)f(\bx)=\frac{1}{(2\pi)^d}\int_{\R^d} e^{i\bx\cdot \xi}a(\bx,\bxi)\hat{f}(\bxi)d\bxi\pl, f\in \S(\R^d,\A)\pl\]
extends to a bounded operator on the Hilbert $\A$-module $L_2(\R^d,\A)$. Moreover, there exists a constant $C$ independent of $a$, such that
\[\norm{op(a)}{}\le C\sup \{\norm{D^\al_{\bx} D^\beta_{\bxi} (a)}{\infty} \pl | \pl 0\le \al,\beta\le (1,1,\cdots, 1)  \}\pl.\]
\end{theorem}
Then $L_2$-boundedness theorem in our setting follows from the commuting diagram.
\begin{theorem}[$L_2$-boundedness]\label{l2}Let $a\in \Sigma^0$ be a symbol of order $0$. Then $Op(a)$ extends to a bounded operator on $L_2(\R_{\theta,\theta'})$.
\end{theorem}
\begin{proof} By definition of $\Sigma^0$, $a$ and all its derivatives $D_x^\al D_\xi^\beta(a)$ are in $\R_{\theta,\theta'}$. Then $\si_\theta\ten \si_{\theta'}(a)\in L_\infty(\R^d\times \R^d, \R_{\theta,\theta'})$ and for any $\al,\beta$,
\[\norm{D_{\bx}^\al D_{\bxi}^\beta(\si_{\theta,\theta'}(a))}{}= \norm{\si_{\theta,\theta'}(D_x^\al D_\xi^\beta(a))}{}\]
are bounded. Thus $\si_{\theta,\theta'}(a)$ is a $\R_{\theta,\theta'}$-valued symbol with all derivatives bounded. Then by
Theorem \ref{ovcv}, we know $op_0\ten id (\si_{\theta,\theta'}(a))$ is a bounded element in $B(L_2(\R^d))\overline{\ten} \R_{\theta,\theta'}$. By diagram chasing, \[\norm{Op(a)}{}=\norm{V_\theta Op(a)V_\theta ^*}{B(L_2(\R_\theta))\overline{\ten}\R_{\theta'}}\le \norm{op \Big(\si_{\theta}\ten\si_\theta'(a)\Big)}{\mathcal{L}(L_2(\R^d, \R_{\theta,\theta'}))}\] and the norm estimates follows from Theorem \ref{ovcv}.
\end{proof}
We now discuss the composition formula. Let us first identify the formula by a heuristic argument. Given two classical operator valued symbol $a,b\in C^\infty(\R^d\times\R^d, \A)$, the composition symbol in the usual Euclidean case is
\[c(\bx,\bxi)=\frac{1}{(2\pi)^d}\int_{\R^{2d}} a(\bx,\bet)b(\by,\bxi)e^{ i(\bet-\bxi)\cdot (\bx-\by)}d \bet d\by .\]
Given symbols $a,b$ affiliated to $\R_{\theta,\theta'}$, the co-multiplication $\si_{\theta,\theta'}$ gives us operator-valued symbol
\[ \sigma_{\theta,\theta'}(a)(\bx,\bxi)=\al_{\bx}^1\al_{\bxi}^2(a)\pl, \sigma_{\theta,\theta'}(b)(\bx,\bxi)=\al_{\bx}^1\al_{\bxi}^2(b)\pl.\]
The operator-valued composition symbol is
\begin{align*}C(\bx,\bxi)=&\frac{1}{(2\pi)^d}\int_{\R^{2d}} \al_{\bx}^1\al_{\bet}^2(a)\al_{\by}^1\al_{\bxi}^2(b)e^{ i(\bet-\bxi)\cdot (\bx-\by)}d \bet d\by
\\ =&\al_{\bx}^1\al_{\bxi}^2\Big(\frac{1}{(2\pi)^d}\int_{\R^{2d}} \al_{\bet-\bxi}^2(a)\al_{\by-\bx}^1(b)e^{ i(\bet-\bxi)\cdot (\bx-\by)}d \bet d\by \Big)
\\ =&\al_{\bx}^1\al_{\bxi}^2\Big(\frac{1}{(2\pi)^d}\int_{\R^{2d}} \al_{\bet}^2(a)\al_{\by}^1(b)e^{ -i\bet\by}d \bet d\by \Big)=\si_{\theta,\theta'}(c)
\end{align*}
where $c$ is a $\M_{\theta,\theta'}$-valued singular integral,
\[c=\frac{1}{(2\pi)^d}\int_{\R^{2d}} \al^2_{\bet}(a)\al^1_{\by}(b)e^{- i \bet \cdot \by}d \bet d\by\pl .\]
We first justify this singular integral and prove its formal series of the following definition.
\begin{defi}
Let $m_j, j\ge 0$ be a decreasing sequence of real numbers and $a_j\in \Sigma^{m_j}$. We write a $m_0$ order symbol $a\sim \sum_{j\ge 0}a_j$ if for any $N$, $a-\sum_{N\le m_j}a_j\in \Sigma^{N}$.
\end{defi}
The proof adapts the argument for the classical case by Stein \cite{stein} to the operator-valued setting.
\begin{theorem}[Composition formula] \label{compo} Let $a\in \Sigma^m$ and $b\in \Sigma^n$. Then there exists a symbol $c\in \Sigma^{m+n}$ such that $Op(c)=Op(a)Op(b)$ and
\[c\sim \sum_{\al} \frac{ i^{|\al|}}{\al !}D_\xi^\al (a)D_x^\al(b)\pl. \]
\end{theorem}
\begin{proof} Let $\phi$ be a positive function on $\R^d$ such that $\phi(\bx)=1$ for $|\bx|\le 1$ and $\phi(\bx)=0$ for $|\bx|> 2$. Write
\[c=\lim_{\epsilon\to 0}\frac{1}{(2\pi)^d}\int \al^2_{\bet}(a)b_\epsilon(\by)e^{-i \bet \cdot \by}d \bet d\by\pl,\]
where for each $\epsilon$, $b_\epsilon(\by)=\phi(\epsilon\by) \al_\by^2(b)$ is compactly supported. This is a Bochner integral, because the integrand function
$(\bet,\by) \mapsto \al^2_{\bet}(a)b_\epsilon(\by)e^{-i \bet \cdot \by}$ is smooth in the Frechet space $\Sigma^{m+n}$ by Proposition \ref{taylor}. We first prove that the above integral converges in $\Sigma^{m+n}$ and admit the series expansion. For the compactly supported $b_\epsilon\in C(\R^d,\Sigma^n)$, the Fourier transform with value in the Frechet space $\Sigma^n$ is well-defined,
\[ \hat{b}_\epsilon(\bet)=\int b_\epsilon(\by)e^{-i\by\bet}d\by\pl.\]
Note that for any compactly supported $b$, $\displaystyle \int b(\by)e^{-i\bet\by}d\bet d\by=(2\pi)^d b(0)$. Then for any $\beta$,
\begin{align}\label{cal}\int \bet^\beta\hat{b}_\epsilon(\bet)d\bet &=(-1)^{|\beta|}\int b_\epsilon(\by)D^\beta_\by (e^{-i\by\bet})d\by d\bet\nonumber
=\int D^\beta_\by(\phi(\epsilon\by ) \al^1_\by(b)) e^{-i\by\bet}d\by d\bet\nonumber
\\ &=\sum_{\beta_1+\beta_2=\beta}\binom{\beta}{\beta_1,\beta_2}\int \epsilon^{|\beta_1|}(D^{\beta_1}\phi)(\epsilon\by) \al^1_\by(D_x^{\beta_2}b) e^{-i\by\bet}d\by d\bet\nonumber
\\ &=(2\pi)^d\sum_{\beta_1+\beta_2=\beta}\binom{\beta}{\beta_1,\beta_2}\epsilon^{|\beta_1|}(D^{\beta_1}\phi)(0) D_x^{\beta_2}b =(2\pi)^dD_x^{\beta}b
\end{align}  We also have
\begin{align*} D_x^\beta D_\xi^\gamma(\hat{b}_\e(\bet))=&D_x^\beta D_\xi^\gamma(\int \phi_\epsilon(\by) \al_\by(b)e^{-i\by\bet}d\by)\\
=&\int \phi_\epsilon(\by) \al_\by(D_x^\beta D_\xi^\gamma b)e^{-i\by\bet}d\by
= \widehat{D_x^\beta D_\xi^\gamma b}_\epsilon(\bet)\pl.
\end{align*}
 We write $c= c_1+c_2$ with
\[c_1=\frac{1}{(2\pi)^d}\int \al^2_{\bet}(a)b_\epsilon(\by)e^{-i \bet \cdot \by}d \bet d\by=\frac{1}{(2\pi)^d}\int \al_{\bet}^2(a)\hat{b}_\epsilon(\bet)d \bet\]
By Proposition \ref{taylor}, we use Taylor expansion with value in the Frechet space $\Sigma^m$,
\begin{align}\label{su}\al_{\bet}(a)=\sum_{|\beta|\le N} \frac{i^{|\beta|}(D_\xi^\beta a)\bet^\beta}{\beta !}+(N+1)\sum_{|\beta|= N+1} \frac{i^{|\beta|}}{\beta!}\bet^\beta\int_{0}^1 \al_{t\bet}(D_\xi^{\beta} a)(1-t)^N d t\pl.\end{align}
Using the calculation \eqref{cal}, the first part leads to
\[\frac{1}{(2\pi)^d}\int \sum_{|\beta|\le N} \frac{D_\xi^\beta a}{\beta !}\bet^\beta\hat{b}_\epsilon(\bet)d \bet=\sum_{|\beta|\le N} \frac{ i^{|\beta|}}{\beta !}D_\xi^\beta a D_x^\beta b \]
which gives the leading terms. For the second term in \eqref{su}, we have
$|\beta|=N+1$ and
\begin{align*}&\norm{\int_{0}^1\al_{t\bet}^2(D_\xi^\al a) (1-t)^N d t\bc{\xi}^{-m+N+1}}{}\\ \le & \int_{0}^{1}(1-t)^N\norm{\al_{t\bet}^2 \big(D_\xi^\beta (a)\bc{\xi}^{-m+N+1}\big) }{}\cdot\norm{\bc{\xi+t\bet}^{m-N-1}\bc{\xi}^{-m+N+1}}{}dt \\ \le & \int_{0}^{1}(1-t)^N\norm{D_\xi^\beta (a)\bc{\xi}^{-m+N+1} }{}\cdot\norm{\bc{\xi+t\bet}^{m-N-1}\bc{\xi}^{-m+N+1}}{}dt \\
\lesssim & \int_{0}^{1}(1-t)^N(t\bc{\bet})^{\lceil -m+N+1\rceil}dt \le A_{N,m}\bc{\bet}^{\lceil -m+N+1\rceil}\pl.
\end{align*}
Here $A_{N,m}$ is some positive constant only depends on $N,m$, and $\lceil r\rceil$ denote the smallest even integer greater than $|r|$.
On the other hand for any $\beta$,
\begin{align*}
\hat{b}_\epsilon(\bet)\bet^\beta=\sum_{\beta_1+\beta_2=\beta}\frac{\beta!}{\beta_1!\beta_2!}\int D_\by^{\beta_1}\phi_\epsilon(\by) \al_\by^2(D_x^{\beta_2}(b)) e^{-i\by\bet} d\by
\end{align*}
For each term
\begin{align*}
&\norm{\bc{\xi}^{m-N-1}  D_\by^{\beta_1}\phi_\epsilon(\by) \al_\by^1(D_x^{\beta_2}(b))\bc{\xi}^{-n-m+N+1}}{}
\\\le & |D_\by^{\beta_1}\phi_\epsilon(\by)|\pl\cdot\norm{\al_\by^1\big(\bc{\xi}^{m-N-1}D_x^{\beta_2}(b)\bc{\xi}^{-n-m+N+1}\big)}{}
\end{align*}
Here we used the assumption that $b, D_x^{\beta_2}(b)\in \Sigma^{n}$. Because $D_\by^{\beta_1}(\phi_\epsilon(\by))$
is a compactly supported function of $\by$, we have for any positive integer $l$,
\[\norm{\bc{\xi}^{m-N-1}\hat{b}_\epsilon(\bet)\bc{\xi}^{-n-m+N+1}}{}\le B_{n,m,N}(1+|\bet|^{-l})\pl,\]
where $B_{l,n,m,N}$ is a constant depending on $(l,n,m,N)$ and $\epsilon$. Thus, by choosing large enough $l$,
\begin{align*}
\norm{\int_{\R^d}\Big( \int_{0}^1 \al_{t\bet}(D^\beta_\xi a)(1-t)^N d t \Big ) \bet^\beta \hat{b}_\epsilon(\bet)d\eta \bc{\xi}^{-m-n+N+1}}{} \lesssim \int \bc{\bet}^{\lceil m-N-1\rceil}(1+|\bet|^{-l})d\bet < \infty \pl.
\end{align*}
Similar argument applies for derivatives,
\begin{align*}
& D_x^{\gamma_1} D_\xi^{\gamma_2}\Big(\int_{\R^d}\big( \int_{0}^1 \al_{t\eta}^2(D_\xi^\beta a)(\bet)(1-t)^N d t \big ) \bet^\beta \hat{b}_\epsilon(\bet)d\bet\Big)\\
\end{align*}
Therefore we obtain that
\begin{align*}c_1=\sum_{|\beta|\le N} \frac{(i)^{-|\beta|}}{\beta !}D_\xi^\beta a D_x^\beta b  + c_3
\end{align*} where $c_3$ is a remainder term in $\Sigma^{n+m-N-1}$.
Now take $\epsilon' <\epsilon$ and \[b_2(\by):=b_{\epsilon'}(\by)-b_\epsilon(\by)=(\phi(\epsilon'\by)-\phi(\epsilon\by))\al_\by(b)\] which is supported on $1/\epsilon<|\by|<2/\epsilon'$. Note that in above argument, we actually show that the singular integral
$\int \al_{\bet}(a)b(\by)e^{ i \bet\cdot \by}d\bet d\by$ converges absolutely if $b$ is compactly supported. Then for each $j$, we can use integration by parts
\begin{align*} \int \al_{\bet}(a)\by_j|\by|^{-2}b_2(\by)e^{ i \bet\cdot \by}d\bet d\by
=&\int \al_{\bet}(a)|\by|^{-2}b_2(\by)D_{\bet_j}e^{ i \bet\cdot \by}d\bet d\by
\\=&\int D_{\bet_j}(\al_{\bet})(a)|\by|^{-2}b_2(\by)e^{ i \bet\cdot \by}d\bet d\by
\\=&\int \al_{\bet}(D_{\xi_j}a)|\by|^{-2}b_2(\by)e^{ i \bet\cdot \by}d\bet d\by\pl.
\end{align*}
Here we used the property $D_{\bet_j}(\al_{\bet})(a)=\al_{\bet}(D_{\xi_j}a).$ Denote $\Delta_{\bet}=\sum_{j}D_{{\bet}_j}^2$, $\Delta_\xi=\sum_{j}D_{\xi_j}^2$ and $\Delta_\by=\sum_{j}D_{\by_j}^2$. Because $\Delta_{\bet} (\al^1_{\bet}(a))=\al_{\bet}^1(\Delta_\xi a)$, using the standard trick in singular integral,
\begin{align*} & \int \al_{\bet}(a)b_2(\by)e^{ i \bet\cdot \by}d\bet d\by
 = \int \al_{\bet}(\Delta_\xi^{m_1} a)|\by|^{-2m_1}b_2(\by)e^{ -i \bet \by}d\bet d\by
\\ =& \int \al_{\bet}(\Delta_\xi^{m_1} a)(1+\Delta_{\by})^{m_2}(|\by|^{-2m_1}b_2(\by))\bc{\bet}^{-2m_2}e^{ -i \bet \by}d\bet d\by
\end{align*}
Here $|\by|^{-2m_1}b_2(\by)$ has no singularity because $b_2$ is supported away from $\by=0$. Because $a\in\Sigma^m, b\in \Sigma^n$,
\[\Delta_\xi^{m_1}(a)\in \Sigma^{m-2m_1}\pl, (1+\Delta_{\by})^{m_2}(|\by|^{-2m_1}b_2(\by))\in \Sigma^{n}\pl.\]
We have
\begin{align}\label{est}
&\norm{\al_{\bet}(\Delta_\xi^{m_1} a)\bc{\xi}^{-m+2m_1}}{}\le \tilde{A}_{m,m_1} \bc{\bet}^{\lceil-m+2m_1\rceil}\pl, \nonumber\\
&\norm{\bc{\xi}^{m-2m_1}(1+\Delta_{\by})^{m_2}(|\by|^{-2m_1}b_2(\by))\bc{\xi}^{-m+2m_1-n}}{}\le \tilde{B}_{m,m_1,n}(1+|\by|)^{-2m_1} {\chi}_{\{ \frac{1}{\epsilon}<|\by|<\frac{2}{\epsilon'}\}}
\end{align}
for some constants $\tilde{A}_{m,m_1}$ and $\tilde{B}_{m,m_1,n}$. We can choose $m_1,m_2$ large enough such that $2m_1>N+1$ and then the integral
\begin{align*}\norm{\int \al_{\bet}(a)b_2(\by)e^{-i \bet\by}d\bet d\by\cdot \bc{\xi}^{-m-n+N+1}}{}\le \int   |{\bet}|^{\lceil-m+2m_1\rceil}\bc{\bet}^{-2m_2} (1+|\by|)^{-2m_1} d\bet d\by <\infty
\end{align*}
converges absolutely. The argument for the derivatives are similar. Hence
\[\int \al_{\bet}(a)b_2(\by)e^{-i \bet\by}d\bet d\by\in \Sigma^{n+m-N-1}\pl,\]
which is of lower order of the leading terms. Note that the above estimates is uniform for $0<\epsilon', \epsilon <1$ and when $\epsilon',\epsilon \to 0$, the norm estimates \eqref{est} goes to $0$. So when $\epsilon\to 0$, the remainder $c_2$
converges to $0$ in $\Sigma^{n+m-N-1}$. This implies
\[c=\lim_{\epsilon \to 0}\int \al_{\bet}^2(a)\phi(\epsilon\by)\al_{\by}^1(b)e^{i\bet \by}d\bet d\by\]
converges in $\Sigma^{m+n}$.

Write $\displaystyle c_\epsilon=\int \al_{\bet}^2(a)\phi(\epsilon\by)\al_{\by}^1(b)e^{i\bet \by}d\bet d\by$. We now show that for any $\la_{\theta,\theta'}(F)\in \S_{\theta,\theta'}$,
\begin{align*}
Op(a)Op(b)\la_{\theta,\theta'}(F)=\lim_{\epsilon\to 0} Op(c_\epsilon)\la_{\theta,\theta'}(F)= Op(c)\la_{\theta,\theta'}(F)
\end{align*}
Indeed, since the integral in $c_\epsilon$ converges absolutely
\begin{align*}
Op(c_\epsilon)\la_{\theta,\theta'}(F)=& \int \al_{\bet_1}^2\Big(\int \phi(\epsilon\by)\al_{\bet}^2(a)\al_{\by}^1(b)e^{-i\bet \by}d\bet d\by\Big)\hat{F}(\bet_1,\by_1)\la_{\theta,\theta'}(\bet_1,\by_1)d\bet_1 d\by_1\\
=& \int\phi(\epsilon\by)e^{-i\bet \by}\al_{\bet+\bet_1}^2(a)\al_{\by}^1\al_{\bet}^2(b)\hat{F}(\bet_1,\by_1)\la_{\theta,\theta'}(\bet_1,\by_1)d\bet_1 d\by_1d\bet d\by\\
=& \int\phi(\epsilon\by)e^{-i(\bxi-\bet_1) \by}\al_{\bxi}^2(a)\al_{\by}^1\al_{\bet_1}^2(b)\hat{F}(\bet_1,\by_1)\la_{\theta,\theta'}(\bet_1,\by_1)d\bet_1 d\by_1d\bxi d\by\\
=& \int\phi(\epsilon\by)e^{-i\bxi\by}\al_{\bxi}^2(a)\al_{\by}^1\Big(\int \al_{\bet_1}^2(b)\hat{F}(\bet_1,\by_1)\la_{\theta,\theta'}(\bet_1,\by_1)d\bet_1 d\by_1\Big)d\bxi d\by\\
=& \int\phi(\epsilon\by)e^{-i\bxi\by}\al_{\bxi}^2(a)\al_{\by}^1\Big(Op(b)\la_{\theta,\theta'}(F)\Big)d\bxi d\by\pl.
\end{align*}
Then it suffices to show that for any $\la_{\theta,\theta'}(G)$,
\[   \lim_{\epsilon\to 0} \int\phi(\epsilon\by)e^{-i\bxi\by}\al_{\bxi}^2(a)\al_{\by}^1\Big(\la_{\theta,\theta'}(G)\Big)d\bet d\by=Op(a)\la_{\theta,\theta'}(G)\pl. \]
Let $\hat{\phi}$ be the Fourier transform of $\phi$.
\begin{align*}
\int\phi(\epsilon\by)e^{-i\bxi\by}\al_{\by}^1(\la_{\theta,\theta'}(G)) d\by
= &\int\phi(\epsilon\by)e^{-i(\bxi-\bet_1)\by}\hat{G}(\bet_1,\by_1) \la_{\theta,\theta'}(\bet_1,\by_1) d\by d\by_1 d\bet_1
\\=& \int \frac{1}{\epsilon^d}\hat{\phi}(\frac{\bxi-\bet_1}{\epsilon})\hat{G}(\bet_1,\by_1) \la_{\theta,\theta'}(\bet_1,\by_1)d\by_1 d\bet_1
\end{align*}
Here $\displaystyle\frac{1}{\epsilon^d}\hat{\phi}(\frac{\cdot}{\epsilon})$ approximates the delta function,
\begin{align*}\int\phi(\epsilon\by)e^{-i\bxi\by}\al_{\bxi}^2(a)\al_{\by}^1\Big(\la_{\theta,\theta'}(G)\Big)d\bet d\by&= \int \frac{1}{\epsilon^d}\hat{\phi}(\frac{\bxi}{\epsilon})Op(\al^2_{\bxi}a)\la_{\theta,\theta'}(G)  d\bxi
\\
= &\int \frac{1}{\epsilon^d}\hat{\phi}(\frac{\bxi}{\epsilon})\al^2_{\bxi}\Big(Op(a)\al^2_{-\bxi}\la_{\theta,\theta'}(G) \Big) d\bxi\pl.
\end{align*}
Since $\bxi\to \al^2_{\bxi}\Big(Op(a)\al^2_{-\bxi}\la_{\theta,\theta'}(G)\Big)$ is continuous in $\S_{\theta,\theta'}$. When $\epsilon \to 0$, the above integral converges to $Op(a)\la_{\theta,\theta'}(G)$ in $\S_{\theta,\theta'}$.
\end{proof}
\subsection{Integrability and trace formula}
In the rest of this section we discuss the integrability of $\Psi$DOs whose symbols is integrable in the first component $\R_\theta$.
\begin{defi}[Tame symbols]
An element $a \in \M_{\theta, \theta'}$ is a {\bf tame symbol} of order $m$ if there exists a $r>d$ such that for any $\al, \beta$ and $\gamma$,
\[\bc{x}^r D_x^\al D_\xi^\beta (a)\bc{\xi}^{|\beta|-m}\]
extends to bounded element in $\R_{\theta,\theta'}$.
We write $\Sigma^m_{tame}$ the set of all tame symbols of order $m$ and $\Sigma_{tame}^{-\infty}:=\cap_{r} \Sigma_{tame}^r$.
\end{defi}
\begin{prop}
A symbol $a\in \Sigma_{tame}^m$ if and only if there exists $r>d$ such that for all $\al,\beta$, $D^\al_x D^\beta_\xi(a)\in O^{-r,m-|\beta|}$. Moreover, if $b\in \Sigma^n$, $ab, ba\in\Sigma_{tame}^{n+m}$.
\end{prop}
\begin{proof}This is a direct consequence of Theorem \ref{char}.
\end{proof}
\begin{lemma}Let $a\in L_2(\R_\theta)$ and $b\in L_2(\R_{\theta'})$. Then $ab\in L_2(\R_\Theta)$ and $\norm{ab}{L_2(\R_\Theta)}=\norm{a}{L_2(\R_\theta)}
\norm{b}{L_2(\R_{\theta'})}$.
\end{lemma}
\begin{proof}It can be verified from the definition of $tr_\Theta$ that for $f\in \S_\theta, g\in \S_{\theta'}$
\[tr_\Theta(\la_\theta(f)\la_{\theta'}(g))=tr_\theta(\la_\theta(f))tr_{\theta'}(\la_\theta(g))\pl.\]
Then we have
\begin{align*}
\norm{\la_\theta(f)\la_{\theta'}(g)}{L_2(\R_\Theta)}^2
=& tr_\Theta(\la_{\theta'}(g)^*\la_\theta(f)^*\la_\theta(f)\la_{\theta'}(g))
= tr_\Theta(\la_\theta(f)^*\la_\theta(f)\la_{\theta'}(g)\la_{\theta'}(g)^*)
\\=& tr_\theta(\la_\theta(f)^*\la_\theta(f))tr_{\theta'}(\la_{\theta'}(g)\la_{\theta'}(g)^*)
\\=&\norm{\la_\theta(f)}{L_2(\R_\theta)}^2\norm{\la_{\theta'}(g)}{L_2(\R_{\theta'})}^2
\end{align*}
The assertion for general $a\in L_2(\R_\theta), b\in L_2(\R_{\theta'})$ follows from density.
\end{proof}
\begin{cor}\label{int}Let $a\in S_{tame}^m$. Then
\begin{enumerate}
\item[i)]
$Op(a)\in L_2(\R_\Theta)$ if $m<-\frac{d}{2}$;
\item[ii)]
$Op(a)\in L_1(\R_\Theta)$ if $m<-d$.
\end{enumerate}
\end{cor}
\begin{proof}We know from the algebraic property that $Op(\la_{\theta}(f_1)\ten \la_{\theta'}(f_2))=\la_{\theta}(f_1) \la_{\theta'}(f_2)$ for $f_1,f_2\in \S(\R^d)$. The $Op$ is a $L_2$-isometry and trace preserving on $\S_{\theta,\theta'}$. Let $a\in \Sigma_{tame}^m$. Then for some $r>d$,
\begin{align*}Op(a)=&\bc{x}^{-r}\bc{\xi}^{m}\bc{\xi}^{-m}\bc{x}^{r}Op(a )=\bc{x}^{-r}\bc{\xi}^{m}\bc{\xi}^{-m}Op(\bc{x}^{r}a)\\=&\Big(\bc{x}^{-r}\bc{\xi}^{m}\Big)\Big(\bc{\xi}^{-m}Op(\bc{x}^{r}a)\Big)\pl.\end{align*}
By symbol calculus, $\bc{\xi}^{-m}Op(\bc{x}^{r}a)$ is a $\Psi$DO of order $0$ hence in $\R_\Theta$. For $m<-d/2$, $\norm{\bc{\xi}^{m}}{L_2(\R_{\theta'})}<\infty$ and $\norm{\bc{x}^{-r}}{L_2(\R_{\theta})}<\infty$. Then $\bc{x}^{-r}\bc{\xi}^{m}\in L_2(\R_\Theta)$ and
\begin{align*}\norm{Op(a)}{2}\le \norm{\bc{x}^{-r}\bc{\xi}^{m}}{2}\norm{\bc{\xi}^{-m}Op(\bc{x}^{r}a)}{\infty}
\end{align*}
For $m<-d$, choose $n=\frac{m}{2}$,
\begin{align*}Op(a)=\Big(\bc{x}^{n}\bc{\xi}^{n}\Big)\Big(\bc{\xi}^{-n}Op(\bc{x}^{-n}a)\Big)\pl.\end{align*}
$\bc{\xi}^{-n}Op(\bc{x}^{-n}a)$ is a tame $\Psi$DO of order less than $d/2$ hence in $L_2(\R_\Theta)$ and $\bc{x}^{-n}\bc{\xi}^{-n}$ is also in $L_2(\R_\Theta)$ by the discussion in $i)$.
 \end{proof}
We end this section with the trace formula.
\begin{prop}\label{trin}Suppose a symbol $a\in L_1(\R_{\theta,\theta'})$ and its operator $Op(a)\in L_1(\R_\Theta)$. Then
\[\tau_\Theta(Op(a))=\tau_{\theta,\theta'}(a)\pl.\]
\end{prop}
\begin{proof}
Using the definition of $Op(a)$,
\begin{align*}\tau_\Theta(Op(a)\la_\Theta(F))=&\tau_{\theta,\theta'}\Big( \int_{\R^{2d}} \hat{F}(\bet,\by)\al^2_{\bet}(a)\la_{\theta,\theta'}(\bet,\by)d\bet d\by \Big)
\\ =& \int_{\R^{2d}} \hat{F}(\bet,\by)\Big(\tau_{\theta,\theta'}  ( \al^2_{\bet}(a)\la_{\theta,\theta'}(\bet,\by)) \Big) d\bet d\by
\\ =& \int_{\R^{2d}} \hat{F}(\bet,\by)\tau_{\theta,\theta'}\Big(  a\al_{-\bet}^2(\la_{\theta,\theta'}(\bet,\by)) \Big) d\bet d\by
\\ =& \int_{\R^{2d}} \hat{F}(\bet,\by)e^{-i\bet\by}\Big(\tau_{\theta,\theta'}  ( a\la_{\theta,\theta'}(\bet,\by))\Big) d\bet d\by
\\ =& \tau_{\theta,\theta'}  ( a \la_{\theta,\theta'}(F'))\pl,
\end{align*}
where $F'$ has the Fourier transform $\hat{F'}(\bet,\by)=\hat{F}(\bet,\by)e^{-i\bet\by}$.
Here we use the Fubini theorem because $a\in L_1(\R_{\theta,\theta'})$. Let $F_n\in \S(\R^{2d})$ be a sequence of Schwartz function in Proposition \ref{ai}. Then $\la_\Theta(F_n)$ (resp. $\la_{\theta,\theta'}(F_n)$) is an approximation of identity in $L_1(\R_\Theta)$ (resp. $L_1(\R_{\theta,\theta'})$). Take $F_n'\in \S(\R^{2d})$ such that $\hat{F_n'}(\bet,\by)=\hat{F_n}(\bet,\by)e^{-i\bet\by}$. Note that $\norm{\hat{F}_n}{1}=1$ and $\hat{F}_n$ is supported in $|(\bet,\by)|\le \frac{1}{n}$. When $n\to 1$,
\[\norm{\la_{\theta,\theta'}(F_n)-\la_{\theta,\theta'}(F_n')}{\infty}\le \norm{\hat{F'}_n-\hat{F}_n}{1}=\int_{\R^{2d}} \hat{F}_n(\bet,\by)|1-e^{-i\bet\by}|d\bet d\by\to 0 \pl. \]
 Therefore,
\begin{align*}
\tau_\Theta(Op(a))=&\lim_{n\to \infty} \tau_\Theta(Op(a)\la_\Theta(F_n))=\lim_{n\to \infty} \tau_{\theta,\theta'}(a\la_{\theta,\theta'}(F_n'))=\lim_{n\to \infty} \tau_{\theta,\theta'}(a\la_{\theta,\theta'}(F_n))\nonumber\\=&\tau_{\theta,\theta'}(a)\pl . \qedhere
\end{align*}
\end{proof}

\section{Local Index formula}\label{sectionloc}In this section we discuss the spectral triple structure on $\R_\theta$ equipped with noncommuting partial derivatives. We first recall the definitions of semi-finite spectral triple from \cite{CGRSmemo}.
  We shall show that the non-commuting derivatives in Section \ref{phido} gives a natural example of semi-finite spectral triple. The main results of this chapter is a simplified index formula and we calculate it for the Bott projector as an example.

\subsection{Semifinite spectral triple} Let $\N$ be a von Neumann algebra equipped with a normal faithful semi-finite trace $\tau$. The $\tau$-compact operators $ \K(\N,\tau)$ is the norm completion of $L_1(\N,\tau)\cap \N$ in $\N$. In our case $\K(\R_\theta,\tau_\theta)=\E_\theta$. The following definitions of semi-finite spectral triple is from \cite{CGRSmemo}.
\begin{defi}A semi-finite spectral triple $(\A,H,D)$, relative to a semi-finite tracial von Neumann algebra $(\N,\tau)$, is by given a Hilbert space $H$, a $*$-subalgebra $\A$ of $\N$ acting on $H$, and a densely defined unbounded self-adjoint operator $D$ affiliated to $\N$ such that
\begin{enumerate}
\item[i)] $a\cdot \operatorname{dom} D\subset \operatorname{dom} D$ for all $a\in \A$, so that $da:=[D,a]$ is densely defined. Moreover, $da$ extends to a bounded operator in $\N$ for all $a\in \A$;
\item[ii)] $a(1+D^2)^{-1/2}\in \K(\N,\tau)$.
\end{enumerate}
$(\A,H,D)$ is {\bf even} if there is an operator $\gamma\in \N$ such that for all $a\in \A$, \[\gamma=\gamma^*, \gamma^2=1, \gamma a=a \gamma , \text{ and }  D\gamma+\gamma D=0.\]
$(\A,H,D)$ is {\bf finitely summable} if there exists $s>0$ such that ${a(1+D^2)^{-\frac{s}{2}}\in L_1(\N, \tau)}$ for all $a\in \A$. Then
\[p=\inf\{s>0 | \pl \text{for all}\pl a\in \A, a(1+D^2)^{-\frac{s}{2}}\in L_1(\N, \tau)\}\]
 is called the spectral dimension of $(\A,H,D)$.
\end{defi}
The subalgebra $\A$ plays the role of smooth functions. The main difference to the compact case is the condition ii), which simplifies to that $(1+D^2)^{-1/2}$ is compact. The semi-finiteness allow locally compact space equipped with non-finite measure. We recall the following sufficient condition for the {\bf smooth summability} of a semi-finite spectral triple and refer to \cite{CGRSmemo} for the detailed definition.
\begin{prop}[Proposition 2.21. of \cite{CGRSmemo}] Let $(\A,H,D)$ be a spectral triple of spectral dimension $p$ relative to $(\N,\tau)$. If for all $a \in \A \cup [D,\A]$, $k\in \mathbb{N}^+$ and $s > p$,
\[(1 + D^2)^{-\frac{s}{4}} L^k(a)(1 + D^2)^{-\frac{s}{4}} \in L_1(\N,\tau),\]
then $(A,H,D)$ is {\bf smoothly summable}. Here $L(T):= (1 + D^2 )^{-\frac{1}{2}} [D^2 ,T]$ and $L^k(T)=L (L^{k-1}(T))$.
\end{prop}

Quantum Euclidean space $\R_\theta$ equipped with its natural partial derivative $D_j$'s were studied as the prototypical example of semi-finite spectral triple in \cite{moyalplane, CGRSmemo}. The rest of this subsection is to show that the non-commuting derivatives also gives a semi-finite spectral triple structure of $\R_\theta$. First, we choose the smooth subalgebra $\A$ to be the noncommutative Sobolev space
\[W^{1,\infty}(\R_\theta)=\{a \pl | \pl D^\al(a)\in L_1(\R_\theta) \pl \text{for all} \pl \al\}\pl. \]
In the classical case $W^{1,\infty}(\R^d)\subset C^\infty_0(\R^d)$  by Sobolev embedding theorem (c.f. \cite{loukas}). The next lemma is a weaker analog on $\R_\theta$.
\begin{lemma} If $D^\al(a)\in L_1(\R_\theta)$ for all $\al$, then $D^\al(a)\in L_p(\R_\theta)$ for all $1\le p\le \infty$ and $\al$. In particular, the unitalization $W^{1,\infty}(\R_\theta)^{\sim}:=(W^{1,\infty}(\R_\theta)+\C) $ is a dense $*$-subalgebra of $\E_\theta^{\sim}$ closed under holomorphic function calculus.
\end{lemma}
\begin{proof}
Denote $\Delta=\sum_j D_{x_j}^2$. For $\la_\theta(f)\in \S_\theta$,
\[(1+\Delta)\la_\theta(f)=\la_\theta((1+\Delta)f)= \int\bc{\bet}^2\hat{f}(\bet)\la_\theta(\bet)d\bet \pl.\]
Choose a integer $2n>d$,
we have $(1+\Delta)^{-n}:L_2(\R_\theta)\to L_\infty(\R_\theta)$ is bounded because
\begin{align*}
\norm{(1+\Delta)^{-n}\la_\theta(f)}{}=&\norm{\int\bc{\bet}^{-n}\hat{f}(\bet)\la_\theta(\bet)d\bet}{}
\le \norm{\bc{\bet}^{-n}\hat{f}}{1}
\\ \le &\norm{\bc{\bet}^{-n}}{2}\norm{\hat{f}}{2}
=\norm{\bc{\bet}^{-n}}{2}\norm{\la_\theta(f)}{2}.
\end{align*}
By duality, we also have that $(1+\Delta)^{-n}:L_1(\R_\theta)\to L_2(\R_\theta)$ is bounded. Indeed, for any $\la_\theta(f),\la_\theta(g)\in \S_\theta$,
\begin{align*}
 \lan \la_\theta(g), (1+\Delta)^{-n}\la_\theta(f)\ran_{\tau_\theta}&=\lan (1+\Delta)^{-n}\la_\theta(g), \la_\theta(f)\ran_{\tau_\theta}\\&\le  \norm{(1+\Delta)^{-n}\la_\theta(g)}{\infty} \norm{\la_\theta(f)}{1}\le  C\norm{\la_\theta(g)}{2} \norm{\la_\theta(f)}{1}
\end{align*}
Here we have used the fact $(1+\Delta)^{-n}$ is self-adjoint on $\S_\theta$. Thus we have that $(1+\Delta)^{-n}:L_1(\R_\theta) \to L_\infty(\R_\theta)$ is continuous. If $D^\al(a)\in L_1(\R_\theta)$ for all $|\al|\le 2n$, then $(1+\Delta)^n (a)\in L_1(\R^d)$ and hence $a \in L_\infty(\R_\theta)$. Therefore $W^{1,\infty}(\R_\theta)$ is closed under product hence a subalgebra of $\E_\theta$. It is dense because $\S_\theta \subset W^{1,\infty}(\R_\theta)$. To show $W^{1,\infty}(\R_\theta)$ is closed under holomorphic calculus, it suffices to consider the resolvent $(\la-a)^{-1}$ for $\la\notin \text{Spec}(a)$. Indeed, $(\la-a)^{-1}$ is bounded and
\[\la^{-1}-(\la-a)^{-1}=\la^{-1}\big((\la-a)-\la\big)(\la-a)^{-1}=-\la^{-1}a(\la-a)^{-1}\in L_1(\R_\theta)\pl. \]
For the derivatives, \[[D_j,(\la-a)^{-1}]=(\la-a)^{-1}[D_j,a](\la-a)^{-1}\in L_1\]
For higher order derivatives $D^\al$, we use induction and Leibniz rule
\begin{align*}D^\al ((\la-a)^{-1})=&D^\al ((\la-a)^{-1}(\la-a)(\la-a)^{-1})
\\=&\sum_{\al_1+\al_2+\al_3=\al}\frac{\al !}{\al_1 !\al_2 !\al_3 !}D^{\al_1}((\la-a)^{-1})D^{\al_2}(\la-a)D^{\al_3}((\la-a)^{-1}) \pl.\qedhere
\end{align*}
\end{proof}
The above lemma implies that the inclusion $W^{1,\infty}(\R_\theta)\subset \E_\theta$ induces $K$-groups isomorphism (c.f. page 292 of \cite{connes94}). In particular, every projection (resp. unitary) in $\E_\theta^{\sim}$ or $M_n(\E_\theta^{\sim})$ can be approximated using projections (resp. unitary) in $W^{1,\infty}(\R_\theta)^{\sim}$.
To verify the finite and smooth summability, we need the following lemma.
\begin{lemma} \label{prodint}\label{L1}Let $a\in W^{1,\infty}(\R_\theta)$. Then $\bc{\xi}^{-\frac{r}{2}}a\bc{\xi}^{-\frac{r}{2}},a\bc{\xi}^{-r}\in L_1(\R_\Theta)$ if $r>d$.
\end{lemma}
\begin{proof}We write $a$ as $a=a_1a_2$ with $a_1,a_2\in L_2(\R_\theta)$. Then
\[\bc{\xi}^{-\frac{r}{2}}a\bc{\xi}^{-\frac{r}{2}}=(\bc{\xi}^{-\frac{r}{2}}a_1)(a_2\bc{\xi}^{-\frac{r}{2}})\in L_1(\R_\Theta)\] because \begin{align*}
\norm{\bc{\xi}^{-\frac{r}{2}}a_1}{L_2(\R_\theta)}=\norm{\bc{\xi}^{-\frac{r}{2}}}{L_2(\R_{\theta'})}\norm{a_1}{L_2(\R_\theta)}\pl,\pl
\norm{a_2\bc{\xi}^{-\frac{r}{2}}}{L_2(\R_\theta)}=\norm{\bc{\xi}^{-\frac{r}{2}}}{L_2(\R_{\theta'})}\norm{a_2}{L_2(\R_\theta)}\pl.
\end{align*}
Note that $\bc{\xi}^{-\frac{r}{2}}[a,\bc{\xi}^{-\frac{r}{2}}]=\bc{\xi}^{-\frac{r}{2}}a\bc{\xi}^{-\frac{r}{2}}-a\bc{\xi}^{-r}$. To show $\bc{\xi}^{-\frac{r}{2}}[a,\bc{\xi}^{-\frac{r}{2}}]\in L_1(\R_\Theta)$, choose $n$ such that $2n>\frac{r}{2}$ and write $s=\frac{r}{4n}$. By operator integral,
\begin{align*}
\bc{\xi}^{-\frac{r}{2}}[a,\bc{\xi}^{-\frac{r}{2}}]=&C_s\bc{\xi}^{-\frac{r}{2}}\int_0^\infty t^{-s} [a,(t+\bc{\xi}^{2n})^{-1}]dt
\\ =&C_s\bc{\xi}^{-\frac{r}{2}}\int_0^\infty t^{-s} (t+\bc{\xi}^{2n})^{-1}[a, t+\bc{\xi}^{2n}](t+\bc{\xi}^{2n})^{-1}dt
\\ =C_s&\int_0^\infty t^{-s} (t+\bc{\xi}^{2n})^{-1} \Big(\bc{\xi}^{-\frac{r}{2}}[a, \bc{\xi}^{2n}]\bc{\xi}^{-2n}\Big)  \bc{\xi}^{2n}(t+\bc{\xi}^{2n})^{-1} dt
\end{align*}Here $C_s$ is some positive constant depending on $s$. Since $[a, \bc{\xi}^{2n}]$ is a linear combination of $a$'s derivatives, we know
\[  \bc{\xi}^{-\frac{r}{2}}[a, \bc{\xi}^{2n}]\bc{\xi}^{-2n}\in L_1(\R_\Theta)\pl.\]
Then the integral converges in $L_1$-norm,
\begin{align*}
&\norm{\bc{\xi}^{-\frac{r}{2}}[a,\bc{\xi}^{-\frac{r}{2}}]}{1}\\ \lesssim &
\int_0^\infty t^{-s} \norm{(t+\bc{\xi}^{2n})^{-1}}{\infty} \norm{\bc{\xi}^{-\frac{r}{2}}[a, \bc{\xi}^{2n}]\bc{\xi}^{-2n}}{1} \norm{\bc{\xi}^{2n}(t+\bc{\xi}^{2n})^{-1}}{\infty} dt
\\ \lesssim &\int_0^\infty t^{-s} (t+1)^{-1}dt <\infty\pl. \qedhere
\end{align*}
\end{proof}

Recall that the Clifford algebra $Cl^d$ is generated by $d$ self-adjoint operators $c_1,\cdots, c_d$ satisfying the anti-commutation relation $c_jc_k+c_kc_j=2\delta_{j,k}$.
For $d=2n$ even,  $Cl^d$ is isomorphic to the $N\times N$ matrix algebra $M_N$ with $N=2^n$. For $d=2n+1$ odd, $Cl^d$ is isomorphic to $M_{2^n}\oplus M_{2^n} \subset M_N$ with $N=2^{n+1}$. When $d$ even, $Cl^d$ is $\mathbb{Z}_2$ graded with the parity element $\gamma=(-i)^\frac{d}{2}c_1\cdots c_d$.
\begin{theorem} $(W^{\infty,1}(\R_\theta)\ten M_N, L_2(\R_\Theta)\ten \C^N, \sum_{j}\xi_j\ten c_j)$ relative to $(\R_\Theta\ten M_N, \tau_\Theta\ten tr)$ is a smooth summable semi-finite spectral triple with spectral dimension $d$. Moreover it is even if $d=2n$ is even, and $\gamma=(-i)^\frac{d}{2}c_1\cdots c_d$.
\end{theorem}
\begin{proof} Note that \[D^2=\sum_{j,k}\xi_j\xi_k\ten c_jc_k =\sum_{j}\xi_j^2-\frac{i}{2}\sum_{j,k} \theta'_{j,k}c_jc_k \pl.\]
Denote $\omega=\frac{i}{2}\sum_{j,k} \theta'_{j,k}c_jc_k$. Then $1+D^2=\bc{\xi}^2-\omega$.
Since $\omega\in M_N$ commutes with $\R_\Theta$, to verify summability it is equivalent to replace $1+D^2$ by $\bc{\xi}^2$. By Lemma \ref{L1}, we know the spectral dimension is less than $d$. On the other hand, if $a\bc{\xi}^{-r}\in L_1(\R_\Theta)$,
\[\norm{a\bc{\xi}^{-\frac{r}{2}}}{2}^2\le\norm{a\bc{\xi}^{-d}a^*}{1}\le\norm{a^*}{\infty}\norm{a\bc{\xi}^{-d}}{1}<\infty\]
which implies $r>d$. For smooth summability, we know $[\bc{\xi}^2,a]\in L_1(\R_\theta)$ and by Lemma \ref{L1} again,
\[ (1+D^2)^{-\frac{s}{2}}L(a)(1+D^2)^{-\frac{s}{2}}\in L_1(\R_\Theta)\]
if $s>d$. The arguments for $L^k(a)$ are similar.
\end{proof}
\subsection{Local Index formula}
We briefly recall the local index formula for the even case and refer to \cite{CM95,CGRSmemo} for detailed information. Let $(\A,H,D)$ be an even spectral triple relative to $(\N,\tau)$ and $\gamma$ is the parity element. Denote $H_+=\frac{\gamma+1}{2}H$ and $H_-=\frac{1-\gamma}{2}H$
For $\mu>0$, define $D_\mu=\left[\begin{array}{cc} D &\mu\\ \mu&D\end{array}\right]$ on $H\oplus H$. Write $F_\mu=D_\mu|D_\mu|^{-1}$ and \begin{align}\label{phase}(F_\mu)_{+}=(\frac{1+\gamma}{2}\ten I_2)F_\mu: H_+\oplus H_+\to H_-\oplus H_- \pl .
\end{align}
Here and in the following $I_n$ represents the $n$-dimensional identity matrix. For a projection $e\in M_n(\A^\sim)$, denote $\hat{e}=\left[\begin{array}{cc} e &0\\0&1_e\end{array}\right] \in M_{2n}(\A^\sim)$ where $1_e\in M_n(\C)$ is the rank element of $e$. Following \cite[Definition 2.12 and Proposition 2.13]{CGRSmemo}, the numerical index pairing between the $K_0(\A)$ element $[e]-[1_e]$ and the even spectral triple $(\A,H,D)$ is given by
\[\lan [e]-[1_e], (\A,H,D)\ran=\text{index}_{\tau\ten tr_{2n}}(\hat{e}(F_{\mu , +}\ten I_n)\hat{e})\]
Here the numerical index  $\text{index}_{\tau}(F)=\tau(\text{ker}F)-\tau(\text{coker}F)$ is defined as the trace of kernel subtracting the trace of cokernel.
Both quantities are topological invariants under homotopy. The local index formula express the index pairings by the following residue cocycle formulas.
\begin{defi}
$(\A,H,D)$ has {\bf isolated spectral dimension} if for all $a_0,\cdots, a_m\in \A$, the zeta function
\[\zeta(z)=\tau(\gamma a_0d a_1^{(k_1)}\cdots d a_m^{(k_m)}(1+D^2)^{-|k|-m/2-z})\]
has an analytic continuation to a deleted neighbourhood of $z=0$.
\end{defi}
Here we introduce the notation $da:=[D,a]$ and $\displaystyle da^{(k)}:=\underbrace{[D^2,[D^2,\cdots [D^2}_{\text{$k$-times}},da]]$. Let $(\A,H,D)$ be a smoothly summable semifinite spectral triple with spectral dimension $d$ and
$M$ be the largest integer in $[0,d+1]$. Suppose $\A$ has isolated spectral dimension. The residue cocycle $\phi_m:\A^{\ten m+1}\to \mathbb{C}$ is the $(m+1)$-linear form given by
\begin{align}
\phi_0(a_0)=&Res_{z=0}z^{-1}\tau( \gamma a_0(1+D^2)^{-z})\\
\phi_m(a_0,\cdots,a_m)=&\sum_{|k|=0}^{M-m}(-1)^{|k|}\al(k)\sum_{j=0}^{|k|+m/2} \nonumber\\ & \si_{|k|+m/2,j}Res_{z=0}z^{j-1}\tau(\gamma a_0d a_1^{(k_1)}\cdots d a_m^{(k_m)}(1+D^2)^{-|k|-m/2-z}) \pl. \label{cocycle}\end{align}
where $\al(k),\si_{|k|+m/2,j}$ are the constant defined as follows. For a multi-index $k=(k_1,\cdots,k_m)$,
\begin{align}\label{alc}\al(k)= k_1!k_2!\cdots k_m!/(k_1+1)(k_1+k_2+2)\cdots(|k|+m)\pl. \end{align} $\si_{n,j}$ are the non negative constant given by the equation
\begin{align}&\prod_{j=0}^{n-1}(z+j)=\sum_{j=1}\si_{n,j}z^j \pl \text{for} \end{align}
In particular, $\al(0)=m!$ and $\si_{n,1}=(n-1)!$. The terms in $\phi_m$ is a linear combination of residue and higher order residue of the zeta function
\[  \zeta(z)=\tau(\gamma a_0d a_1^{(k_1)}\cdots d a_m^{(k_m)}(1+D^2)^{-|k|-m/2-z})\pl.\]
The isolated spectral dimension condition assumes that these residues are well-defined.
\begin{theorem}[Theorem 3.33 of \cite{CGRSmemo} (even case)] Let $(\A,H,D)$ relative to $(\N,\tau)$ be an even smoothly summable semi-finite spectral triple. Suppose that $(\A,H,D)$ has isolated spectral dimension. Then the numerical index pairing can be computed by
\begin{align*}&\lan [e]-[1_e], [(\A,H,D)]\ran= \sum_{m=0,even}^M \phi_m(Ch^m(e)-Ch^m(1_e)) \pl,
\end{align*}
where for a projection $e\in M_n(\A^\sim)$,  $Ch_0(e)=(e)$ and \begin{align*} &Ch^{2k}(e)=(-1)^k\frac{2k!}{k!}(e-\frac{1}{2})\ten e \ten \cdots\ten e\in \A^{\ten 2k+1}\pl.
\end{align*}
\end{theorem}
We shall now calculate the local index formula for the spectral triple $(W^{\infty,1}(\R_\theta),L_2(\R_\Theta)\ten \C^N, \sum_{j}\xi_j\ten c_j)$. Recall that $\omega =\frac{i}{2}\sum\theta_{jk}' c_jc_k$ is the analog of curvature form. Let us denote the super trace on $Cl^d$ as $str(a)=tr(\gamma a)$ and the corresponding super trace on $\R_\Theta\ten Cl^d$ (resp. $\R_\theta\ten Cl^d$) as $Str_\Theta=\tau_\Theta\ten str$ (resp. $Str_\theta=\tau_\theta\ten str$).
\begin{theorem}Let $d$ be even. The spectral triple $(W^{\infty,1}(\R_\theta),L_2(\R_\Theta)\ten \C^N, \sum_{j}\xi_j\ten c_j)$ has isolated spectral dimension. Moreover,
$a_0,\cdots, a_{m}\in W^{\infty,1}(\R_\theta)$,
\[\phi_{m}(a_0,\cdots, a_{m})=\begin{cases}
         \frac{\pi^\frac{d}{2}}{m!}Str_\theta (a_0 da_1\cdots  da_m  \frac{\omega^{\frac{d-m}{2}}}{\frac{(d-m)}{2}!}), & \mbox{if } m \pl \text{even} \\
         0, & \mbox{if } m \pl \text{odd}.
       \end{cases}\pl.\]
\end{theorem}
\begin{proof} We first consider $m>0$. Let us denote $\Psi_k=a_0d a_1^{(k_1)}\cdots d a_m^{(k_m)}$. The cocycle $\phi_{m}$ is a linear combination of residue of the zeta functions at $z=0$,
\[\zeta_{k}(z)=Str_\Theta(\Psi_k(1+D^2)^{-|k|-\frac{m}{2}-z})\pl.\] Because $a_0,\cdots,a_{m}\in  W^{\infty,1}(\R_\theta)^\sim$ and $d a_j^{(k_j)}$ are derivatives of $a_j$, ${\Psi_k\in W^{\infty,1}(\R_\theta)\ten Cl^d}$. Using the same argument of Lemma \ref{prodint}, one can obtain that $\Psi_k(1+D^2)^{-r}\in L_1(\R_\Theta\ten M_N)$ if $r>\frac{d}{2}$. Then $\zeta_{k}(z)$ is analytic for $|k|+\frac{m}{2}+
\text{Re} \pl z>\frac{d}{2}$, and hence it suffices to consider the nonzero residue of $\zeta_{k}$ at $z=0$ for $m+2|k|\le d$. Applying Cahen–Mellin integral, we have
\begin{align}(1+D^2)^{-|k|-\frac{m}{2}-z}=\frac{1}{\Gamma(|k|+\frac{m}{2}+z)}\int_0^\infty e^{-s(1+ D^2)}s^{|k|+\frac{m}{2}+z-1}ds \pl. \label{mel}\end{align}
For $a\in W^{\infty,1}(\R_\theta)$ and $\nu\in Cl^d$,
\[\norm{(a\ten\nu)  e^{-s(1+ D^2)}}{L_1(\R_\Theta\ten M_N)}\le e^{-s}\norm{(a\ten\nu)(1+D^2)^{-r}}{1} \norm{ (1+D^2)^{r}e^{-sD^2}}{\infty} \]
By functional calculus,  \[\norm{ (1+D^2)^{r}e^{-sD^2}}{\infty}\le \begin{cases}
                                                                            \frac{r^r}{s^r}, & \mbox{if $s<r$}  \\
                                                                            1, & \mbox{if $s\ge r$}.
                                                                          \end{cases}\]
Then the integral $\displaystyle \int_0^\infty \norm{(a\ten\nu)  e^{-s(1+ D^2)}}{L_1(\R_\Theta\ten M_N)}s^{|k|+\frac{m}{2}+z-1}ds$ converges for $|k|+\frac{m}{2}+Re(z)>r>\frac{d}{2}$. Hence by Fubini Theorem
\[\zeta_{k}(z)=\int_0^\infty Str_\Theta(\Psi_k e^{-s(1+ D^2)})s^{|k|+m/2+z-1}ds\]
Using the trace formula from Proposition \ref{trin},
\begin{align*}Str_{\Theta}(\Psi_{k} e^{-s(1+D^2)})
 =&Str_{\Theta} (\Psi_k(e^{-s(1+|\xi|^2)}\ten e^{-s\omega}))
=  tr_{\theta'}(e^{-s(1+|\xi|^2)})Str_{\theta} (\Psi_ke^{s\omega})
 \\=& \sum_{n}Str_{\theta} \Big(\Psi_k \frac{\omega^n}{n!} \Big) \pi^\frac{d}{2} e^{-s}s^{n-\frac{d}{2}}h(s)
\end{align*}
Here we used the calculation in Proposition \ref{interg} that
\[tr_\theta'(e^{-s|\xi|^2})=s^{-\frac{d}{2}}\det(\frac{i\pi s\theta'}{\sinh i s\theta'})^{\frac{1}{2}}=s^{-\frac{d}{2}}\pi^{\frac{d}{2}}h(s)\pl,\]
where \[h(s)=\det(\frac{is\theta'}{\sinh is\theta'})=\Pi_{j=1}^l \frac{\la_j s}{\sinh \la_j s},\]
where $i\la_1,-i\la_1, \cdots, i\la_l,-i\la_l$ are the nonzero eigenvalues of $\theta'$. Using L'Hospital's Rule, we know $\displaystyle\lim_{s\to 0}s^{-1}(h(s)-1)=0$.
Then we split the residue into two parts
\begin{align*}
Res_{z=0}\zeta_k(z)=&Res_{z=0}Str_\Theta(\Psi_k(1+D^2)^{-m/2-|k|-z})
\\=& Res_{z=0}\frac{1}{\Gamma(m/2+|k|+z)}\int_0^\infty Str_\Theta(\Psi_k e^{-s(1+ D^2)})s^{|k|+m/2+z-1}ds
\\=& \frac{\pi^\frac{d}{2}}{\Gamma(m/2+|k|)}\sum_{n}\frac{1}{n!}Str_\theta(\Psi_k \omega^n)\Big( Res_{z=0}\int_0^\infty e^{-s}s^{n-\frac{d}{2}+|k|+m/2+z-1}ds\\&+  Res_{z=0}\int_0^\infty (h(s)-1)e^{-s}s^{n-\frac{d}{2}+|k|+m/2+z-1}ds\Big)
\end{align*}
Note that for any $j_1,j_2$ and $j_3$, $[c_{j_1}c_{j_2},c_{j_3}]=0$ or of order $1$. Then
\[\textstyle[D^2,da]=[|\xi|^2-\omega,\sum_{j}D_j(a)\ten c_j]=\sum_{j}[|\xi|^2,D_j(a)]\ten c_j+\sum_{j}D_j(a)\ten [\omega,c_{j}]\]
is of Clifford order $1$ and similarly for $da^{(k_0)}$. Thus $\Psi_k=a_0d a_1^{(k_1)}\cdots d a_m^{(k_m)}$ contains Clifford term of at most order $m$ and $\Psi_k \omega^n$ contains Clifford elements of order at most $m+2n$. Hence the super trace $Str_{\theta}(\Psi_k \omega^n )=0$ for $2n+m<d$. It suffices to consider the residue for $2n+m\ge d$. On one hand,
\begin{align}\label{resi}&Res_{z=0}\int_0^\infty(h(s)-1)e^{-s}s^{n-\frac{d}{2}}s^{|k|+m/2+z-1}ds \nonumber\\=&Res_{z=0}\int_0^\infty\frac{h(s)-1}{s}e^{-s}s^{n-\frac{d}{2}+|k|+m/2+z}ds=0 \end{align}
because the integral converges absolutely for $Re(z)>-1\ge -n+\frac{d}{2}-|k|-m/2-1$. For the other residue
\begin{align*}
Res_{z=0}\int_0^\infty e^{-s}s^{n-\frac{d}{2}+|k|+m/2+z-1}ds=Res_{z=0}\Gamma(n-\frac{d}{2}+|k|+m/2+z)
\end{align*}
is zero if $n-\frac{d}{2}+|k|+m/2\ge 0$. Therefore, the only nonzero residue is at $2n+m-d=|k|=0$ and it is a simple pole. Then $\phi_m$ vanishes for odd $m$ and for even $m\ge 2$,
\begin{align*}\phi_m(a_0,\cdots,a_m)&=\al(0) \si_{\frac{m}{2},1} Res_{z=0}\zeta_{0}(z)=\frac{\Gamma(m/2)}{m!}\frac{\pi^\frac{d}{2}}{\Gamma(m/2)}Res_{z=0} \Gamma(z)Str_\theta(\Psi_0 \frac{\omega^{(d-m)/2}}{\frac{d-m}{2}!} )\\
=&\frac{\pi^\frac{d}{2}}{m!}Str_\theta(a_0da_1 \cdots da_m \frac{\omega^{(d-m)/2}}{\frac{d-m}{2}!}) \pl.
\end{align*} 
For $m=0$, we follow the same argument
\begin{align*}\phi_0(a_0)=&Res_{z=0}z^{-1}Str_\Theta(a_0(1+D^2)^{-z})
\\=&Res_{z=0}z^{-1}\frac{1}{\Gamma(z)}\int_0^\infty Str_\Theta(a_0e^{-s(1+D^2)})s^{z-1}ds
\\=&Res_{z=0}\frac{1}{z\Gamma(z)}\int_0^\infty tr_\theta(a_0)tr_{\theta'}(e^{-s|\xi|^2})str(e^{s\omega})e^{-s}s^{z-1}ds
\\=&tr_\theta(a_0)Res_{z=0}\frac{1}{\Gamma(z+1)}\int_0^\infty \sum_{n=0}\frac{str(\omega^n)}{n!}h(s)e^{-s}\pi^{\frac{d}{2}}s^{n-\frac{d}{2}+z-1}ds
\\=&\pi^{\frac{d}{2}}tr_\theta(a_0)\sum_{n=0}\frac{str(\omega^n)}{n!}\Big(Res_{z=0}\int_0^\infty e^{-s}s^{n-\frac{d}{2}+z}ds\\&+ Res_{z=0}\int_0^\infty (h(s)-1)e^{-s}s^{n-\frac{d}{2}+z-1}ds\Big)
\end{align*}
The super trace $str(\omega^n)$ is non-zero if $n< \frac{d}{2}$. For $n\ge \frac{d}{2}$, the second residue \begin{align}\label{resi}&Res_{z=0}\int_0^\infty(h(s)-1)e^{-s}s^{n-\frac{d}{2}+z-1}ds \nonumber\\=&Res_{z=0}\int_0^\infty\frac{h(s)-1}{s}e^{-s}s^{n-\frac{d}{2}+z}ds=0 \end{align}
because the integral converges for integral converges absolutely for $Re(z)>-1\ge n-\frac{d}{2}-1$. The first residue
\begin{align*}
Res_{z=0}\int_0^\infty e^{-s}s^{n-\frac{d}{2}+z-1}ds=Res_{z=0}\Gamma(n-\frac{d}{2}+z)
\end{align*}
is non-zero only if $n-\frac{d}{2}\le 0$. Therefore, $\phi_0(a_0)=\pi^{d/2}Str_\theta(a_0\frac{\omega^{d/2}}{(d/2)!})$.
\end{proof}

 For compact Spin manifolds, the isolated spectral dimension condition always holds and the only nonzero residues when $j=0$ and $k=0$. This simplification recovers the Atiyah-Singer index theorem for Spin Dirac operator (see \cite{CM95}, \cite{higson} and \cite{Ponge}).
The above theorem gives a simplification of the cocycle formula for \[(W^{\infty,1}(\R_\theta),L_2(\R_\Theta)\ten \C^N, \sum\xi_j\ten c_j)\]to the terms only for $|k|=j=0$. As a consequence, the local index formula for $\R_\theta$ simplifies too. We can see the term $\omega$ plays the role of the curvature form.
\begin{cor}\label{formula}For any projection $e\in M_n(W^{\infty,1}(\R_\theta))$ and with $F_{\mu,+}$ defined as in \eqref{phase}, \[Index(e(F_{\mu,+}\ten id_n)e) = \pi^{\frac{d}{2}}Str_\theta\Big((e-1_e)\frac{\omega^n}{n!}+\sum_{m=2, even}^d\frac{1}{m!}e(de)^{m} \frac{\omega^{d-m}}{(d-m)!}\Big)\pl.\]
\end{cor}

\subsection{A concrete example for $d=2$}We shall now calculate a concrete example in dimension $d=2$. In the classical case, a canonical generator for $K_0(C_0(\R^2))$ is the Bott projector
\[ e_B(\bx,\by)=\frac{1}{1+\bx^2+\by^2}\left[\begin{array}{cc} 1 & \bx-i\by\\ \bx+i\by & \bx^2+\by^2 \end{array}\right]\in M_2 (C_0(\R^2)^\sim )\pl, 1_{e_B}= \left[\begin{array}{cc} 0 & 0\\ 0 & 1 \end{array}\right]\in M_2(\C)\pl.\]
Now let $\theta$ be a real number and $\R_\theta$ is the Moyal plane generated by two self-adjoint element $x,y$ with $[x,y]=-i\theta$. We consider an analog of Bott projection for $\R_\theta$. Write $z=x+iy$, $R=(1+z^*z)^{-1}$ and $\displaystyle u=\left[\begin{array}{c}1\\ z\end{array}\right]$. Then
$e:= u\left[\begin{array}{cc}R &0\\ 0& 0\end{array}\right]u^*=\left[\begin{array}{cc}R&Rz^*\\ zR&zRz^*\end{array}\right]$
is a projection because $u^*Ru=1$. The only drawback of $e$ is that it does not belongs to $M_2(W^{\infty,1}(\R_\theta)^\sim)$. Indeed,
by Proposition \ref{interg} and Theorem \ref{order}, we know that $R,zR,zRz^*\notin L_1(\R_\theta)$. Nevertheless, $dede$ and $id\ten tr_2(e-1_e)=R+zRz^*-1$ do belong to $L_1$ so that the cocycle formula in Corollary \ref{formula} are well defined. The next lemma shows that by approximation the cocycle formula remains valid for $e$.
\begin{lemma}\label{ap}There exists a sequence of projection $e_n\in M_2(W^{\infty,1}(\R_\theta)^\sim)$ such that $1_{e_n}=1_e$ and
$\lim_{n\to \infty}\norm{e_n-e}{\infty}=0, \lim_{n\to \infty}\norm{id\ten tr_2(e_n-e)}{1}=0$. As a consequence,
\[\lan [e]-[1_e], (W^{\infty,1}(\R_\theta),L_2(\R_\Theta)\ten \C^N, \sum\xi_j\ten c_j)\ran=\pi Str_\theta((e-1_{e})\omega)+ \pi Str_\theta(ede de)\]
\end{lemma}
\begin{proof} Let $\la_\theta(\phi_n)$ be the approximation identity in Propsition \ref{ai}. Define \[\tilde{e}_n:=(\la_\theta(\phi_n)\ten 1)(e-1_e)+1_e\in M_2(W^{\infty,1}(\R_\theta))\pl. \]Because $e-1_e\in \E_\theta$ and $id\ten tr_2(e-1_e)\in L_1(\R_\theta)$, we have
\begin{align*}&\norm{\tilde{e}_n-e}{\infty}=\norm{(\la_\theta(\phi_n)\ten 1)(e-1_e)-(e-1_e)}{\infty}\to 0\pl,\\ &\norm{id\ten tr_2(\tilde{e}_n-1_e)-id\ten tr_2(e-1_e)}{1}\to 0\pl.\end{align*}
Using holomorphic functional calculus, we can made projections $e_n\in M_2(W^{\infty,1}(\R_\theta))$ from $\tilde{e}_n$ with satisfies the same limits above. It is known that if two projections $e,f$ satisfy that $\norm{e-f}{}< 1$ then $e$ is homotopic to $f$ hence $[e]=[f]$ (see e.g. \cite{rordam00}). Then by the homotopy invariance of index pairing, we know for $n$ large enough
\begin{align*}\lan [e]-[1_e], (\A,H,D)\ran &= \lan [e_n]-[1_{e_n}], (\A,H,D)\ran =\phi_0(e_n-1_{e_n})+\phi_2(e_n-\frac{1}{2},e_n,e_n)\\&=\pi Str_\theta(e_n-1_{e_n}\omega)+\pi Str_\theta((e_n-\frac{1}{2})de_nde_n))\pl.
\end{align*}
Taking the limit $n\to \infty$,
\begin{align*}
\lim_{n\to \infty }Str_\theta((e_n-1_{e_n})\omega)=Str_\theta((e-1_e)\omega)\pl.
\end{align*}
For the second term, we first note that $Str_\theta(de_nde_n)=Str_\theta(-de_nde_n)=0$ because ${de_n\gamma=-\gamma de_n}$. For the same reason, we have the cyclicity that
\begin{align*} &Str_\theta(ede_nde)=Str_\theta(d(ee_n)de)-Str_\theta(d(e)e_nde)=Str_\theta(e_ndede), \\ &Str_\theta(e_ndede_n)=Str_\theta(d(e_ne)de_n)-Str_\theta(d(e_n)ede_n)=Str_\theta(ede_nd(e_n))\pl.\end{align*}
Therefore,
\begin{align*}
&Str_\theta(ede  de)-\tau_\theta\ten Str_\theta(e_nde_n  de_n)\\
=&Str_\theta(ede  de-e_nde  de)+ Str_\theta(e_nde  de-e_nde_n  de)+Str_\theta(e_nde_n  de-e_nde_n  de_n)\\
=&Str_\theta(ede  de-e_nde  de)+ Str_\theta(ede  de_n-e_nde  e_n )+Str_\theta(ede_n  de_ne_nde_n  de_n)\\
=&Str_\theta\big((e-e_n)de  de\big)+ Str_\theta\big((e-e_n)de  de_n\big)+Str_\theta\big((e-e_n)de_n  de_n\big),
\end{align*}
All the three terms above converges to $0$, since $\norm{e-e_n}{\infty} \to 0$ and $dede,dede_n,de_nde_n$ are in $M_2(L_1(\R_\theta))$.
\end{proof}
\begin{theorem}\label{example} For any $\theta,\theta'$, \[\lan [e]-[1_e], (W^{\infty,1}(\R_\theta),L_2(\R_\Theta)\ten \C^N, \sum\xi_j\ten c_j)\ran=4\pi^2(1-\theta\theta')\pl.\]
In particular, $[e]$ is a generator of $K_0(\E_\theta)=\Z$.
\end{theorem}
\begin{proof} The super trace $Str_\theta(ede de)$ is of eight terms
\begin{align*}
&Str_\theta(ede de)=Str_\theta\ten tr_{2}\Big(\left[\begin{array}{cc}R&Rz^*\\ zR&zRz^*\end{array}\right]\left[\begin{array}{cc}dR& d(Rz^*)\\ d(zR)&d(zRz^*)\end{array}\right] \left[\begin{array}{cc}dR&d(Rz^*)\\ d(zR)&d(zRz^*)\end{array}\right]\Big)\nonumber\\
=&Str_\theta\Big(Rd(R)d(R)+Rd(Rz^*)d(zR)+Rz^*d(zR)d(R)+Rz^*d(zRz^*)d(zR) \nonumber
\\&+zRd(R)d(Rz^*)+zRd(Rz^*)d(zRz^*)+zRz^*d(zR)d(Rz^*)+zRz^*d(zRz^*)d(zRz^*)\Big)\pl.
\end{align*}
We will repeatedly use Leibniz rule and cyclicity of trace (in the strong sense \cite[Theorem 17]{brownkosaki}) that
\[d(a_1a_2)=(da_1) a_2+a_1d a_2\pl, Str_\theta(da_1(da_2) a_3)=Str_\theta(a_3da_1 da_2)\]
Denote $\tau=Str_\theta$ in short. For the first and fifth term,
\begin{align*}
\tau\Big(Rd(R)d(R)+zRd(R)d(Rz^*)\Big)&= \tau\Big(d(R)d(R) R+d(R)d(Rz^*)zR\Big)\\
& =\tau\Big(d(R)d(R) R+d(R)d(R)z^*zR+ d(R) R d(z^*)zR\Big)\\
& =\tau\Big(d(R)d(R) R+d(R)d(R)(1-R)+ d(R) R d(z^*)zR\Big)\\
& =\tau\Big(d(R)d(R) + d(R) R d(z^*)zR\Big)
\end{align*}
Similarly we have for the second and sixth term, third and seventh term , fourth and eighth term,
\begin{align*}
\tau\Big(Rd(Rz^*)d(zR)+zRd(Rz^*)d(zRz^*)\Big)&= \tau\Big(d(Rz^*)d(zR)+ zRd(Rz^*)zRdz^*\Big)\\
\tau\Big(Rz^*d(zR)d(R)+zRz^*d(zR)d(Rz^*)\Big)&= \tau\Big(z^*d(zR)dR+zRz^*d(zR)Rdz^*\Big)\\
\tau\Big(Rz^*d(zRz^*)d(zR)+zRz^*d(zRz^*)d(zRz^*)\Big)&= \tau\Big(z^*d(zRz^*)d(zR)+zRz^*d(zRz^*)zR dz^*\Big)\\
\end{align*}
Recoupling these terms,
\begin{align*}
\tau\Big(dRdR+z^*d(zR)dR\Big)&= \tau\Big(R^{-1}dRdR+z^*(dz) RdR\Big)\\
\tau\Big(zR(dR) Rdz^*+zRz^*d(zR)Rdz^*\Big)&= \tau\Big(z(dR) Rdz^*+zRz^*dz R^2dz^*\Big)\\
\tau\Big(d(Rz^*)d(zR)+z^*d(zRz^*)d(zR)\Big)&= \tau\Big(R^{-1}d(Rz^*)d(zR)+z^*(dz) Rz^*d(zR)\Big)\\
\tau\Big(zRd(Rz^*)zRdz^*+zRz^*d(zRz^*)zR dz^*\Big)&= \tau\Big(zd(Rz^*)zRdz^*+zRz^*(dz) Rz^*zRdz^*\Big)\\
\end{align*}
On the right hand side, there are only three terms still contains derivatives of products. We again use Leibniz rule,
\begin{align*}
\tau(R^{-1}d(Rz^*)d(zR))=&\tau(R^{-1}d(R)z^*d(zR) + dz^*d(zR))\\
=&\tau(d(R)z^*d(z)+R^{-1}d(R)(R^{-1}-1)dR) + dz^*d(z)R+dz^*zdR)\\
\tau(z^*(dz) Rz^*d(zR))=&\tau(z^*(dz) (1-R)dR+z^*(dz) Rz^*d(z)R)\\
\tau(zd(Rz^*)zRdz^*)=&\tau(z^*Rdz^*zRdz^*+zdR (1-R)dz^*)
\end{align*}
Gathering all the terms we have,
\begin{align*}&((dR) z^*dz+z^*dz dR)+(dz^*  zdR+zdR  dz^*)+
\\&(zR(dz^*)zRdz^*+ R^{-1}dR  R^{-1}dR+(dz)Rz^* (dz)Rz^*) +Rdz^*  dz+zRz^*(dz)  Rdz^*\pl.
\end{align*}
Here only the last two terms has nonzero trace. This is because for any $a_1,a_2,a_3, b_1,b_2 b_3$
\begin{align*}&Str_\theta\Big(a_1(da_2)a_3b_1(db_2)b_3\Big)=- Str_\theta\Big(b_1(db_2)b_3a_1(da_2)a_3\Big), \\ &Str_\theta\Big(a_1(da_2)a_3a_1(da_2)a_3\Big)=0.
 \end{align*}
 This follows from that fact $a_1(da_2)a_3$ has Clifford term of order $1$ hence $a_1(da_2)a_3\gamma=-\gamma a_1(da_2)a_3$.
It remains to calculate the trace of $Rdz^*dz+zRz^*dzRdz^*$. Note that
$zz^*=z^*z-2\theta=R^{-1}-1-2\theta \pl, dz=-ic_1+c_2 \pl ,dz^*=-ic_1-c_2\pl.$
Then
\begin{align*}
Str_\theta(Rdz^* dz+zRz^*(dz)  Rdz^*)&=4\tau_\theta(R-zRz^*R)
\end{align*}
Finally we use the spectrum of quantum harmonic oscillator the above trace. Assume that $\theta>0$. By Proposition \ref{trace}, there is a trace preserving $*$-isomorphism (up to a factor $2\pi\theta$ $\pi:\R_\theta \to B(L_2(\R))$ such that
\[x \mapsto \sqrt{\theta}D_{\bx}\pl , \pl  y\mapsto \sqrt{\theta}\bx \pl ,\]
Recall that $H=D_{\bx}^2+\bx^2$ is the Hamiltonian of $1$-dimensional quantum harmonic oscillator which has eigenbasis $\ket{n}, n\ge 0$ with $H\ket{n}=(2n+1)\ket{n}$.
For the creation operator $a^*=D_{\bx}+i\bx$ and the annihilation $a=D_{\bx}-i\bx$,
\[ a^*\ket{n}= \sqrt{2n+2}\ket{n+1}\pl, a \ket{n}= \sqrt{2n}\ket{n-1}\]
Now take $z=\sqrt{\theta}a^*, z^*=\sqrt{\theta}a$ and $R^{-1}=1+2\theta+zz^*=\theta (H+1)+1$. We have \begin{align*}
4\tau_\theta(R-zRz^*R)&=2\theta\pi\cdot 4\sum_{k=0}\frac{1}{1+2\theta +2k\theta}-\frac{1}{1+ 2k\theta}\frac{2k\theta}{1+2\theta +2k\theta}
\\&=8\theta\pi \sum_{k=0}\frac{1}{1+ 2k\theta}\frac{1}{1+2\theta +2k\theta}=4\pi\pl.
\end{align*}
For $\phi_0$, we have
\begin{align*}
\phi_0(e-1_e)&=Str_\theta((e-1_e)\omega)
=\tau_\theta(R+zRz^*-1)tr(\gamma\omega)=2\theta'\tau_\theta(R+zRz^*-1)
\end{align*}
Note that $R^{-1}=1+z^*z=1+\theta+x^2+y^2$ and $[R^{-1},z]=[x^2+y^2,x+iy]=2\theta z$. Then,
\begin{align*}
R+zRz^*-1=&R(1+z^*z)-1+[z,Rz^*]=[z,Rz^*]
\\=&[z,R]z^*+R[z,z^*]=R[R^{-1},z]Rz^*-2\theta R
=2\theta (RzRz^*- R)
\end{align*}
We have calculated that $\tau_\theta(R-RzRz^*)=2\pi $. So $Str_\theta((e-1_e)\omega)=-\theta\theta'4\pi$. To conclude, we have the index pairing \begin{align*}\lan [e]-[1_e], (W^{\infty,1}(\R_\theta), L_2(\R_\Theta)\ten M_N,D )\ran= & \pi Str_\theta((e-1_{e})\omega)+ \pi Str_\theta(ede de)
\\=& -4\pi^2\theta\theta'+4\pi^2=4\pi^2(1-\theta\theta')
\end{align*}
Recall for $d=2$ that $\Theta=\left[\begin{array}{cccc}0&\theta & 1& 0\\ -\theta&0 & 0& 1\\ 1&0 & 0& \theta'\\ 0&1 & -\theta'& 0 \end{array}\right]$. When $\det\Theta=(1-\theta\theta')^2\neq 0$, we have $\R_\Theta$ is $*$-isomorphic to $B(L_2(\R^2))$ with the trace differs by a factor $\tau_\Theta=(2\pi)^2|1-\theta\theta'|tr$, which is exactly the normalization constant we obtained. In other words, if we replace $\tau_\Theta$ with the matrix trace $tr$, $\text{Index}_{tr}(e F_{\mu,+}e)=1$ (or $-1$). Since for every $\theta$, we can choose $\theta'$ such that $\theta\theta'\neq 1$, then the index pairing shows that $e\in M_2(\E_\theta^\sim)$ is a representative of generator of the $K_0(\E_\theta)=\Z$.
\end{proof}


\bibliographystyle{alpha}
\bibliography{thesisbib_edited_190821}

\end{document}